\documentclass[reqno,11pt]{amsart}
\usepackage{a4wide,color,eucal,enumerate,mathrsfs}
\usepackage[normalem]{ulem}
\usepackage{amsmath,amssymb,epsfig,amsthm} 
\usepackage[latin1]{inputenc}
\usepackage{psfrag}
\usepackage{hyperref,cleveref,tikz}

\def\rr{{\mathbb R}}

\def\nn{{\mathbb N}}

\def\ch{{\mathcal H}}

\def\az{\alpha}

\def\dist{{\mathop\mathrm{\,dist\,}}}
\def\loc{{\mathop\mathrm{\,loc\,}}}

\def\ez{\epsilon}

\def\boz{{\Omega}}

\def\wz{\widetilde}

\def\ls{\lesssim}
\def\gs{\gtrsim}

\def\bint{{\ifinner\rlap{\bf\kern.35em--}
\int\else\rlap{\bf\kern.45em--}\int\fi}\ignorespaces}

\def\bbint{{\ifinner\rlap{\bf\kern.35em--}
\hspace{0.078cm}\int\else\rlap{\bf\kern.45em--}\int\fi}\ignorespaces}

\def\diam{{\mathop\mathrm{\,diam\,}}}

\newtheorem{thm}{Theorem}[section]
\newtheorem{lem}[thm]{Lemma}

\newtheorem{cor}[thm]{Corollary}
\newtheorem{defn}[thm]{Definition}
\numberwithin{equation}{section}

\theoremstyle{remark}
\newtheorem{rem}[thm]{Remark}

\def\bint{{\ifinner\rlap{\bf\kern.35em--}
\int\else\rlap{\bf\kern.45em--}\int\fi}\ignorespaces}

\begin{document}

\title[Planar $W^{1,\,1}$-extension Domains]
{Planar $\boldsymbol{W^{1,\,1}}$-extension domains}

\author{Pekka Koskela}
\author{Tapio Rajala}
\author{Yi Ru-Ya Zhang}

\address{University of Jyvas\-kyla\\
         Department of Mathematics and Statistics \\
         P.O. Box 35 (MaD) \\
         FI-40014 University of Jyvas\-kyla \\
         Finland}
\email{pekka.j.koskela@jyu.fi} 
\email{tapio.m.rajala@jyu.fi}

\address{Academy of Mathematics and Systems Science, the Chinese Academy of Sciences, Beijing 100190, China}
\email{yzhang@amss.ac.cn}

\thanks{The first author has been  supported by the Academy of Finland via the Centre of Excellence in Analysis and Dynamics Research (project No. 307333) and
 Finnish Centre of Excellence in Randomness and Structures (project No. 364210). The second author has been funded by the Academy of Finland grant No.\ 274372. The third author is funded by National Key R\&D Program of China (Grant No. 2025YFA1018400 \&  No. 2021YFA1003100), NSFC grant No. 12288201 \& No. 12571128 and the Chinese Academy of Sciences.}
\date{\today}


\begin{abstract}
We show that a bounded planar simply connected domain $\Omega$ is a 
$W^{1,\,1}$-extension domain 
if and only if for every pair $x,y$ of points in  $\Omega^c$ there exists a 
curve $\gamma \subset \Omega^c$ connecting $x$ and $y$ with
$$ \int_\gamma\frac{1}{\chi_{\rr^2\setminus \partial\Omega}(z)}\,ds(z) \le 
C|x-y|.$$
Consequently, a planar Jordan domain $\Omega$ is a $W^{1,\,1}$-extension domain 
if and only if it is a $BV$-extension domain,
and if and only if its complementary domain $\wz \Omega$ is a $W^{1,\,\infty}$-extension domain. 
\end{abstract}


\maketitle

\tableofcontents

\section{Introduction}

Given a domain $\Omega \subset\mathbb R^2$, define for $1 \le p\le \infty$ 
the Sobolev space $W^{1,p}(\Omega)$ as
\[
 W^{1,p}(\Omega) = \left\{u \in L^p(\Omega)\mid\nabla u \in L^p(\Omega,\,\mathbb R^2)\right\},
\]
where $\nabla u$ denotes the distributional gradient of $u$. 
One usually employs $W^{1,p}(\Omega)$ with the non-homogeneous norm
$\|u\|_{L^p(\Omega)} + \|\nabla u\|_{L^p(\Omega,\,\mathbb R^2)}$ or with the
homogeneous (semi-)norm $\|\nabla u\|_{L^p(\Omega,\,\mathbb R^2)}$.

We call $E \colon W^{1,p}(\Omega) \to W^{1,p}(\mathbb R^2)$ an extension 
operator for $\Omega$
if there exists a constant $C \ge 1$ so that for every 
$u \in W^{1,p}(\Omega)$ we have 
\[
\|Eu\|_{W^{1,p}(\mathbb R^2)} \le C\|u\|_{W^{1,p}(\Omega)} 
\]
and $Eu|_\Omega = u$. A domain $\Omega \subset \mathbb R^2$
is called a $W^{1,p}$-extension domain if there exists an extension operator
$E \colon W^{1,p}(\Omega) \to W^{1,p}(\mathbb R^2)$. By \cite{HK1991}, a 
bounded domain $\Omega$ admits an extension operator with respect to the 
above non-homogeneous norm
if and only if there is a bounded extension operator $F\colon W^{1,p}(\Omega)\to W^{1,p}_{\loc}(\mathbb R^2)$ with respect to
the homogeneous semi-norm
when $1\le p<\infty$. 

Geometric properties of simply connected planar $W^{1,\,p}$-extension domains 
are by now rather
well understood, see 
\cite{GLV1979}, \cite{J1981}, \cite{S2010} and \cite{KRZ2015}.
In particular, a full geometric characterization for bounded simply 
connected planar $W^{1,\,p}$-extension domains 
is available for $1<p\le \infty$. 
For the range $1 < p < 2$  bounded simply connected planar 
$W^{1,p}$-extension domains were characterized in \cite{KRZ2015}
as  those bounded domains $\Omega \subset \rr^2$
for which any two points $x,y \in \rr^2 \setminus \Omega$ can be connected 
with a curve  $\gamma\subset \rr^2 \setminus \Omega$  satisfying
\begin{equation}\label{eq:pgeaterthan1}
 \int_\gamma \dist(z,\partial \Omega)^{1-p}\,ds(z)  \le C |x-y|^{2-p}.
\end{equation}

Regarding $p=1,$ in \cite[Corollary 1.2]{kosmirsha10} it was shown that the 
complement of a bounded
simply connected $W^{1,1}$-extension domain is quasiconvex.
This was obtained as a corollary to a characterization of 
bounded simply connected $BV$-extension domains.
Recall that a set $E \subset \mathbb R^2$
is called {\it $C$-quasiconvex} if there exists a constant $C\ge1$ such that 
any pair of 
points $z_1,z_2 \in E$ can be connected by a rectifiable curve 
$\gamma \subset E$
whose length $\ell(\gamma)$ satisfies $\ell(\gamma) \le C|z_1-z_2|$.

Quasiconvexity of the complement does not imply $W^{1,\,1}$-extendability in 
general. Indeed, a slit disk
$$\Omega=\mathbb D\setminus\{(x,\,0)\mid x\ge 0 \}$$
has a quasiconvex complement but $\Omega$ fails to be a $W^{1,\,1}$-extension 
domain. 
In this paper we show that in order to obtain a characterization of 
$W^{1,\,1}$-extendability
in the spirit of the curve estimate \eqref{eq:pgeaterthan1}, in addition to 
the quasiconvexity of the complement,
we have to take into account the size of the set of self-intersections of the 
boundary $\partial \Omega$.
Our main result is the following theorem.

\begin{thm}\label{thm:main}
 Let $\Omega \subset \rr^2$ be a bounded simply connected domain. Then $\Omega$ is a $W^{1,\,1}$-extension domain
 if and only if there exists a constant $C < \infty$, quantitatively, such that
 for each pair $x,y \in \Omega^c$ of points there exists a curve 
$\gamma \subset \Omega^c$ connecting $x$ and $y$ with
 \begin{equation}\label{eq:curvecondition}
  \int_\gamma\frac{1}{\chi_{\rr^2\setminus \partial\Omega}(z)}\,ds(z) \le C|x-y|.
 \end{equation} 
\end{thm}

In other words, \eqref{eq:curvecondition} requires that $\ell(\gamma)\le
C|x-y|$ and that $\mathcal H^1(\gamma\cap \partial \Omega)=0,$ where $\mathcal H^1$ denotes the $1$-dimensional Hausdorff measure. 

Both the necessity and sufficiency in Theorem~\ref{thm:main} are new.
Our construction of an extension operator actually gives a linear operator.

\begin{cor}\label{hakotu}
Let $\Omega \subset \rr^2$ be a bounded simply connected domain.
If $\Omega$ is a $W^{1,1}$-extension domain, then $\Omega$ admits a bounded
linear extension operator for $W^{1,1}.$
\end{cor}

Corollary~\ref{hakotu} gives the first linearity result for Sobolev extension
operators for $p=1$ beyond concrete cases like Lipschitz domains. 
Up to now it has not been clear if linearity could be expected; recall that
there is no bounded linear extension operator from the trace space  $L^1(\rr)$
of $W^{1,1}(\rr^2)$ to $W^{1,1}(\rr^2).$ 
The existence of a bounded linear extension operator in the case of $p>1$ 
was established in \cite{hakotu2008} for $W^{1,p}$-extension domains via
a particular method that cannot be applied for $p=1.$ 


In the case of a Jordan domain we do not have boundary self-intersections
and hence our characterization easily reduces to quasiconvexity of the 
complementary domain. From Theorem~\ref{thm:main} and (the proof of) Lemma~\ref{qconvex} we obtain the following corollary. 

\begin{cor}\label{mainthm}
Let $\Omega$ be a planar Jordan domain. Then $\Omega$ is a $W^{1,\,1}$-extension domain if and only if its complementary domain $\wz\Omega:=\mathbb R^2\setminus \overline{\Omega}$ is quasiconvex. 
\end{cor}

Corollary~\ref{mainthm} together with earlier results (see e.g.\ \cite{BM1967}, \cite{Z1999}, \cite{kosmirsha10}) yields the following 
somewhat surprising corollary.

\begin{cor}
Let $\Omega$ be a planar Jordan domain and $\wz \Omega$ be its 
complementary domain. Then the following are equivalent: 
\begin{enumerate}[(1) ]
\item $\Omega$ is a $W^{1,\,1}$-extension domain. 
\item $\Omega$ is a $BV$-extension domain. 
\item $\wz \Omega$ is quasiconvex. 
\item  $\wz \Omega$ is a $W^{1,\,\infty}$-extension domain. 
\end{enumerate}

\end{cor}

This corollary together with 
Lemma~\ref{lma:unboundedbounded} below  also extends the duality result from 
\cite[Corollary 1.3]{KRZ2015} to the case of all $1\le p\le \infty.$

\begin{cor}\label{cor:dual}
 Let $1 \le p,q \le \infty$ be H\"older dual exponents and let 
$\Omega \subset \rr^2$ be a Jordan
 domain. Then $\Omega$ is a $W^{1,p}$-extension domain if and only if 
 $\rr^2 \setminus \bar\Omega$ is a $W^{1,q}$-extension domain.
\end{cor}

One further obtains a monotonicity property via  Theorem~\ref{thm:main} and 
\cite{KRZ2015}: if $\Omega$ is a bounded 
simply connected $W^{1,\,p}$-extension domain with $1<p<2$, 
then it is also a $W^{1,\,q}$-extension domain for all $1\le q\le p.$ 
However, a 
$W^{1,\,1}$-extension domain need not  be a $W^{1,\,p}$-extension domain for 
any $1<p<2$. To see this consider the Jordan domain
$$\Omega=\mathbb D\setminus\{(x,\,y)\mid y\le x^2\}.$$ 
The complement of $\Omega$ is clearly quasiconvex, but for any curve 
$\gamma$ connecting the origin and the point $(t,\,0)$ with $0<t<1$, we have 
$$ \int_\gamma \dist(z,\partial \Omega)^{1-p}\,ds(z)\gtrsim \int_{0}^t  x^{2-2p}\, dx\gtrsim t^{3-2p}. $$
Since there is no constant $C>0$ such that  $t^{3-2p}<C t^{2-p}$ for all 
$0<t<1$ when $1<p<2$,  \eqref{eq:pgeaterthan1} fails, and thus $\Omega$ is not 
an extension domain for $W^{1,\,p}(\Omega)$ when  $1<p<2$.

Resulting from \cite[Corollary 1.5]{KRZ2015} and the argument above, we 
conclude that the set of all $1\le s \le\infty$ for which a bounded simply 
connected planar domain $\Omega$ is a $W^{1,\,s}$-extension domain, can only 
be one of:
$$\text{$\emptyset$, $\{1\}$, $[1,\,q)$ with $q<2$, $[1,\,\infty]$,  $(q,\,\infty]$ with $q> 2$ or $\{\infty\}$.}$$

Let us close this introduction by briefly discussing the ideas of the proofs of Theorem~\ref{thm:main} and Corollary~\ref{mainthm}. First of all, we cannot use the approach in \cite{KRZ2015} which gave a characterization for $1<p<2$. This is because there both   the sufficiency and necessity of \eqref{eq:pgeaterthan1} were first proven for approximating Jordan domains and the general case was concluded via a limiting argument. Our condition \eqref{eq:curvecondition} is not stable under (uniform) convergence of the complement and $W^{1,\,1}$ is not weakly compact.  For sufficiency, an approximation procedure would only give $BV$-extension. Nevertheless, we first prove sufficiency in the Jordan case via a hyperbolic triangulation of the complementary domain  instead of a usual Whitney decomposition. For the general case we construct a generalized hyperbolic triangulation for the complement of $\Omega$ and apply the Jordan result for suitable subdomains of $\Omega$ to show that the extension obtained via this triangulation is a Sobolev extension. Namely, when $\Omega$ is simply connected, it becomes necessary to manually construct the extension, drawing upon the approximation techniques employed in the Jordan case in Section~\ref{sec:sufficiency_Jordan}. This process inevitably gives rise to a range of non-trivial technical challenges, and has been done in Sections~\ref{sec:phg}, \ref{sec:sufficiency}, and \ref{sec:necessity}. The necessity of \eqref{eq:curvecondition} is proven directly in the simply connected case and the necessity in Corollary~\ref{mainthm} is deduced from the general case.  The proof of Corollary~\ref{mainthm} is given at the end of Section~\ref{sec:necessity}.

\section{Preliminaries}\label{sec:preliminaries}
In this section we recall some definitions and properties. 
The notation in our paper is quite standard. When we make estimates, we sometimes write the constants as 
positive real numbers $C(\cdot)$ with
the parenthesis including all the parameters on which the constant depends. The constant $C(\cdot)$ may
vary between appearances, even within a chain of inequalities. 
Especially, $C$ refers to  denote an absolute constant. 
By $a\lesssim b$ we mean $a\le C b$, and $a \sim b$ means that $b/C \le a \le Cb$ for some constant $C \ge 2$. 
The Euclidean distance between two sets $A,\,B \subset \mathbb R^2$ is denoted by $\dist(A,\,B)$. 
By $\mathbb D$ we always mean the open unit disk in $\mathbb R^2$ and by $\partial \mathbb D$ its boundary.
We denote by $\ell(\gamma)$ the length of a curve $\gamma$. For a set $A\subset \mathbb R^2$, we denote by $A^o$ its interior, $\partial A$ its boundary, $A^c=\mathbb R^2\setminus A$ its complement, and $\overline A$ its closure respectively with respect to the Euclidean topology, unless another specific explanation is given. 
For an injective curve $\gamma\colon [0,\,1] \to \mathbb R^2$ we use $\gamma^o$ to denote $\gamma((0,\,1))$.
The line segment connecting $x$ and $y$ is referred to by $[x,\,y]$, and for an injective curve $\gamma$ with $x,\,y\in \gamma$, the subarc of $\gamma$ joining $x$ and $y$ is denoted by $\gamma[x,\,y]$. 
We denote the $1$-dimensional Hausdorff measure by $\mathcal H^1$. 
For a disk $B\subset \mathbb R^2$ and a constant $c>0$, we denote by $cB$ the disk with the same center as $B$ but a radius that is $c$ times the radius of $B$.

  Let us recapitulate the main definitions and properties that are used throughout the paper frequently.

\begin{itemize}
    \item A Jordan domain $\Omega \subset \mathbb{R}^2$ is defined as the image of the unit disk $\mathbb{D}$ under a global homeomorphism of $\mathbb{R}^2$ onto itself. In particular, $\Omega$ is bounded and its boundary $\partial \Omega$ is a compact, simple closed curve. 

    \item Given two open subsets $A,B \subset \mathbb R^2$, by a conformal mapping $\varphi \colon A \to B$  we mean a holomorphic homeomorphism between the sets $A$ and $B$.

    \item Let $\Omega \subset \mathbb R^2$ be a simply connected domain. Then the Riemann mapping theorem guarantees the existence of a conformal mapping $\varphi \colon \mathbb{D} \to \Omega$.

    \item For any Jordan domain $\Omega \subset \mathbb{R}^2$ and its complementary domain $\wz{\Omega}:=\mathbb R^2\setminus \overline{\Omega}$, in addition to the conformal mapping $\varphi \colon \mathbb{D} \to \Omega$, the Riemann mapping theorem also guarantees the existence of conformal mapping $\wz{\varphi} \colon \mathbb{R}^2 \setminus \overline{\mathbb{D}} \to \wz{\Omega}$. Furthermore, by the Carath\'eodory-Osgood theorem \cite[Section IX.4]{P1991}, both $\varphi$ and $\wz{\varphi}$ extend continuously to the boundary and induce homeomorphisms there.

    \item Let $\Omega \subset \mathbb{R}^2$ be a bounded, simply connected domain whose complement is quasiconvex. Then $\Omega$ is a John domain (see Definition~\ref{John}, and every conformal mapping $\varphi \colon \mathbb{D} \to \Omega$ admits a continuous extension to the boundary. However, this extension need not be a homeomorphism of the closed unit disk onto $\overline{\Omega}$. See Lemma~\ref{property of John domain} below. 

    \item Given a sequence of uniformly bounded curves whose lengths are also uniformly bounded, the Arzel\`a-Ascoli lemma \cite[Theorem 2.5.14]{BBI2001} implies, possibly after passing to a subsequence, the existence of a limit curve whose length has the same uniform bound.
\end{itemize}

\medskip
 
We have the following swapping lemma. 
\begin{lem}\label{lma:unboundedbounded}
 Let $\Omega \subset \mathbb R^2$ be a bounded domain. 
Fix $x \in \Omega$ and define
 an unbounded domain $\hat \Omega = i_x(\Omega)$ using the inversion 
 \[
  i_x \colon \mathbb R^2\setminus\{x\} \to \mathbb R^2\setminus\{x\} \colon y \mapsto x + \frac{y-x}{|y-x|^2}.
 \]
 Then
 \begin{enumerate}
  \item The domain $\Omega$ is a $W^{1,1}$-extension domain 
        if and only if $\hat \Omega$ is such a domain.
  \item The domain $\Omega$ satisfies \eqref{eq:curvecondition} with $\gamma\subset \Omega^c$ if and only if 
$\hat \Omega$ has this property 
 \end{enumerate}
\end{lem}
\begin{proof}
The first part of the lemma follows from the part (1) of \cite[Lemma 5.1]{KRZ2015}. Thus we only need to verify (2).

Let  $R=2\diam(\Omega)$ and $2r= \dist(x,\,\partial \Omega) $. Then 
$$\partial \Omega \subset A(x,\,r,\,R) := \overline B(x,\,R) \setminus {B(x,\,r)}.$$ 

Suppose that $\Omega$ satisfies \eqref{eq:curvecondition} with $\gamma\subset \Omega^c.$ We prove that $\hat \Omega$ also satisfies \eqref{eq:curvecondition}. Given $z_1,\,z_2\in \Omega^c$,  let $x_1=i_x^{-1}(z_1)$ and $x_2=i_x^{-1}(z_2).$ Then $x_1,\,x_2\in \overline {\Omega}=\hat \Omega^c$. Even though the point $x$ is not obtained in this manner, it is easy to see that it suffices to establish the claim for all such $x_1,x_2.$

Notice that the inversion $i_x$ is a biLipschitz map when restricted to $A(x,\,r,\,R)$, with the biLipschitz constant only depending on $r$ and $R$. 
If a curve $\gamma\subset \Omega^c$ connecting $z_1,z_2$  from   \eqref{eq:curvecondition} lies in $A(x,\,r,\,R)$, then the biLipschitz property of $i_x$ directly gives the desired inequality for the curve $i_x \circ\gamma$  connecting $x_1,x_2$ up to a multiplicative constant depending only on $p$, $r$ and $R$.

Next if $z_1,\,z_2\in A(x,\,r,\,R)$ but the corresponding curve is not contained in  $A(x,\,r,\,R)$, since $R= 2\diam( \Omega)$, we may replace a part of the curve by a subarc of the circle $S^1(x,\,R)$ so that we stay inside $A(x,\,r,\,R)$. The new curve, which is denoted by $\gamma$, satisfies inequality 
\eqref{eq:curvecondition} with a constant that only depends on the original
constant.  Then the desired inequality for the curve $i_x \circ\gamma$ follows by the argument in the previous case. 

The case where $z_1,\,z_2\in B(x,\,R)^c$ is trivial, since then $x_1,\, x_2$ are contained in a disk in $\Omega.$  The case $z_1\in B(x,\,R)^c$ while $z_2\in A(x,\,r,\,R)$ follows easily from the combination of  the previous cases, and by symmetry we finish the proof of this direction. 

The case where \eqref{eq:curvecondition}  holds for $\hat \Omega$ is analogous to the above one.
\end{proof}
 
\subsection{Hyperbolic geodesics}

Recall that the hyperbolic geodesics in $\mathbb D$ and in
$\mathbb R^2\setminus \overline{\mathbb D}$ are defined as the arcs of 
(generalized) circles that intersect the 
unit circle orthogonally. 
  To be more specific, the hyperbolic distance in the unit disk
is defined to be
$$\dist_h(z_1,\,z_2)=\inf_{\gamma}\int_{\gamma} \frac {2}{1-|z|^2}\, d s(z),$$
where the infimum is taken over all rectifiable curves $\gamma$ joining $z_1$ to $z_2$ in $\mathbb D$. The density of the hyperbolic metric is comparable to $\dist(z,\,\partial \mathbb D)^{-1}$, and the infimum is attained uniquely by a hyperbolic geodesic: an arc of a circle orthogonal to $\partial \mathbb D$. If the geodesic contains the origin, it is a Euclidean segment. Hyperbolic distance is invariant under conformal self-maps of $\mathbb D$, i.e., M\"obius transformations preserving $\mathbb D$. 

For a simply connected domain $\Omega$, we take a conformal map $\varphi\colon \mathbb D\to \Omega$ and for $x, y \in \Omega$ define 
$$\dist_{h}(x,\,y)=\dist_h(\varphi^{-1}(x),\,\varphi^{-1}(y)) .$$
This definition is independent of $\varphi$. 
Equivalently, 
$$\dist_h(x,\,y)=\inf_{\gamma}\int_{\gamma} \frac {2|g'(z)|}{1-|g(z)|^2}\, d s(z),$$
where $g=\varphi^{-1}$ and the infimum is taken over all rectifiable curves that join $x$ to $y$ in $\Omega.$ 
By the Koebe distortion  theorem \cite[Theorem 2.10.6]{AIM2009}, the density is comparable to $\dist(z,\,\partial \Omega)^{-1}$ with an absolute constant. 

For the complement of the closed unit disk, one can define the hyperbolic distance via the inversion $z\mapsto \frac{1}{z}$.
(Notice that for points in the complement of the closed unit disk for which the origin lies on the line-segment between them, the hyperbolic geodesic does not exist in $\mathbb R^2 \setminus \overline{\mathbb D}$ as it would go via the point at infinity.)
The density is still controlled from above by an absolute constant multiple of $\frac 1 {|z|-1}=\dist(z,\,\partial \mathbb D)^{-1}$ (and also from below when $z\in B(0,\,10)$). 
More generally, for a Jordan domain $\Omega$ and its complementary domain $\wz \Omega=\mathbb R^2\setminus \overline{\Omega}$, one can define the hyperbolic distance via conformal equivalence with $\mathbb D$ or  $\mathbb R^2\setminus \overline{\mathbb D}$. Uniqueness holds because any two such conformal maps differ by a rotation. Furthermore, as $\Omega$ is Jordan, any hyperbolic geodesic in $\wz \Omega$ can be extended to the boundary points in $\partial \Omega$ via the extended conformal map between $\mathbb R^2\setminus {\mathbb D}  \to \overline{\wz \Omega}$, which is a homeomorphism  according to Carath\'eodory-Osgood  theorem  \cite[Section IX.4]{P1991}.
We refer to \cite[Chapter 2]{AIM2009} for more details.

The following theorem was established by  Gehring and Hayman \cite{GH1962}; see also  e.g.\ {\cite[Theorem 4.20, Page 88]{P1992}}. 
\begin{thm}\label{GM bdd}
Let $\Omega$ be a bounded simply connected domain, and $\varphi\colon \mathbb D \to \Omega$ be a conformal map. Then there exists an absolute constant $C>0$ so that,  given a pair of points
$x,\,y\in \mathbb D$, 
denoting the corresponding hyperbolic geodesic in 
$\mathbb D$ by $\Gamma_{x,\,y}$, 
and by $\gamma_{x,\,y}$ any 
arc connecting $x$ and $y$ in $ \mathbb D$, we have
$$\ell(\varphi(\Gamma_{x,\,y}))\le C \ell(\varphi(\gamma_{x,\,y})). $$
\end{thm}

One also has the following complementary version of the Gehring-Hayman theorem for Jordan domains. We record a proof in the Appendix. 
 
\begin{thm}\label{GM unbounded}
Let $\wz \Omega$ be the complementary domain of a Jordan domain and $x,\,y\in \overline{\wz \Omega}$. 
If $\wz \varphi\colon  \overline{\wz \Omega}  \to \mathbb R^2\setminus  {\mathbb  D}$ is a homeomorphism  (given by Carath\'eodory-Osgood  theorem) which is conformal in $\wz\Omega$ and $\Gamma$ is the hyperbolic
geodesic in $\wz \Omega $ connecting $x$ and $y$ satisfying 
$ \wz \varphi (\Gamma) \subset B(0,\,100),$ and $\gamma\subset \overline{ \wz \Omega}$ is any curve joining $x$ and $y$, then we have
$$\ell( \Gamma  )\le C \ell( \gamma ) $$
for some absolute constant $C$. 
\end{thm}

According to Theorem~\ref{GM bdd}, hyperbolic geodesics in a (bounded) simply
connected domain are essentially the shortest possible curves joining given
endpoints. For the complementary domain case, one needs to rule out the case of
geodesics with large diameters; see Theorem~\ref{GM unbounded}.

\subsection{Whitney-type sets} 
Let us recall the following properties of Whitney-type sets. 

\begin{defn}\label{whitney-type set}
 A connected set $A \subset \boz\subset \mathbb R^2$ is called a $\lambda$-Whitney-type set in $\Omega$ 
with  $\lambda\ge 1$ if the following holds. 
 \begin{enumerate}
\item There exists a disk with radius $\frac {1}{\lambda}\diam(A)$ contained in $A$;

\item $ \frac {1}{\lambda} \diam(A)\le \dist(A,\,\partial\Omega)\le {\lambda } \diam(A)$. 
 \end{enumerate}
\end{defn}

We need the following lemma that  follows
from \cite[Corollary 4.18]{P1992}. Also see \cite{BHK2001} and 
\cite[Theorem 0.1]{BB2003}, where it is called the \lq\lq{ball separation property}\rq\rq{}.  

\begin{lem}\label{ball separation}
Let $\varphi \colon \mathbb R^2\setminus \overline {\mathbb D}\to G$ be a conformal 
map, and $x,\,y\in G$ be a pair of points and $\Gamma$ a hyperbolic geodesic 
joining them. If $B$ is a disk with 
$B\cap\Gamma\neq \emptyset$ and of $\lambda$-Whitney type, then any curve 
$\gamma$ connecting $x,\,y$ in 
$G$ has non-empty intersection with $c B,$ where $c=c(\lambda)$.
\end{lem}

Conformal mappings preserve Whitney-type sets in the following sense. 

%

\begin{lem}\label{whitney preserving}
Suppose $\varphi\colon \Omega \to \Omega'$ is conformal  (and injective, particularly), 
where $\Omega,\Omega'\subset \mathbb R^2$  
are domains conformally equivalent to $\mathbb D$ or $ \mathbb R^2\setminus \overline{\mathbb D}$,
and $A\subset \Omega$ is a $\lambda_1$-Whitney-type set. Then 
$\varphi(A)\subset \Omega'$ is a $\lambda_2$-Whitney-type set with $ \lambda_2=\lambda_2(\lambda_1)$. Moreover, the restriction of the conformal map $\varphi|_{A}$ is a $L(\lambda_1)$-biLipschitz map up to a dilation factor. 
\end{lem}

\begin{proof}
The first claim follows, e.g.\,, from  \cite[Lemma 2.12]{KRZ2015}. Thus we only need to verify the second claim. 

To begin, for any two points $w_1,\,w_2\in A$, by \cite[Lemma 2.11]{KRZ2015} the hyperbolic distance between $w_1$ and $w_2$ is smaller than $C(\lambda_1)$. 
Denote by $\gamma[w_1,\,w_2]$ the  hyperbolic geodesic joining $w_1$ and $w_2$. Since the hyperbolic distance is realized by hyperbolic geodesics, we conclude that for any $z\in \gamma[w_1,\,w_2]$ the hyperbolic distance between $z$ and $w_1$ is smaller than $C(\lambda_1)$. Fix $w_0\in A$.  Then \cite[Theorem 2.10.8]{AIM2009}  implies that
$$|\varphi'(z)|\sim |\varphi'(w_1)|\sim |\varphi'(w_0)|,$$
where the constants depend only on $\lambda_1$. 
Thus we have
$$|\varphi(w_1)-\varphi(w_2)|\le \int_{\gamma[w_1,\,w_2]} |\varphi'(z)| \,dz \le C(\lambda_1) |\varphi'(w_0)| \int_{\gamma[w_1,\,w_2]}  1 \, dz \le C(\lambda_1) |\varphi'(w_0)| |w_1-w_2|, $$
where the last inequality is a consequence of the observation that 
 $\gamma[w_1,\,w_2]$ can be encompassed by no more than    $C(\lambda_1)$-many Whitney squares, and within each of these squares, the length of the hyperbolic geodesic is comparable to the length of the straight-line segment connecting its endpoints.
Similarly, since  $\varphi^{-1}$ is also conformal and $\varphi(A)$ is of $\lambda_2$-Whitney-type with $\lambda_2=\lambda_2(\lambda_1)$, 
by the above estimate we get
$$|w_1-w_2|\le \int_{\gamma[\varphi(w_1),\,\varphi(w_2)]} |\varphi'(z)|^{-1} \,dz \le  C(\lambda_1) |(\varphi^{-1})'(\varphi(w_0))| |\varphi(w_1)-\varphi(w_2)|. $$
Since $\varphi$ is conformal, we have $|\varphi'(w_0)|\neq 0$ and 
$$|(\varphi^{-1})'(\varphi(w_0))|=|\varphi'(w_0)|^{-1}.$$
To conclude, we have that
$$C(\lambda_1)^{-1} |\varphi'(w_0)| |w_1-w_2|\le   |\varphi(w_1)-\varphi(w_2)|\le C(\lambda_1) |\varphi'(w_0)| |w_1-w_2|,$$
as desired. 
\end{proof}

Sometimes we omit the constant $\lambda$ when we are dealing with a
$\lambda$-Whitney-type set whose constant is clear from the context.

\subsection{John domains}

Let us recall some known results about John domains. The relation to the curve condition \eqref{eq:curvecondition} is given by part (1) of Lemma~\ref{property of John domain}. 

\begin{defn}[John domain]\label{John}
An open subset $\Omega\subset \mathbb R^2$ is called a John domain provided it satisfies the following condition:
There exist a distinguished point $x_0 \in \Omega$ and a constant $J>0$ such that, for every $x\in\Omega$, 
there is a curve $\gamma\colon [0,\,\ell(\gamma)] \to \Omega$ parameterized by the arc length, such that $\gamma(0)= x$, $\gamma(\ell(\gamma))= x_0$
and
\[
\dist(\gamma(t),\, \mathbb R^2\setminus\Omega)\ge Jt. 
\]
The curve $\gamma$ is called a John curve, and $J$ is called a John constant. 

 In particular, via an approximation argument, the Arzel\`a-Ascoli lemma  yields the existence of John curve also for boundary points. 
\end{defn}

%

We collect a number of results related to John domains in the following lemma. 
\begin{lem}\label{property of John domain}
The following statements hold: 

\noindent{\rm (1)}
Let $\Omega\subset \mathbb R^2$ be a bounded simply connected domain. 
If $\Omega^c$ is $C$-quasiconvex, then $\Omega$ is a $J$-John domain
with $J=J(C).$ 

\noindent{\rm (2)}
Given a simply connected (bounded) $J$-John domain 
$\Omega\subset \mathbb R^2$ together with a conformal map 
$\varphi \colon \mathbb D \to \Omega$ with $\varphi(0)=x_0,$  where $x_0$ is 
the distinguish point of $\Omega,$ we can 
extend $\varphi$ 
continuously up to the boundary. 

\noindent{\rm (3)}
 When $\Omega\subset \mathbb R^2$ is a simply connected $J$-John domain with a distinguished point $x_0$,  the hyperbolic geodesic connecting $x\in \Omega$ and $x_0$ in $\Omega$
is a $J'$-John curve, where $J'=J'(J).$

\noindent{\rm (4)}
The Lebesgue measure of the boundary of a simply connected (bounded) $J$-John domain 
$\Omega\subset \mathbb R^2$ is zero. 
\end{lem}
\begin{proof}
(1) is from \cite[Theorem 4.5]{NV1991}; observe that the $C$-quasiconvexity property is  stronger  than the corresponding $C$-bounded turning property employed in \cite[Theorem 4.5]{NV1991}, i.e.\ every pair of points $x,\,y\in \mathbb R^2\setminus \Omega$ can be joined by a curve $\gamma$ with $\diam(\gamma)\le C |x-y|$. (2) follows from \cite[Theorem 2.18]{NV1991} and \cite[Theorem 4.7,\,Page 441]{P1991}, and (3) comes from \cite[Theorem 4.1]{GHM1989}. 

To show (4), let $x$ be an arbitrary point in $\partial\Omega$. Define $B_r=B(x,\,r)$ for some $0<r<\frac 1 2 |x-x_0|$, where $x_0$ is the John center of $\Omega$. Choose a point $z\in B$ such that $|z-x|<r/ 4 $, and denote by $\gamma$ the hyperbolic geodesic from $x_0$  to $z$. By (4), $\gamma$ is a $C(J)$-John curve, and hence by choosing a point $y\in \gamma$ such that $|y-z|=r/2$, the definition of John domain, shows that there exists a constant $c=c(J)$ such that the disk $B(y,\,cr)$ is compactly contained in $B_r\cap \Omega$ for every $0<r<\frac 1 2 |x-x_0|$. Hence $x$ cannot be a point of density of $\partial \Omega$. By applying the Lebesgue differentiation theorem to the characteristic function of $\partial \Omega$ we then conclude (5) by the arbitrariness of $x$. 
\end{proof}

In fact we have a stronger version of (2) of Lemma~\ref{property of John domain}. 
Define the {\it{inner distance of $\Omega$}} between 
$x,\,y\in\Omega$ by
$$\dist_{\Omega}(x,\,y)=\inf_{\gamma\subset \Omega} \ell(\gamma),$$
where the infimum runs over all curves joining $x$ and $y$ in $\Omega.$ If $x\in \partial\Omega$ and $y\in \overline{\Omega}$ and $\gamma\colon [0,\,1]\to \overline{\Omega}$ is continuous with $\gamma(0)=x$,  $\gamma(1)=y$ and $\gamma((0,\,1))\subset\Omega$, we say that $\gamma$ joins $x$ and $y$ in $\Omega$. Furthermore, if $\ell(\gamma)$ is finite, then we say that $x$ and $y$ are rectifiably joinable in $\Omega$ and define $\dist_{\Omega}(x,\,y)$ via $\inf_{\gamma} \ell(\gamma)$ over all such $\gamma$.  
The inner diameter $\diam_{\Omega}(E)$ of a set $E\subset \overline{\Omega}$ is then defined in
the usual way. 

\begin{lem}\label{inner continuous}
Let $\Omega\subset \mathbb R^2$  be a simply connected (bounded) $J$-John domain 
 and let $\varphi \colon \mathbb D \to \Omega$ be conformal. Then   $\varphi\colon (\mathbb D,\,|\cdot|)\to (\Omega,\,\dist_{\Omega})$  is uniformly continuous, where $|\cdot|$ denotes the Euclidean metric. 
\end{lem}
\begin{proof}
By \cite{H2012} together with  (2) of Lemma~\ref{property of John domain}, 
$\varphi\colon (\mathbb D,\,|\cdot|)\to (\Omega,\,d_{\Omega})$ is uniformly continuous, where for each pair of points $x,\,y\in \Omega$,
$$d_{\Omega}(x,\,y):=\inf_{\gamma\subset \Omega} \diam(\gamma),$$
and the infimum runs over all curves joining $x$ and $y$ in $\Omega.$

By \cite[Theorem 5.14]{GNV94}, for a simply connected (bounded) $J$-John domain $\Omega$, we have 
$$d_{\Omega}(x,\,y)\sim \dist_{\Omega}(x,\,y)$$
for  any pair of two points $x,\,y\in \Omega$, where the constant   depends only on $J$.
Thus  the claim follows form the previous uniform continuity result. 
\end{proof}

Recall that the extended complex plane consists of the complex plane together with a point at infinity. As a complex manifold it can be described by two charts via the stereographic projections with $\frac 1 z$ as the transition map.
Combining Lemma~\ref{property of John domain} above with results in \cite{DH}, we have the following conclusion. 


\begin{lem}\label{complement of John}
Let $\Omega\subset \mathbb R^2$ be a bounded simply connected domain whose complement is quasiconvex. Then $\Omega^c$ is locally connected and each bounded component of its interior is a Jordan domain. (In fact, every component of the interior of $\Omega^c$ is Jordan on the Riemann sphere. )
\end{lem}
\begin{proof}
By Lemma~\ref{property of John domain} (2), we know that a Riemann mapping $\varphi\colon \mathbb D \to \Omega$ can be extended continuously to the boundary. Hence  \cite[Page 14, Corollary]{DH} shows that $\Omega^c$ is locally connected. 

To show the second part of the lemma, we may assume that $0\in \Omega$. Then by letting $g(z)=\frac z {|z|^2}$ we have that $g(\Omega^c)$ is compact, locally connected and its complement is connected (on the Riemann sphere). Therefore by applying \cite[Page 17, Proposition 2.3]{DH} to $g(\Omega^c)$ we conclude the second part of the lemma; note that $g$ is a homeomorphism on the Riemann sphere. 
\end{proof}

%

\subsection{Quasiconvex sets}

In this subsection we list some important properties of quasiconvex sets in the plane.

We begin with the following corollary of Theorem~\ref{GM unbounded}. 
 
\begin{cor}\label{GH thm}
Let $\wz \Omega$ be the complementary domain of some Jordan domain.  Assume that $\overline{\wz \Omega}$ is $C_1$-quasiconvex and $x,\,y\in \mathbb R^2\setminus  {\mathbb  D}$. 
If $\wz \phi\colon \mathbb R^2\setminus  \mathbb D   \to\overline{\wz \Omega}$ is a homeomorphism (given by Carath\'eodory-Osgood  theorem) which is conformal in $\mathbb R^2\setminus \overline{\mathbb D}$ and 
$\Gamma$ is a hyperbolic
geodesic in $\mathbb R^2\setminus \overline{\mathbb  D}$ connecting $x$ and $y$ with
$ \Gamma \subset B(0,\,100),$ then
$$\ell(\wz \phi (\Gamma))\le C(C_1) |\wz \phi (x)- \wz \phi (y)|.$$
\end{cor} 
\begin{proof}
According to the definition of quasiconvexity, there exists a curve $\gamma\subset \overline{\wz \Omega}$ joining $x$ and $y$ with $\ell(\gamma)\le C_1 |x-y|$. Then by Theorem~\ref{GM unbounded}, we have 
$$\ell(\wz \phi (\Gamma) )\le C \ell(\gamma)\le C(C_1) |\wz \phi (x)- \wz \phi (y)|$$
with the constant depending only on $C_1$.
\end{proof}

The proof of  the following technical lemma is given in Section~\ref{appendix quasiconvex} of the Appendix.

\begin{lem} \label{ulkodist}
Given $C_1>0$  and an unbounded $C_1$-quasiconvex domain $\wz \Omega$ whose boundary is Jordan, there is a constant
$\delta_1=\delta_1(C_1)>0$ so that the following holds: If
$\wz \varphi\colon \wz \Omega \to \mathbb R^2\setminus \overline{\mathbb  D}$ is  a conformal mapping extending to a homeomorphism on the boundary, and
$z_1,z_2\in \overline{\wz \Omega}$ satisfy 
$$|z_1-z_2|<\delta_1 \diam(\partial \wz  \Omega)\quad  \text{ and }\quad \dist(z_i,\,\partial \wz  \Omega)\le \delta_1 \diam(\partial \wz  \Omega) \quad \text{ for }\ i=1,\,2,$$ then 
$|\wz \varphi (z_1)-\wz \varphi (z_2)| <\frac 1 4$ and $\dist(\varphi(z_i),\,\partial \mathbb D)<\frac 1 4$ for $i=1,\,2$. 
\end{lem}



\begin{lem}\label{qconvex}
Let $\wz \Omega$ be the complementary domain of a Jordan domain. Then if $\wz \Omega$ is $c_1$-quasiconvex, 
its closure is $c_2$-quasiconvex with $c_2$ depending only on $c_1$. Conversely, if the closure of $\wz \Omega$ is $c_2$-quasiconvex, then $\wz \Omega$ is $c_1$-quasiconvex with $c_1$ depending only on $c_2$.
\end{lem}
\begin{proof}
We first show the sufficiency of $c_2$-quasiconvexity of the closure. 
Let $x,\,y\in \wz \Omega.$
To begin, let us consider the special case where 
$$|x-y|\le \delta \diam( \partial {\wz \Omega}), \ \dist(x,\,\partial {\wz \Omega})\le \delta \diam( \partial {\wz \Omega}) \  \text{ and } \  \dist(y,\,\partial {\wz \Omega})\le \delta \diam( \partial {\wz \Omega})$$
  with the constant $\delta$ in Lemma~\ref{ulkodist}. Then by applying Corollary~\ref{GH thm} and Lemma~\ref{ulkodist} to the hyperbolic geodesic connecting $x$ and $y$ we get the desired control on the length of the geodesic.

Let us now prove the general case. Let $x, y \in \widetilde\Omega$ be two distinct points. Since $\overline{\widetilde\Omega}$ is $c_2$-quasiconvex,
there exists a curve $\gamma \subset \overline{\widetilde\Omega}$ connecting $x$ to $y$ with $\ell(\gamma) \le c_2 |x-y|$.
Let $\{x_i\}_{i=0}^k \subset \gamma$ be such that $x_0 = x$, $x_k = y$,
\begin{equation}\label{eq:1}
 0 < |x_i-x_{i-1}| \le \frac{1}{6}\delta\diam(\partial \widetilde\Omega) \qquad \text{for all }i \in \{1,\dots, k\}. 
\end{equation}
The existence of these points follows from the finiteness of the length of $\gamma$.
By the definition of the length of a curve, we further have
\begin{equation}\label{eq:2}
 \sum_{i=1}^k|x_i-x_{i-1}| \le \ell(\gamma) \le c_2 |x-y|.
\end{equation}

For each $i \in \{1,\dots, k-1\}$ take a point
\[
x_i' \in \widetilde \Omega \cap B(x_i,\min\{|x_i-x_{i-1}|, |x_i-x_{i+1}|\}) 
\]
and define $x_0' = x$ and $x_k'=y$.
By the choice of $x_i'$  we have for every $i \in \{1,\dots, k\}$ the estimate
\begin{equation}\label{eq:3}
 |x_i'-x_{i-1}'| \le |x_i'-x_i| + |x_i - x_{i-1}| + |x_{i-1}-x_{i-1}'| \le 3|x_i-x_{i-1}|.
\end{equation}
Combining this with \eqref{eq:1} and \eqref{eq:2}  gives
\begin{equation}\label{eq:4}
  |x_i'-x_{i-1}'| \le \frac{1}{2}\delta\diam(\partial \widetilde\Omega) \ \ \text{ and }  \ \sum_{i=1}^k|x_i'-x_{i-1}'|\le \sum_{i=1}^k 3|x_i-x_{i-1}| \le 3c_2 |x-y|.
\end{equation}

Now for each $i \in \{1,\dots, k\}$ we define a curve $\gamma_i\subset \wz \Omega$ connecting $x_{i-1}'$ and $x_i'$. There are two cases. If
\begin{equation}\label{eq:5}
|x_i'-x_{i-1}'| \le \delta\diam(\partial\widetilde\Omega), \quad \dist(x_i',\partial\widetilde\Omega)\le \delta\diam(\partial\widetilde\Omega)
\quad \text{and} \quad \dist({x_{i-1}',\partial\widetilde\Omega})\le \delta\diam(\partial\widetilde\Omega),
\end{equation}
then by the special case we find $\gamma_i \subset \widetilde\Omega$ connecting $x_{i-1}'$ and $x_i'$ with $\ell(\gamma_i) \le C(c_2) |x_i'-x_{i-1}'|$.
If \eqref{eq:5} fails, then by \eqref{eq:4}
\[
 \quad \dist(x_i',\partial\widetilde\Omega)> \delta\diam(\partial\widetilde\Omega)
\quad \text{or} \quad \dist({x_{i-1}',\partial\widetilde\Omega})> \delta\diam(\partial\widetilde\Omega).
\]
In this case we define $\gamma_i := [x_i',x_{i-1}'] \subset \widetilde\Omega$. Then $\ell(\gamma_i) = |x_i'-x_{i-1}'|$. By concatenating all the $\gamma_i$, via \eqref{eq:4} we finally get a curve $\gamma'$ in $\widetilde\Omega$ connecting $x$ to $y$ with
\[
 \ell(\gamma') \le \sum_{i=1}^k\ell(\gamma_i) \ls \sum_{i=1}^k |x_i'-x_{i-1}'| \ls |x-y|,
\]
with the constant depending only on $c_2$ as required.

For the necessity, we choose sequences of points with $x_n\to x$ and $y_n\to y$ with  $x_n,\,y_n\in \wz \Omega$. Then for every large $n\in \mathbb N$, there exists a curve $\gamma_n$ of length  uniformly bounded  from above by a multiple of $|x-y|$ which joins $x_n$ and $y_n$, according to the quasiconvexity of $\wz \Omega$. Therefore up to reparametrizing $\gamma_n$'s, Arzel\`a-Ascoli lemma tells us that (a subsequence of) $\{\gamma_n\}$ converges to a curve $\gamma\subset \overline{\wz \Omega}$ joining $x$ and $y$, whose length is bounded from above by a multiple of $|x-y|$. 
\end{proof}

\begin{lem} \label{inherited} Let $\Omega$ be a bounded simply connected 
domain so that $\Omega^c$ is $C_1$-quasiconvex and let $\varphi\colon \mathbb D\to \Omega$
be conformal. Set $B_n=B(0,1-2^{-n})$ and $\Omega_n=\varphi(B_n)$ for $n\ge 1.$
Then $\Omega_n^c$ is $C(C_1)$-quasiconvex for each $n.$
\end{lem}

\begin{proof}
Notice that, by  part (1) of Lemma~\ref{property of John domain} $\Omega$ is a $J$-John domain with $J=J(C_1)$.  
Fix a pair of points $x,\,y\in \Omega_n^c$. If $x,\,y\in \Omega^c$, then we are done by the quasiconvexity of $\Omega^c$. 
Hence, by symmetry we may assume that $x\in \Omega\setminus \Omega_n$.

Let $x_0\in \partial \Omega$ be a point such that the line segment $[\varphi^{-1}(x),\,\varphi^{-1}(x_0)]$ joining $\varphi^{-1}(x)$ and $\varphi^{-1}(x_0)$ lies in the line segment $[0,\,\varphi^{-1}(x_0)]$. 
Notice that $\gamma_1=\varphi([\varphi^{-1}(x),\,\varphi^{-1}(x_0)])$ is a hyperbolic geodesic. 

Suppose first that  $y\in \Omega^c$. Then we can join $x_0$ and $y$ by a curve $\gamma_2$ in $\Omega^c$ of length at most $C_1|x_0-y|$ because $\Omega^c$ is quasiconvex. 
Note that part (3) of Lemma~\ref{property of John domain}    implies 
\begin{equation}\label{equat 1}
|x-x_0|\le \ell(\gamma_1)\le C(J) \dist(x,\,\partial \Omega).
\end{equation}
 Then  by the fact that $\dist(x,\,\partial \Omega)\le |x-y|$ as $ x\in \Omega$ and $y\in \Omega^c$, via \eqref{equat 1}
 we have
$$\ell(\gamma_1)+\ell(\gamma_2)\le C(J) |x-y| + C(C_1)|x_0-y|\le C(C_1)| (|x-y| + |x-x_0|)\le C(C_1)|x-y|, $$
where in the second inequality we applied the triangle inequality to get
$$|x_0-y|\le |x-y|+|x-x_0|.$$
Hence $\gamma_1\cup\gamma_2$ is a desired curve.

Next consider the case where $y\in \Omega\setminus \Omega_n$. Here we have two subcases depending whether \eqref{equ 9} below holds. 

Firstly suppose $\varphi^{-1}(y)\in B(\varphi^{-1}(x),\,(1-|\varphi^{-1}(x)|)/2)=:B$. Then the existence of the desired curve follows directly from Lemma~\ref{whitney preserving}; note that we can join 
$\varphi^{-1}(x)$ and $\varphi^{-1}(y)$ inside $B\setminus \overline{B_n}$ by a curve of length uniformly bounded by a multiple of $|\varphi^{-1}(x)- \varphi^{-1}(y)|$ for some absolute constant, and then by applying a change of variable we obtain the desired estimate from Lemma~\ref{whitney preserving}. 
A similar conclusion also holds if the roles of $x,\,y$ are reversed.   Note that each $B(z,\,(1-|z|)/2)$ is a Whitney-type set. Then by Lemma~\ref{whitney preserving} we know that  $\varphi(B(z,\,(1-|z|)/2))$ is also of Whitney-type for any $z\in \mathbb D$, and hence there exists an absolute constant $C$ such that if 
$$C|x-y|\le \max \{\dist(x,\,\partial \Omega),\, \dist(y,\,\partial \Omega) \},$$
then these two points fall into  one of the two cases above. 

Hence we only need to consider the case where $x,\,y\in \Omega\setminus \Omega_n$ and
\begin{equation}\label{equ 9}
C|x-y|\ge \max \{\dist(x,\,\partial \Omega),\, \dist(y,\,\partial \Omega) \}. 
\end{equation}
Recall the definition of $x_0$ above, and similarly define $y_0\in \partial \Omega$. We can connect $x_0$ to $y_0$ by a curve $\gamma_2$ in $\Omega^c$ with length controlled from above by $C_1|x_0-y_0|$. Also join $y_0$ to $y$ by the hyperbolic geodesic $\gamma_3$, analogous to $\gamma_1$ that joins $x_0$ to $x$. The desired curve is obtained by concatenating $\gamma_1$ with $\gamma_2$ and $\gamma_3$, respectively. Indeed by \eqref{equ 9} and \eqref{equat 1}
\begin{align*}
\ell(\gamma_1)+\ell(\gamma_2)+\ell(\gamma_3)\le C(J) \dist(x,\,\partial \Omega)+ C(J) \dist(y,\,\partial \Omega)+C(C_1)|x_0-y_0|\\
\le C(C_1) (|x-y|+ |x_0-y_0|)\le C(C_1)| |x-y|, 
\end{align*}
where again in the last inequality we apply the  triangle inequality with \eqref{equ 9} to have
$$|x_0-y_0|\le |x-y|+\dist(x,\,x_0)+ \dist(y,\,x_0)\le |x-y|+C(J) \dist(x,\,\partial \Omega)+ C(J) \dist(y,\,\partial \Omega).$$
\end{proof}

\section{Sufficiency for Jordan domains}\label{sec:sufficiency_Jordan}

Let $\Omega$ be a Jordan domain satisfying \eqref{eq:curvecondition}. By
Lemma~\ref{qconvex}, the complementary domain
$\wz \Omega$ of $\Omega$ is $C_1$-quasiconvex, where $C_1$ depends only on $C$ in \eqref{eq:curvecondition}. We prove the following theorem in this section. 

\begin{thm}\label{Jordan suff}
Let $\Omega\subset \mathbb R^2$ be a Jordan domain whose complementary domain is $C_1$-quasi-convex. Then there exists an extension operator $E\colon W^{1,\,1}(\Omega)\to W^{1,\,1}(\mathbb R^2)$ with the norm depending only on $C_1$. 
\end{thm}

We begin by decomposing a part of $\wz \Omega$ into 
hyperbolic triangles (defined below), and then define auxiliary 
functions on them. 
After this we construct an extension operator based on these functions. 

Recall first that $\Omega$ is a $J$-John domain with $J=J(C_1)$ by part (1) of
Lemma~\ref{property of John domain}. 
Let $\varphi\colon \overline{\mathbb D} \to \overline{\Omega} $ be a 
\emph{homeomorphism} which is conformal 
in $\mathbb D$ and satisfies $\varphi(0)=x_0$, where $x_0$ is the 
distinguished point of $\Omega$. 
Such a homeomorphism  exists by the Riemann mapping theorem and  
Carath\'eodory-Osgood 
theorem \cite[Section IX.4]{P1991} since $\Omega$ is \emph{Jordan}.

Moreover we denote by 
$\wz \varphi\colon \overline{\wz \Omega} \to \mathbb R^2\setminus \mathbb D$ a 
homeomorphism which 
is conformal in $\wz \Omega$. The map $\wz \varphi$ is obtained also from an 
extension of a Riemann mapping via  
Carath\'eodory-Osgood 
theorem.   Even though 
$$\varphi(\partial\mathbb D)=\wz\varphi(\partial\mathbb D)=\partial\Omega,$$ 
it could happen that $\varphi(z)$ might not coincide with $\wz\varphi(z)$ for each $z\in\partial \mathbb D$.

\begin{figure} 
 \centering
 \includegraphics[width=0.75\textwidth]{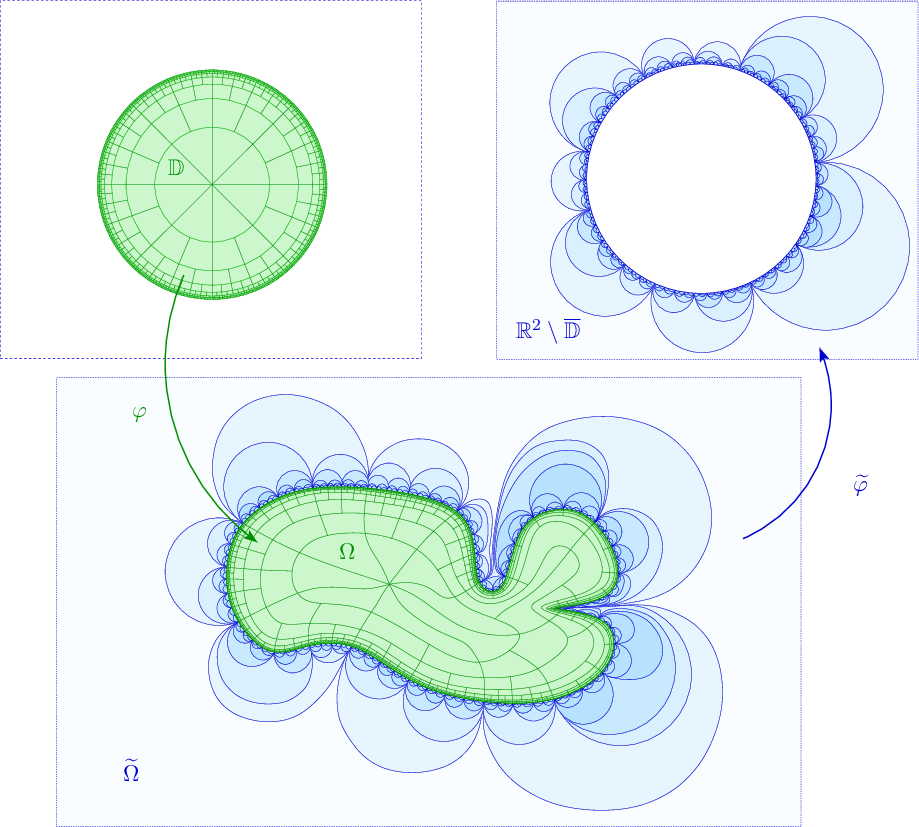}
 \caption{The hyperbolic triangulation of the complementary domain $\wz \Omega$ of $\Omega$ via conformal maps $\varphi \colon \mathbb D  \to  \Omega $ and $\wz\varphi\colon \wz \Omega \to \mathbb R^2\setminus \overline{\mathbb D}$.}
 \label{fig:hyperbolic triangles}
\end{figure}

\subsection{Decomposition of the domain and the complementary domain}\label{sec:jordandeco}

Fix $k_0\in \mathbb N$ to be determined later. 
Let $A_1=B(0,\,\frac 1 2)$ and 
$$A_{k}=B(0,\,1-2^{-k})\setminus B(0,\,1-2^{-k+1})$$ for $k\ge 2$. 
For each $k\ge 1$, the radial rays passing through the points $x^{(j)}_k=e^{{j}{{2^{-k-k_0}}}\pi i} \in \partial \mathbb D$ with $1\le j\le 2^{k+k_0+1}$ 
divide $A_k$ into $2^{k+k_0+1}$ subsets. For $1 \le j\le 2^{k+k_0+1}$, we denote  by $D^{(j)}_k$ each of the subsets. 
Here the upper indices are labeled anticlockwise from the positive real axis. Define $Q^{(j)}_k=\varphi(D^{(j)}_k),$ and $z^{(j)}_k=\varphi(x^{(j)}_k) \in \partial \Omega$ with  $x^{(0)}_k=x^{(2^{k+k_0+1})}_k$. Since $D^{(j)}_k$ is a  Whitney-type set with the constant depending only on $k_0$, we obtain that $Q^{(j)}_k$ is also of $\lambda(k_0)$-Whitney-type by Lemma~\ref{whitney preserving}. 

Moreover, we show that
\begin{equation}\label{equ 10}
 \dist_{\Omega}(z^{(j-1)}_k,\,z^{(j)}_k)\le C(J,\,k_0) \diam(Q^{(j)}_k) 
\end{equation}
Towards this, let $x_1$ be a point in $\Gamma_1\cap Q^{(j)}_k$, where $\Gamma_1$ is the hyperbolic geodesic joining $z^{(j-1)}_k$ and the origin. Recall that $\Gamma_1$ is a $J$-John curve by part (3) of Lemma~\ref{property of John domain}. Thus we have
$$\dist(x_1,\,\partial \Omega)\ge C(J) \ell(\Gamma_1[x_1,\,z^{(j-1)}_k])\ge C(J) \dist_{\Omega}(x_1,\,z^{(j)}_k)$$
where $\Gamma_1[x_1,\,z^{(j-1)}_k]$ denotes the subarc of $\Gamma_1$ joining $x_1$ and $z^{(j-1)}_k$. Likewise we define $\Gamma_2$ as the hyperbolic geodesic joining $z^{(j)}_k$ and the origin, $x_2$ to be a point in $\Gamma_2\cap Q^{(j)}_k$, and have the estimate
$$\dist(x_2,\,\partial \Omega)\ge C(J) \ell(\Gamma_2[x_2,\,z^{(j)}_k])\ge C(J) \dist_{\Omega}(x_2,\,z^{(j)}_k). $$
Since $Q^{(j)}_k$ is a connected set of $\lambda(k_0)$-Whitney-type, we have
$$\dist_{\Omega}(x_1,\,x_2)\le C(k_0) \dist(x_1,\,\partial \Omega)\sim_{k_0} \dist(x_2,\,\partial \Omega). $$
Therefore by the triangle inequality we conclude that
\begin{align*}
\dist_{\Omega}(z^{(j-1)}_k,\,z^{(j)}_k) & \le \dist_{\Omega}(z^{(j-1)}_k,\,x_1)+ \dist_{\Omega}(z^{(j)}_k,\,x_2)+ \dist_{\Omega}(x_1,\,x_2)\\
&\le \dist(x_1,\,\partial \Omega)\sim \diam(Q^{(j)}_k)
\end{align*}
with the constant depending only on $k_0$ and $J$, which concludes \eqref{equ 10}.

Denote by $\gamma_k^{(j)}$ the hyperbolic geodesic of $\wz \Omega$ connecting $z^{(j)}_k$ and $z^{(j-1)}_k$. 
Since $\Omega$ is $J$-John,   by (2) of Lemma~\ref{property of John domain}, it follows that $\varphi$ is uniformly continuous; also see \cite[Page 100,\,Corollary 5.3]{P1992}.
Therefore we can choose $k_0=k_0(C_1)\in \mathbb N$ large enough such that
$$|z^{(j-1)}_k-z^{(j)}_k|\le \delta \diam(\Omega)$$
for any $k,\,j\in \mathbb N$, where $\delta$ is the constant in Lemma~\ref{ulkodist}. This allows us to apply Corollary~\ref{GH thm}   to $x^{(j)}_k=\varphi^{-1}(z^{(j)}_k)$ and $x^{(j+1)}_k=\varphi^{-1}(z^{(j+1)}_k)$   with \eqref{equ 10} and conclude that
\begin{equation}\label{distance estimate}
\ell(\gamma_k^{(j)})\le C(C_1) |z_{k}^{(j-1)} - z_{k}^{(j)}|\le C(C_1) \diam(Q_{k}^{(j)}). 
\end{equation}

A set is called a {\it hyperbolic triangle} if  it is enclosed by three 
hyperbolic geodesics, meeting at the three vertices.  
For $k\in \mathbb N$ and $1 \le j\le 2^{k_0+k+1}$, we denote by $ R_{k}^{(j)}$  
the relatively closed set (with respect to the topology of $\wz \Omega$) 
enclosed by $\gamma_k^{(j)}$, $\gamma_{k+1}^{(2j-1)}$ and $\gamma_{k+1}^{(2j)}$. 
It is a closed hyperbolic triangle 
(with respect to the topology of $\wz \Omega$); see Figure~\ref{fig:hyperbolic triangles}.

 Note that the union of hyperbolic triangles
$$\bigcup_{k\ge 1} \bigcup_{1\le j\le 2^{k_0+k+1}} R_{k}^{(j)} $$
fails to cover any Euclidean neighborhood of $\partial \Omega$ in $\wz\Omega$. To address this, we define additional hyperbolic triangles $R_{0}^{(j)}$
on top of the existing $R_{1}^{(j)}$'s.

Notice that for $1\le j\le 2^{k_0+2}$ the image of each $\gamma^{(j)}_1$ 
under $\wz \varphi$ is a circular arc; 
recall that 
$\wz \varphi\colon \overline{\wz \Omega} 
\to \mathbb R^2\setminus \mathbb D$ is a homeomorphism which is conformal 
in $\wz \Omega$. Denote the preimage under $\wz \varphi$ of the midpoint of 
each arc $\wz \varphi(\gamma_1^{(j)})$ by $z^{(j)}_0$, and the 
corresponding hyperbolic geodesic connecting $z_0^{(j)}$ and $z_0^{(j+1)}$ 
by $\gamma^{(j)}_0$; here $z_0^{(2^{k_0+2}+1)}=z_0^{(1)}$. 
Then with the help of $\gamma^{(j)}_1$ we obtain    $2^{k_0+2}$ closed hyperbolic 
triangles $R_{0}^{(j)}$ such that every $R_{0}^{(j)}$ is enclosed by 
$\gamma_0^{(j)},\,\gamma_{1}^{(j)}$ and $\gamma_{1}^{(j+1)}$; 
here $\gamma_{1}^{(2^{k_0+2}+1)}=\gamma_{1}^{(1)}$. 
Moreover $\wz\varphi(\gamma_0^{(j)})\subset B(0,\,10)$ for any 
$1\le j\le 2^{k_0+2}$ since we defined the constant $\delta$ as in Lemma~\ref{ulkodist}. 
Therefore it is legitimate to apply Corollary~\ref{GH thm}   to $x^{(j)}_0=\varphi^{-1}(z^{(j)}_0)$ and $x^{(j+1)}_0=\varphi^{-1}(z^{(j+1)}_0)$ together with \eqref{distance estimate} to conclude
\begin{align}\label{distance estimate2}
\ell(\gamma^{(j)}_0)&\le C_1 |z^{(j+1)}_0-z^{(j)}_0|\le C_1 \diam(\gamma_0^{(j+1)}\cup \gamma_0^{(j)})\le C_1 \ell(\gamma_0^{(j+1)})+ C_1 \ell( \gamma_0^{(j)})\nonumber\\
&\le C(C_1) |z^{(j-1)}_1-z^{(j)}_1| + |z^{(j+1)}_1-z^{(j)}_1|\le C(C_1)   \diam(Q^{(j)}_1),
\end{align}
where in the last inequality we used the fact that 
\begin{equation}\label{comparable diam}
    \diam(Q^{(j+1)}_1) \sim_{C_1} \diam(Q^{(j)}_1).
\end{equation} 
Indeed, by triangle inequality, there exist points $y_1,\,y_2\in Q^{(j)}_1$ with
$$|y_1-y_2|\le \diam(Q^{(j)}_1)\le 2|y_1-y_2|,$$
which satisfies
\begin{equation}\label{change variable y}
\int_{\varphi^{-1}([y_1,\,y_2])} |\varphi'|\,ds(z)=|y_1-y_2|.
\end{equation}
Similarly, we can find points $w_1,\,w_2\in Q^{(j+1)}_1$ with the same property. Since $\overline{D^{(j)}_1\cup D^{(j+1)}_1}$ is of $\lambda$-Whitney-type for $\lambda=\lambda(C_1)>0$, then $\varphi|_{\overline{D^{(j)}_1\cup D^{(j+1)}_1}}$ is $L$-biLipschitz modulo a dilation factor by Lemma~\ref{whitney preserving} with $L=L(C_1)$. As $D^{(j)}_1$ coincides with $D^{(j+1)}_1$ up to a rotation, we conclude from \eqref{change variable y} (and the one for $w_1,\,w_2$) that
$$|y_1-y_2|\sim_{C_1}|w_1-w_2|,$$
 hence \eqref{comparable diam} follows. 


\subsection{Definition of the extension}

We define an auxiliary function as follows. Our construction is linear in the
parameters $a_1,a_2,a_3.$

\begin{lem}\label{axillary function}
Let $R$ be a closed hyperbolic triangle with vertices $z_1,\,z_2,\,z_3$ and enclosed by three Jordan arcs $\gamma_1,\,\gamma_2$, $\gamma_3$.  
Then for $a_1,\,a_2,\,a_3\in \mathbb R$, 
there exists a function $\phi$, locally Lipschitz in $R \setminus\{z_1,\,z_2,\,z_3\}$, with the following properties: 
\begin{equation}\label{bound of function}
\min\{a_1,\,a_2,\,a_3\} \le  \phi  \le \max\{a_1,\,a_2,\,a_3\},
\end{equation}
and for $i=1,\,2,\,3$, we have $\phi(x)=a_i$ when $x\in \gamma_i.$
Moreover, 
\begin{equation}\label{norm of function}
||\nabla \phi||_{L^{1}(R)}\lesssim \sum_{i,\,j\in\{1,\,2,\,3\}} |a_i-a_j|\min\{\ell(\gamma_{i}),\,\ell(\gamma_{j})\} 
\end{equation}
with the constant only depending on $C_1.$
\end{lem}

\begin{proof}
Without loss of generality we may assume that $\gamma_3$ has the longest length among these three curves. 
We first consider the case where $\max\{a_1,\,a_2\}\le a_3$. Define 
\begin{align*}
\phi(x)=\min &\left\{a_1+\inf_{ \theta_1}\int_{ \theta_1}  \frac {20(a_3-a_1)}{\dist(z,\,\partial R\setminus \gamma_1)} \, ds(z),\, a_2+\inf_{ \theta_2}\int_{ \theta_2}  \frac {20(a_3-a_2)}{\dist(z,\,\partial R \setminus \gamma_2)} \, ds(z),\, a_3  \right\},
\end{align*}
where the infima are taken over all curves $\theta_i\subset R$ connecting $x$ to $ \gamma_{i}$. Observe 
that $\phi$ is bounded by its definition and locally Lipschitz in the (Euclidean) interior of $R$ via the triangle inequality; indeed for any $x\in R$,
$$|\phi(x) - \phi(y)|\le C |x-y|\max_{i=1,\,2} (a_3-a_i)\dist(z,\,\partial R\setminus \gamma_i)^{-1}$$ 
for $y\in B(x,\,1/3 \dist(x,\,\partial R))$ with an absolute constant $C$. Moreover $\phi$ takes the correct boundary value by its definition; note that both of the integrals in the definition of $\phi$ have logarithmic growth towards the boundary.

Define
$$A=\bigcup_{i=1}^{2} \bigcup_{y\in \gamma_i} B\left(y,\, \frac 1 {10} \dist(y,\,\partial R\setminus \gamma_i)\right).$$
Note that for $x\in  R \setminus A $ we have 
\begin{equation}\label{equ 700}
\phi(x)=a_3.
\end{equation}
By the $5r$-covering theorem, we know that there exists a 
countable collection of pairwise disjoint disks 
$$\left\{B\left(y_{ij},\, \frac 1 {10} \dist(y_{ij},\,\partial R\setminus \gamma_i)\right)\right\}_{y_{ij}\in \gamma_i}$$ such that
$$A\subset \bigcup_{i=1}^{2} \bigcup_{y_{ij}\in \gamma_i}  B\left(y_{ij},\, \frac 1 {2} \dist(y_{ij},\,\partial R \setminus \gamma_i)\right)=:\bigcup_{i=1}^{2} \bigcup_{y_{ij}\in \gamma_i} B_{ij}. $$

Let us  estimate the gradient inside $B_{ij}$. 
Define
$$A_i=\bigcup_{y_{ij}\in \gamma_i} B_{ij}$$
for $i=1,\,2.$
Notice that $A_1\cap A_2=\emptyset$ by the triangle inequality; $B_{ij}$ are open disks.
Then for every $x\in A_i,\, i=1,\,2$, we have
$$\phi(x)= \min \left\{a_i+\inf_{ \theta_i}\int_{ \theta_i}  \frac {20(a_3-a_i)}{\dist(z,\,\partial R\setminus \gamma_i)} \, dz,\,a_3\right\},$$
and hence for $x\in A_i,\, i=1,\,2$, 
\begin{equation}\label{local Lip}
|\nabla \phi(x)|\ls  \frac { a_3-a_i }{\dist(z,\,\partial R\setminus \gamma_i)}.
\end{equation}
By the definition of $A_i$, \eqref{local Lip} together with \eqref{equ 700} gives us the local Lipschitz property of $\phi$ in $R \setminus\{z_1,\,z_2,\,z_3\}$.

Notice that, for any $z\in B_{ij},\, i=1,\,2$ by the triangle inequality we have
$$2\dist(z,\,\partial \Omega\setminus \gamma_i)  \ge  \dist(y_{ij},\,  \partial \Omega\setminus \gamma_i). $$
Then by \eqref{equ 700} and \eqref{local Lip}, via the triangle inequality we further have
\begin{align*}
\|\nabla \phi \|_{L^{1}(R)}&\lesssim \sum_{i=1,\,2} \sum_{j}\int_{B_{ij}} \frac {|a_3-a_i|}{\dist(z,\,\partial R \setminus \gamma_i)} \, ds(z)\\
&\lesssim \sum_{i=1,\,2} \sum_{j}\int_{B_{ij}} \frac {|a_3-a_i|}{\dist(y_{ij},\,\partial R\setminus \gamma_i)} \, ds(z)\\
&\lesssim \sum_{i=1,\,2} \sum_{j} {|a_3-a_i|}{\dist(y_{ij},\,\partial R\setminus \gamma_i)} \,
\lesssim |a_3-a_1|\ell(\gamma_{1})+|a_3-a_2|\ell(\gamma_2),
\end{align*}
where the last inequality follows from the fact that $\frac{1}{5}B_{ij}$ are centered on $\gamma_i$ and pairwise disjoint. 

For the other cases, we apply suitable modifications to $\phi$; the idea is to increase or decrease the value of $\phi(x)$ towards a longest curve according to the values on the other two curves. Namely if  $a_3\le \min\{a_1,\,a_2\}$, we   define
$$\phi(x)=\max \left\{a_2-\inf_{ \theta_2}\int_{ \theta_2}  \frac {20(a_2-a_3)}{\dist(z,\,\partial R\setminus \gamma_1)} \, ds(z),\,  a_1- \inf_{ \theta_1}\int_{ \theta_1}  \frac {20(a_1-a_3)}{\dist(z,\,\partial R \setminus \gamma_2)} \, ds(z),\, a_3  \right\}.$$
For the remaining case, by symmetry we only consider the case where $a_1\le a_3\le a_2$, and  define
$$\phi(x)=\max \left\{a_2-\inf_{ \theta_2}\int_{ \theta_2}  \frac {20(a_2-a_3)}{\dist(z,\,\partial R\setminus \gamma_1)} \, ds(z),\, \min \left\{ a_1+\inf_{ \theta_1}\int_{ \theta_1}  \frac {20(a_3-a_1)}{\dist(z,\,\partial R \setminus \gamma_2)} \, ds(z),\, a_3\right\}  \right\}.$$
An argument  similar to the above one gives the estimate for the $L^1$-norm of the gradient of $\phi$. 
\end{proof}

\begin{figure} 
 \centering
 \includegraphics[width=0.9 \textwidth]{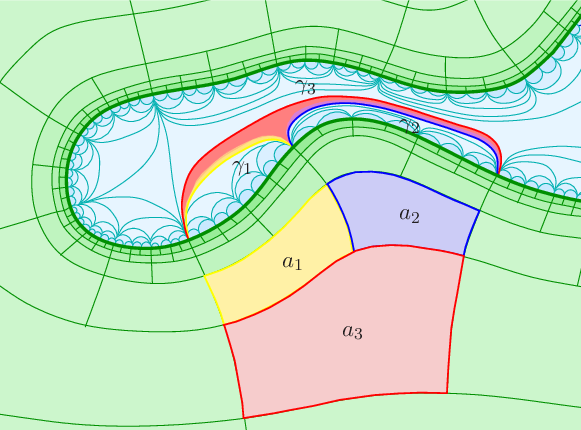}
 \caption{A hyperbolic triangle is shown with the three Whitney-type sets corresponding to its three edges, respectively. The colors of the edges and the Whitney-type sets show how  the boundary values are taken by the function in the hyperbolic triangle. }
 \label{fig:value association}
\end{figure}

With the help of Lemma~\ref{axillary function}, we are ready to construct our 
extension operator.


Fix a function 
\begin{equation}\label{assumption u}
u\in W^{1,\,1}(\Omega)\cap C^{\infty}(\mathbb R^2).
\end{equation}
Recall that $\Omega$ is $J$-John by Lemma~\ref{property of John domain}. For $k\ge 1$ and $1\le j\le 2^{k_0+k+1}$ we associate to $\gamma_k^{(j)}$  a constant
$$a_k^{(j)}=\bint_{Q_k^{(j)}} u(x) \,dx, $$
where $Q_k^{(j)}\in W$ is the Whitney-type set defined in Section~\ref{sec:jordandeco} above.

Let 
$$R=\bigcup R_k^{(j)},$$
where the indices run over all the $k,\,j\in \mathbb N$ such that $ R_k^{(j)} $ is defined. 
First, by Lemma~\ref{axillary function} we associate to each $R_{k}^{(j)}$ 
a function $\phi_{k}^{(j)}\in W^{1,\,1}(R_{k}^{(j)})$ for every $k\ge 1$ and $1\le j\le 2^{k_0+k+1}$ so that 
\begin{align*}
\phi_{k}^{(j)}(x) = \left\{ \begin{array}{ll}
a_k^{(j)} \ & \ x\in \gamma_k^{(j)}\cap R_{k}^{(j)}\\
a_{k+1}^{(2j-1)} \ & \ x\in \gamma_{k+1}^{(2j-1)}\cap R_{k}^{(j)}\\
a_{k+1}^{(2j)} \ & \ x\in \gamma_{k+1}^{(2j)}\cap R_{k}^{(j)}
\end{array} \right. ;
\end{align*}
see Figure~\ref{fig:value association}. 
When $k=0$, again by Lemma~\ref{axillary function}, for each $1\le j\le    2^{k_0+2}$ we get a function $\phi_{0}^{(j)}\in W^{1,\,1}(R_{0}^{(j)})$ satisfying

\begin{align*}
\phi_{0}^{(j)}(x) = \left\{ \begin{array}{ll}
a \ & \ x\in \gamma_0^{(j)}\setminus(\gamma_{1}^{(j)}\cup \gamma_{1}^{(j+1)})\\
a_{1}^{(j)} \ & \ x\in \gamma_{1}^{(j)}\cap R_{0}^{(j)}\\
a_{1}^{(j+1)} \ & \ x\in \gamma_{1}^{(j+1)}\cap R_{0}^{(j)}
\end{array} \right. ,
\end{align*}
where $ a_1^{( 2^{k_0+2}+1)}=a_1^{(1)}$ and $$a=\bint_{\Omega} u\, dx=\frac 1 {|\Omega|}\int_{\Omega} u\, dx.$$

Finally set 
$$Eu(x)=
 \begin{cases}
 u(x), & \text{if }x \in \overline{\Omega}\\
 \phi_{k}^{(j)}(x), & \text{if $x\in   R_{k}^{(j)}\cap \wz \Omega$ }\\
a, & \text{otherwise} 
 \end{cases}.
$$

Now let us estimate the homogeneous Sobolev norm of $Eu$.
It is clear from the construction that $Eu\in W^{1,\,1}_{\loc}(\wz\Omega)$ since $Eu$ is locally Lipschitz in $\wz \Omega$ by Lemma~\ref{axillary function}. 
Since by Lemma~\ref{whitney preserving} the conformal map $\varphi$ is uniformly $C(\lambda)$-bi-Lipschitz on $\lambda$-Whitney-type sets up to a dilation factor, we have that $Q_k^{(j)}$'s are $C(C_1)$-John domains, and so are, for example, also $Q_k^{(j)}\cup Q_k^{(j+1)}$ and $Q_k^{(j)}\cup Q_{k+1}^{(2j+1)}$. Recall the following $(1,\,1)$-Poincar\'e inequality for John domains: For any $J$-John domain $G \subset \mathbb R^2$ and  $v\in W^{1,\,1}(G)$ we have
$$\bint_{G} |v-v_G|\, dx \le C(J) \diam(G) \bint_{G}  |\nabla v|\, dx, $$
where $v_G$ is the integral average of $v$ in $G$; 
see e.g.\ \cite{BK1995} and the references therein. Also note that by the definition of a John domain, the facts that every $Q_1^{(j)}$ is of $\lambda(C_1)$-Whitney-type  and that the John center of $\Omega$ is contained in every $Q_1^{(j)}$, we have 
\begin{equation}\label{equ 30}
\diam (Q_1^{(j)})\sim \diam(\Omega) \qquad \text{ and then } \qquad
|Q_1^{(j)}|\sim |\Omega|\sim \diam(\Omega)^2
\end{equation}
for each $1\le j\le   2^{k_0+2}$, where the constants depend only on $C_1$. 

Now we can estimate the homogeneous Sobolev norm of $Eu$ over $\wz \Omega$.
By recalling \eqref{distance estimate} and \eqref{distance estimate2}, we apply Lemma~\ref{axillary function} and the above Poincar\'e inequality  to get
\begin{align*}
& \|\nabla Eu\|_{L^{1}(\wz \Omega)} \\
\lesssim& \sum_{\substack{k\ge 1\\ 1\le j\le 2^{k_0+k+1}}} \int_{R_{k}^{(j)}} |\nabla \phi_{k}^{(j)}(x)| \, dx + \sum_{1\le j\le 2^{k_0+2}} \int_{R_{0}^{(j)}} |\nabla \phi_{k}^{(j)}(x)| \, dx \\
\lesssim& \sum_{\substack{k\ge 1 \\ 1\le j\le 2^{k_0+k+1}}} \left(|a_k^{(j)}-a_{k+1}^{(2j-1)}|+|a_k^{(j)}-a_{k+1}^{(2j)}| +|a_{k+1}^{(2j-1)}-a_{k+1}^{(2j)}|\right) \diam (Q_k^{(j)})\\
&  +\sum_{j=1}^{   2^{k_0+2}} \left(|a_1^{(j)}-a_{1}^{(j+1)}|+|a_1^{(j)}-a| +|a_{1}^{(j+1)}-a|\right)\diam (Q_1^{(j)}) \\ 
\lesssim& \sum_{\substack{k\in \mathbb N\\ 1\le j\le 2^{k_0+k+1}}} \left(  \int_{  Q_k^{(j)} \cup Q_{k+1}^{(2j-1)}} |\nabla u|\, dx +\int_{ Q_k^{(j)} \cup Q_{k+1}^{(2j)}} |\nabla u|\, dx + \int_{ Q_{k+1}^{(2j-1)} \cup Q_{k+1}^{(2j)}} |\nabla u|\, dx  \right) \\
& \qquad +\int_{\Omega}|\nabla u|\, dx,
\end{align*}
where the constants depend only on $C_1$. 

Additionally, by our construction, we only have uniformly finite overlaps for the Whitney-type sets in the summation. Therefore we have 
\begin{equation}\label{eqn201}
\|\nabla Eu\|_{L^{1}(\widetilde\Omega)} \le C(C_1) \int_{\Omega} |\nabla u|\, dx. 
\end{equation}

Let $B$ be a ball of radius $4\diam(\Omega)$ such that  $\Omega$ is  contained in $\frac 1 2 B$. Then 
\begin{align}\label{L1 control}
\| Eu\|_{L^{1}(B)}\ls & \sum_{\substack{k\ge 1\\ 1\le j\le 2^{k_0+k+1}}}\int_{R_k^{(j)}} |Eu(x)| \, dx + \int_{\Omega} |u(x)|\, dx \nonumber \\
\ls & \sum_{\substack{k\ge 1\\ 1\le j\le 2^{k_0+k+1}}}\int_{Q_k^{(j)}} |u(x)| \, dx  +\int_{\Omega} |u(x)|\, dx\ls \int_{\Omega} |u(x)|\, dx
\end{align}
where the constant depending only on $C_1$. Here we used \eqref{bound of function} with the fact that
$$|R_k^{(j)}|\le C(C_1) |Q_k^{(j)}|\sim_{C_1} |Q_{k+1}^{(2j-1)}| \sim_{C_1} |Q_{k+1}^{(2j)}| $$
coming from \eqref{distance estimate} and the fact that $Q_k^{(j)}$ is of $C(C_1)$-Whitney-type.

\subsection{Absolute continuity}\label{sec:jordanabscont}

We first check that $Eu$ is absolutely continuous along almost every line segment parallel to the 
coordinate axes for $u\in C^{\infty}(\mathbb R^2)\cap W^{1,\,1}(\Omega)$. 
Notice that $u$ is already smooth in $\overline{\Omega}$. 
Denoting the collection of the countable number of points $z_k^{(j)}$ by $Z$,   where  $k\ge 1$ and $1\le j\le 2^{k_0+k+1}$, or $k=0$ and $1\le j\le 2^{k_0+2}$,
we first claim that the restriction of $Eu$ to $\mathbb R^2 \setminus Z$ is continuous.

Indeed since $u$ is Lipschitz in $\overline{\Omega}$ and $Eu$ is continuous
in $\wz \Omega$ by our construction, we just need to verify continuity
of $Eu$ at points of $\partial\Omega\setminus Z$. Given a sequence of 
points $\{y_i\}_{i=1}^{\infty}$ in $\wz \Omega$ converging to 
$y\in \partial\Omega\setminus Z$, we claim that there is
a sequence of hyperbolic triangles $\{R_i\}_{i=1}^{\infty}$ among  $R_k^{(j)}$ 
such that, 
\begin{equation}\label{equat 3}
 y_i\in R_i \ \, \text{ and } \ \ R_i \to  y  \text{ when $i\to \infty$. }
\end{equation}
If so, then by \eqref{bound of function} we know that $Eu(y_i)$ is 
bounded from above and below respectively by one of the integral averages of $u$ over 
the corresponding Whitney-type sets $Q_i$ or over two neighbors of $Q_i$. 
Let us show that now also $Q_i$ converge to $y$ (inside $\Omega$). 
Indeed since $R_i$ converge to $y$, especially the three vertices of $R_i$ also converge to $y$ as $i\to \infty$, by the continuity of $\varphi$ and $\wz \varphi$ on the boundary, we know that $\diam(Q_i)\to 0$. Therefore $Q_i$  also converge to $y$. 
Since $u$ is Lipschitz in $\overline{\Omega}$, we 
get the continuity of $Eu(x)$ in $\mathbb R^2\setminus Z$.

Now let us show the existence of the asserted sequence $\{R_i\}_{i=1}^{\infty}$ as in \eqref{equat 3}. Fix $y\in \partial\Omega\setminus Z$, and for
any $s\in \mathbb N$ define
\begin{equation}\label{curve family}
\Gamma_s=\overline{\bigcup_{0\le k\le s} \bigcup_ j\gamma_k^{(j)}},
\end{equation}
where the second union in $j$ is over all the indices $j$ such that 
$\gamma_k^{(j)}$ is defined. Notice that $\Gamma_s$ is compact, 
and $y\not\in\Gamma_s$.
Therefore there exists $\delta_s>0$ such that $\dist(y,\,\Gamma_s)\ge\delta_s$. 
Thus if $\{y_i\}_{i=1}^{\infty}$ converges to $y$, then the lower indices 
of our hyperbolic triangles $R_{k(i)}^{(j)}=R_{k(i)}$ containing $y_i$ tend to 
$\infty$ when $i\to \infty.$  
This together with the fact that $\wz \varphi\colon \overline{\wz \Omega} 
\to \mathbb R^2\setminus \mathbb D$ is homeomorphic implies that $\wz \varphi(R_{k(i)})$ converges to a point on the unit circle  as $i\to \infty$, and thus $R_{k(i)} \to y$ as desired since $y_i\to y$ and $y_i\in R_{k(i)}$.
Thus  \eqref{equat 3} follows.

Since $Z$ is countable, almost every line parallel to the $x_1$-axis contains 
no point of $Z$. Additionally, for almost every such line $S$,
\begin{eqnarray}\label{integrable}
\int_{S\cap \wz \Omega} |\nabla Eu|\, dx <\infty 
\end{eqnarray}
by Fubini's theorem. Fix $S$ with these properties. Since $Eu$ is  locally Lipschitz in $\wz \Omega$,   $Eu|_{S}$ is absolutely continuous for  every closed segment $I\subset S$   with $I\subset \wz \Omega$.

Now for any line segment $I\subset S$, we first observe that $Eu|_{I}$ is 
certainly continuous as $Eu|_{\mathbb R^2\setminus Z}$ is continuous. 
Next we show that $Eu|_{I}$ is absolutely continuous. 
Let $z,w\in I.$ 

If $z,\,w\in \overline{\Omega}$, then the $L$-Lipschitz continuity  
of $u$ in $\overline{\Omega}$ gives
\begin{equation}\label{eka}
|Eu(z)-Eu(w)|\le L|z-w|.
\end{equation}
If $[z,\,w] \cap \overline{\Omega} = \emptyset$, then since $Eu$ is 
absolutely continuous on line segments in $S$, we obtain 
\begin{equation}\label{toka}
|Eu(z)-Eu(w)|\le \int_{[z,\,w]} |\nabla Eu|\,dx.
\end{equation}

For what is left, by symmetry we may assume that $z\in \overline{\Omega}$ 
while $w\in \wz \Omega$. Then let $z_0$ be the closest point of 
$\overline{\Omega}\cap I$ to $w$ between $z$ and $w.$ Then 
\begin{equation}\label{absolute continuity}
|Eu(z)-Eu(w)|\le |Eu(z)-Eu(z_0)|+|Eu(z_0)-Eu(w)| \le  
L|z-z_0|+\int_{[z_0,\,w]} |\nabla Eu|\, dx.
\end{equation}
For the last inequality, we used the facts that $Eu$ is continuous on $S$ 
and absolutely continuous on each closed subinterval $I\subset S$ when 
$I\subset \wz \Omega$, together with \eqref{integrable}.

Suppose that $\ez>0$ and we are given $z_1,\,\dots,\,z_{2n}\in I$ so that 
the one-dimensional
open intervals $(z_{2i-1},z_{2i}),$ $1\le i\le n,$ are pairwise disjoint. By
applying the relevant one of \eqref{eka},\eqref{toka} or
\eqref{absolute continuity} to each pair $z_{2i-1},z_{2i},$ our assumption 
\eqref{integrable} together with the absolute continuity of integrals
gives the existence of  $\delta>0$ so that 
 $\sum_{i=1}^{n} |Eu(z_{2i-1})-Eu(z_{2i})|\le \ez$ provided
 $\sum_{i=1}^{n} |z_{2i-1}-z_{2i}|<\delta.$ This implies the desired absolute
continuity property. The case of lines parallel to the $x_2$-axis is handled analogously.

Thus $Eu$ is absolutely continuous along almost every line segment parallel to the 
coordinate axes for $u\in C^{\infty}(\mathbb R^2)\cap W^{1,\,1}(\Omega)$, and with the norm estimates \eqref{eqn201} and \eqref{L1 control} we conclude that $E\colon C^{\infty}(\mathbb R^2)\cap W^{1,\,1}(\Omega) \to   W^{1,\,1}(B)$ is a (linear) extension operator. 

\subsection{Proof of Theorem~\ref{Jordan suff}}
By subsection~\ref{sec:jordanabscont}, the norm inequalities \eqref{eqn201} and \eqref{L1 control}, and the fact that the Lebesgue measure of $\partial \Omega$ is zero (see part (4) of Lemma~\ref{property of John domain}),
$$E\colon C^{\infty}(\mathbb R^2)\cap W^{1,\,1}(\Omega) \to   W^{1,\,1}(B), \quad  \text{where $\Omega\subset \frac{1}{2}B\subset B$}.$$
 is both linear and bounded with respect to the non-homogeneous $W^{1,\,1}$-norm. Since    $C^{\infty}(\mathbb R^2)\cap W^{1,\,1}(\Omega)$ is dense 
in $W^{1,\,1}(\Omega)$ 
provided that $\Omega$ is a Jordan domain  \cite{KZ2015}, 
linearity and boundedness of $E$ allow us to extend the domain of $E$ 
to the entire $W^{1,\,1}(\Omega)$.  This concludes Theorem~\ref{Jordan suff} because $B$  is a $W^{1,\,1}$-extension domain.

\section{Piecewise hyperbolic geodesics and cut-points}\label{sec:phg}

Recall that we relied on a hyperbolic triangulation of the complementary domain
in the construction of our extension operator in the Jordan case. In order
to obtain a suitable variant of this for the (strongly quasiconvex)
complement of our simply connected domain, we need a counterpart of hyperbolic geodesics used in the triangulation. 


Let $\varphi \colon \mathbb D \to \Omega$ be a conformal map. Our assumption 
that $\Omega^c$ is quasiconvex, say with constant $C_1,$ together with
Lemma~\ref{property of John domain} 
allows us to extend $\varphi$ \emph{continuously}  to entire $\overline {\mathbb D}$  (the extension may fail to  be a homeomorphism).
As usual, the extension is also denoted by $\varphi.$ 

We set $\wz\Omega := \rr^2 \setminus \overline{\Omega}$ and 
let $\{\widetilde\Omega_i\}_{i=1}^N$ be an enumeration of the connected 
components of $\widetilde\Omega$ with $N\in \mathbb N \cup\{+\infty\}$, so 
that $\wz \Omega_1$ is the (unique) unbounded component. Then each 
$\widetilde\Omega_i$ is Jordan for $i\ge 2$ by Lemma~\ref{complement of John}.


Because the boundary of $\Omega$ may have self-intersections, we eventually base our
construction on labeling via $\partial \mathbb{D}.$ Towards this end,
let $x',\, y'\in \partial \mathbb{D}$ be distinct points such that  $x:=\varphi(x')$, $y:=\varphi(y')$ satisfy $|x-y|<\delta$, where
$\delta>0$ will be fixed later. When $\delta<\pi/4,$
this guarantees that the hyperbolic geodesic between $x',\,y'$ in the complement
of the disk is contained in $B(0,\,3).$ 
A {\it piecewise hyperbolic geodesic} (in $\Omega^c$) associated to
$x,y$ via $\varphi$ (between 
$x,\,y \in \partial \Omega$) in the complement of  
$\Omega$ is any curve obtained via the following construction. 
%
%
%
%


Define
$x'_n=(1-2^{-n})x',$  
$y'_n=(1-2^{-n})y'.$ Then $x_n:=\varphi(x'_n),y_n:=\varphi(y'_n)\in \partial D_n,$
where  
$$D_n=\varphi(B(0,\,1-2^{-n})),$$ for $n\ge 2,$ is a Jordan domain.
Denote by $\wz D_n$ the complementary domain of $D_n$.
Since $\varphi$
is continuous up to the boundary, the points
$x_n,\,y_n\in \partial \wz D_n$ converge to $x,\,y,$
respectively.
Also observe that $ \Omega^c\subset \wz D_{n+1}\subset\wz D_n$, 
and $\Omega^c=\bigcap_n 
\wz D_n.$  
Pick a hyperbolic geodesic $\gamma_n\subset \wz D_n$ of $\wz D_n$ connecting 
$ x_n $ and $ y_n.$  
Since $\Omega^c$ is $C_1$-quasiconvex, Lemma~\ref{qconvex} and Lemma~\ref{inherited} show that
$\wz D_n$ is $c_1$-quasiconvex with $c_1=c_1(C_1).$ We now fix
$\delta_1=\delta_1(c_1)<\frac 1 {4c_1}$ as in Lemma~\ref{ulkodist}. Then Lemma~\ref{ulkodist}
guarantees that $|\wz \varphi_n (z)-\wz \varphi_n (w)|\le  \frac {1}{4}$ for any 
conformal map $\wz \varphi_n\colon\wz D_n\to \mathbb R^2\setminus \overline{\mathbb D}$
and $z,w\in \partial D_n$ with $|z-w|<\delta_1\diam(\Omega)$;  
recall here that such a $\wz \varphi_n$ extends continuously to the boundary
and hence we may apply Lemma~\ref{ulkodist} even for points on the boundary.
By the uniform continuity of $\varphi$ on $\mathbb D$  \cite[Page 100, Corollary 5.3]{P1992},
we may choose $\delta=\delta(\delta_1,\,C_1)$ so that
$$|x_n-y_n|<\delta_1 \diam(\Omega)$$
whenever $|x'-y'|<\delta;$ notice that $|x'_n-y'_n|<|x'-y'|.$ This determines the 
value of $\delta.$ Under this choice, 
$|\wz \varphi_n (x_n)-\wz \varphi_n (y_n)|<\frac 1 {4}$
and hence the $c_1$-quasiconvexity of $\wz D_n$ together with
Corollary~\ref{GH thm}
guarantees that
\begin{equation}\label{equ1000}
\ell(\gamma_n)\le c_2 |x_n-y_n|,
\end{equation}
where $c_2=c_2(c_1).$ 

Recalling that $x_n\to x $ and $y_n\to  y $
when $n$ tends to infinity, we conclude from \eqref{equ1000} that $\gamma_n\subset B(\varphi(x),M)$
and $\ell(\gamma_n)\le M$ for some $M$ independently of $n.$
Parametrize each $\gamma_n$ by arc length. Then, by the  
Arzel\`a-Ascoli lemma, we obtain via uniform convergence a curve 
$\gamma$ joining 
$ x $ and $ y$ in $\Omega^c.$ We call it a 
{\it piecewise hyperbolic geodesic joining $ x $ and $ y $.} 
Also notice that by \eqref{equ1000}, if $\varphi(x')=\varphi(y')$, then  the piecewise hyperbolic geodesic joining $ x $ and $ y $ is  a constant curve.

\begin{rem}
A natural question arises: If the approximating sequences  of $\varphi(x)$ and $\varphi(y)$ were chosen differently in the previous steps, especially if the preimages  of $x,\,y$ were chosen differently (but still distinct),  could this result in a different curve $\gamma$? This question is answered affirmatively by Lemma~\ref{unique} below, modulo reparametrization.

Moreover we could also define piecewise hyperbolic geodesics joining sufficiently close points near $\partial\Omega$, for example, the points on the piecewise hyperbolic geodesics above; this is the reason why we chose $\delta_1< \frac {1}{4c_1}$ when we applied Lemma~\ref{ulkodist} above. 
More precisely, the argument above applies whenever $x,\,y\in \Omega^c$ satisfy
$$|x-y|\le \delta \diam(\Omega),\, \dist(x,\,\partial \Omega)\le \delta \diam(\Omega),\, \text{ and } \dist(y,\,\partial \Omega)\le \delta \diam(\Omega)$$
with our choice of $\delta$. Indeed, regarding Lemma~\ref{ulkodist} with our approach above, it suffices to show that when $n$ is large enough  we have
$$|x-y|\le \delta \diam(D_n),\, \dist(x,\,\partial D_n)\le \delta \diam( D_n),\, \text{ and } \dist(y,\,\partial D_n)\le \delta \diam(D_n);$$
notice that $x,\,y\in \wz D_n$ for any $n\in \mathbb N$. This follows from the uniform continuity of $\varphi$.

Some of the lemmas below also hold for these  piecewise hyperbolic geodesics. However, in order to simplify the statements of lemmas, and since those additional curves are not needed in what follows, we only prove the lemmas in the case where the end points of piecewise hyperbolic geodesics are on the boundary of $\Omega$. 
\end{rem}



The following lemma justifies our terminology. 

\begin{lem}\label{lem_phg}
 Let $\Omega$ be a bounded simply connected domain whose complement is 
$C_1$-quasiconvex. Then there exist constants $C(C_1)$
and $\delta=\delta(C_1)>0$ 
such that, 
 for every pair of (distinct) $x,\,y\in \partial\Omega$
with
\begin{equation}\label{equ 300}
|x-y|\le \delta \diam(\Omega), 
\end{equation}
any piecewise hyperbolic geodesic
 $\gamma$ joining $x$ to $y$ in $\Omega^c$ satisfies
 \[
  \ell(\gamma[z,\,w])\leq C(C_1) |z-w|,
 \]
whenever $z,w\in \gamma$ and $\gamma[z,\,w]$ is any subcurve of $\gamma$ 
joining $z$ and $w.$
In particular, $\gamma$ is an injective curve (when parametrized by arc length).

Additionally, for any $1\le i\le N$ with $\gamma\cap \wz\Omega_i\neq\emptyset,$
we have that  $\gamma\cap \wz\Omega_i$ is a hyperbolic 
geodesic in $\wz\Omega_i.$ 
\end{lem}
\begin{proof}
Let $\gamma[z,w]$ be a subcurve of $\gamma$ joining $z$ and $w.$
Then by our construction of the piecewise hyperbolic geodesics, there are $t_n<s_n$ such that the restrictions of $\gamma_n$
to $[t_n,s_n]$ converge to $\gamma[z,w].$ We may assume that $z_n=\gamma_n(t_n)$
converges to $z$ and that $w_n=\gamma_n(s_n)$
converges to $w.$ 
Hence by the lower semi-continuity of length 
(as a functional on curves), Corollary~\ref{GH thm} and Lemma~\ref{inherited} 
one has
$$\ell(\gamma[z,\,w])\lesssim \liminf_{n\to \infty} 
\ell(\gamma_n[z_n,\,w_n])\lesssim \liminf_{n\to \infty} |z_n-w_n|\lesssim |z-w|.$$
Hence our first claim follows.

For the final claim, let 
$\wz \varphi_n\colon \mathbb R^2\setminus \overline{\mathbb D}\to \wz D_n$ to be conformal for each $n\in \mathbb N$, 
and extend it to $\overline{\mathbb R^2}\setminus \overline{\mathbb D}$ by setting 
$\wz \varphi_n(\infty)=\infty$. Also when $i\ge 2$ and 
$\gamma\cap \wz\Omega_i\neq\emptyset$, let $x_0\in\wz\Omega_i$ 
be a point of maximal distance to $\partial\wz\Omega_i$; 
for $\wz \Omega_1$ we set $x_0=\infty$.

Denote by $\phi\colon \overline{\mathbb R^2} \setminus \overline{\mathbb D} 
\to \overline{\mathbb R^2} \setminus \overline{\mathbb D}$ a M\"obius 
transformation such that $\phi(\infty)=\wz\varphi_n^{-1}(x_0)$. Then the 
compositions $\wz \varphi_n\circ \phi\colon \overline{\mathbb R^2} 
\setminus \overline{\mathbb D} \to \wz D_n\cup\{\infty\}$ form  a normal 
family by \cite[Theorem 19.2]{V1971}. Hence there is a subsequence  that
converges locally uniformly to a conformal map 
$\wz \varphi\colon \overline{\mathbb R^2} \setminus \overline{\mathbb D} \to 
\wz \Omega_i$; see \cite[Theorem 21.1]{V1971}. 
Notice that hyperbolic geodesics are invariant under conformal maps. 
Therefore $\gamma_n$ can be regarded as a hyperbolic geodesic induced by 
$\wz \varphi_n\circ \phi$ on the Riemann sphere by the uniqueness of the limit, and hence also $\gamma$ as 
induced by $\wz\varphi\circ \phi$. Thus the part of $\gamma$ in 
$\wz \Omega_i$ is a hyperbolic geodesic; it easily follows from the above
argument that  $\gamma \cap \wz \Omega_i$ is connected. 
%
\end{proof}


We need the following further properties of piecewise hyperbolic geodesics.

\begin{lem}\label{curve condition for geodesic}
Let $\Omega$
 and $\delta$ 
 be as in Lemma~\ref{lem_phg}, and $x,\,y\in \partial \Omega$.
Given any piecewise hyperbolic geodesic 
$\Gamma\subset \Omega^c$ associated to the pair $x,y$ one has
$$\Gamma\cap \partial\Omega\subset \gamma\cap \partial\Omega, $$
for any curve $\gamma\subset \Omega^c$ connecting $x$ 
and $y$. 
Moreover, if the complement of $\Omega$ satisfies the curve condition 
\eqref{eq:curvecondition}, then
$$  \int_\Gamma\frac{1}{\chi_{\rr^2\setminus \partial\Omega}(z)}\,ds(z) \le 
C|x-y|,$$
where the constant $C$ depends only on the constant in \eqref{eq:curvecondition}.
\end{lem}
\begin{proof}
Let $\gamma\subset \Omega^c$ be an arbitrary curve joining $x,\,
y\in \partial\Omega$. 
Denote the corresponding   hyperbolic geodesics approximating 
$\Gamma$ by $\Gamma_n\subset \wz D_n$ with endpoints $x_n,\,y_n$, according to the definition of 
$\Gamma$.

Suppose $z\in \Gamma\cap \partial \Omega$. 
We claim that $z\in \gamma\cap \partial \Omega$. We may assume that $z\neq x,\, z\neq y$. 
Let $\{z_n\},\,z_n\in \Gamma_n$ be a sequence of points converging to $z$, 
and define for every $n\ge 2$, $$B_n=B\left(z_n,\,\frac 1 2 
\dist(z_n,\,\partial D_n)\right). $$

Since each $B_n$ is a $2$-Whitney-type set in $\wz D_n$, then Lemma~\ref{ball separation} gives a constant $c$ independent of $n$ such that 
$\gamma'\cap cB_n\neq\emptyset$ for any $\gamma'\subset \wz D_n$ joining $x_n$ and $y_n$.  
Since $x_n\to x$ and $y_n \to y$, then when $n$ is large enough so that
$$ |x_n-z_n|\ge C(C_1)|x-x_n|,\quad|y_n-z_n|\ge C(C_1)|y-y_n|,$$
the $C(C_1)$-quasiconvexity of $\wz D_n$ given by Lemma~\ref{inherited}
implies that there exist curves in $\wz D_n$ joining $x$ to $x_n$ and $y$ to $y_n$, respectively, such that these two curves do not intersect $c B_n$. 
By concatenating  them with $\gamma$, we apply Lemma~\ref{ball separation}  to conclude that 
\begin{equation}\label{equ 500}
\gamma\cap c B_n\neq \emptyset
\end{equation}
 when $n$ is large enough; recall that $\gamma\subset  \Omega^c\subset \wz D_n$. 

Observe that by $z\in \Gamma\cap \partial \Omega$, 
we have that $$\lim_{n\to \infty}\dist(z_n,\,\partial D_n)= 0. $$
Then by the assumption $z_n\to z$ when $n\to \infty$, 
we conclude that $cB_n\to z$, and hence $z\in \gamma$ by \eqref{equ 500}. 
Thus $z\in \gamma\cap \partial \Omega$. This proves the claim, and the 
first part of the lemma follows. The second part follows from 
Lemma~\ref{lem_phg} together with the first part applied to a curve $\gamma$
satisfying \eqref{eq:curvecondition}; note that $\Gamma$ is an injective curve. 
\end{proof}

For further reference we record the following consequences of the above lemmas.

\begin{lem} \label{parempi}
Given $\Omega$ as in Lemma~\ref{lem_phg}, and a piecewise hyperbolic geodesic $\Gamma,$ a subcurve
$\Gamma[x,y]$ of $\Gamma$ and a curve $\gamma$ joining $x$ to $y$ in $\Omega^c,$ we
have
$$\Gamma[x,y]\cap \partial \Omega\subset \gamma\cap \partial \Omega.$$

Moreover  $\Gamma\cap \partial \wz\Omega_i$ consists of  at most two points for each $1\le i\le N$, and it is a doubleton if and only if $\Gamma \cap \wz\Omega_i$ is a hyperbolic geodesic joining boundary points. 
\end{lem}

\begin{proof}
Let $x_0,\,y_0$ be the end points of $\Gamma$ on $\partial \Omega$. 
Then the concatenation of three curves, the subcurve $\Gamma[x,\,x_0]$ of $\Gamma$, $\gamma$, and the subcurve $\Gamma[y,\,y_0]$ of $\Gamma$, is a curve $\beta$ joining $x_0$ to $y_0$. Then the first conclusion follows from Lemma~\ref{curve condition for geodesic} applied to $\Gamma$ and $\beta$. 

If $\Gamma\cap \partial \wz\Omega_i$ has more than two points, then consider the first  and the last points of $\Gamma$ intersecting $\partial \wz\Omega_i$ according to its parametrization, and join them by a hyperbolic geodesic $\az$ inside $\wz\Omega_i$; recall that $\wz \Omega_i$ is Jordan by Lemma~\ref{complement of John}. Then we obtain a new curve $\gamma$ by rerouting the subarc of $\Gamma$ via $\az$. This contradicts Lemma~\ref{curve condition for geodesic} for $\gamma$ since we  have assumed that $\Gamma\cap \partial \wz\Omega_i$ has more than two points. Thus $\Gamma\cap \partial \wz\Omega_i$ has at most two points. 

Let us show the last part of the lemma. We may assume that $\wz\Omega_i$ is bounded; otherwise we just apply suitable M\"obius transformations on the Riemann sphere. 
If $\Gamma \cap \wz\Omega_i$ is a hyperbolic geodesic  joining boundary points, then $\Gamma\cap \partial \wz\Omega_i$ consists of two points since $\wz \Omega_i$ is Jordan. For the other direction, by Lemma~\ref{lem_phg} it suffices to show that $\Gamma\cap  \wz\Omega_i$ is non-empty. On the contrary let us assume that $\Gamma\cap \wz\Omega_i=\emptyset$. Then by joining the given two points  $x,\,y\in \Gamma\cap \partial \wz\Omega_i$  with the hyperbolic geodesic $\az \subset \wz\Omega_i$, we obtain a Jordan curve $\gamma'$ by concatenating $\Gamma[x,\,y]$ with $\az$; see Lemma~\ref{lem_phg} and note that the open subarc $\Gamma(x,\,y)$ is contained in the exterior of $\wz \Omega$ by our assumption and the conclusion of the previous paragraph. However by our construction, $\partial \Omega$ intersects both the interior and the exterior of the Jordan domain  given by $\gamma'$. This implies that there are points belonging to $\Omega$ on both sides of $\gamma'$. This contradicts the Jordan curve theorem since $\Omega$ is connected while $\gamma'\subset \Omega^c$. Therefore we conclude that $\Gamma\cap \wz\Omega_i\neq \emptyset$.  
\end{proof}

The following corollary is a by-product of the lemmas above. 
\begin{lem}\label{unique}
Given a domain $\Omega$ together with a pair of points $x,\,y\in \partial \Omega$ as in Lemma~\ref{lem_phg}, there exists a unique piecewise hyperbolic geodesic joining them, up to a reparametrization. 
\end{lem}
\begin{proof}
The existence follows directly from Lemma~\ref{lem_phg}, and we only need to show the uniqueness.

Suppose that there are two piecewise hyperbolic geodesics $\gamma_1,\,\gamma_2$ connecting $x,\,y$. Then by Lemma~\ref{curve condition for geodesic} we have
\begin{equation}\label{equ 60}
\gamma_1\cap \partial \Omega=\gamma_2\cap \partial \Omega. 
\end{equation}

Since $\Omega^c=\partial(\Omega^c)\cup \wz \Omega=\partial \Omega \cup \wz \Omega$, we then only need to show that $\gamma_1$ coincides with $\gamma_2$ componentwise. This follows from \eqref{equ 60}, Lemma~\ref{lem_phg} and Lemma~\ref{parempi} according to the uniqueness of hyperbolic geodesics in   Jordan domains. 
\end{proof}

In order to deal with self-intersections of the boundary, it is convenient
to classify points on the boundary in terms of their preimages.

\begin{defn}
 Let $\varphi  \colon \mathbb{D} \to \Omega$ be conformal and assume that it
extends 
as a continuous map (still denoted by $\varphi$)
 to entire $\overline{\mathbb{D}}$. 
 A point $x \in \partial \Omega$ is called 
a \emph{cut-point} if $\varphi^{-1}(x)$ is not a singleton. 
A point $x$ that is not a cut-point is called \emph{one-sided}.
\end{defn}

We warn the reader that our terminology above is not standard. See e.g.\  \cite[Chapter 14]{R1974}
for a discussion on the relation with the usual topological definition of 
cut-points. As stated, the definition appears to depend on the choice of the
continuous conformal map. This is not the case since any other (continuous)
conformal map differs from the chosen one only by a precomposition via a 
M\"obius 
self-map of the (closed) disk.



\begin{lem}\label{lma:twosidedhyperbolic}
Let $\Omega \subset \rr^2$ be a bounded simply connected John domain.
 Then the set $T$ of cut-points can be characterized as
\begin{equation}\label{two sided}
  T = \bigcup_{\gamma} \gamma^o \cap \partial\Omega,
\end{equation}
 where the union is over all piecewise hyperbolic geodesics $\gamma$ 
associated to pairs of (suitable) points in $\partial \mathbb D$
 and $\gamma^o$ denotes the curve $\gamma$ without its endpoints.
 Moreover, the union can equivalently be taken over a countable set of 
curves. 
\end{lem}
\begin{proof}
 First of all, by part (2) of Lemma~\ref{property of John domain} we have a continuous 
map $\varphi\colon \overline{\mathbb D} \to \overline{\Omega}$ which is 
conformal in $\mathbb D$.

Let $\{x_i\}_{i=1}^\infty \subset \partial \mathbb D$ be dense. 
We first claim that 
\begin{equation}\label{inclusion 1}
  T \subset \bigcup  \gamma_{i,j}^o \cap \partial\Omega. 
\end{equation}
 where the union is taken over all the piecewise hyperbolic geodesics $\gamma_{i,j}$  joining those 
$\varphi(x_i)$ to $\varphi(x_j)$ with $$| \varphi(x_i) - \varphi(x_i)|\le \delta \diam(\Omega).$$ Recall that $\gamma_{i,j}$ is obtained via
a subsequence of the domains $\wz D_n.$
Since we only have countably many curves, by a usual diagonal argument,
we may assume that each of them is obtained via the same subsequence.
Observe that by Lemma~\ref{lem_phg}, these curves are all injective.

Fix $z \in T$. By the definition of cut-points there 
exist $z_1,z_2 \in \partial \mathbb D$, $z_1 \ne z_2$,
 such that $\varphi(z_1) = \varphi(z_2) = z$.  Recall here that the image
of any nontrivial arc of $\partial \mathbb D$ under our conformal map
that is continuous up to the boundary is connected and not a singleton. By the density of 
$\{x_i\}_{i=1}^\infty$
 there exist $i,j \in \nn$ such that 
 $$\varphi(x_i) \ne z \ne \varphi(x_j),\ \ | \varphi(x_i) - \varphi(x_j)|\le \delta \diam(\Omega)$$
with $\varphi(x_i)\neq \varphi(x_j)$ and $x_i,\,x_j$ divide $\partial \mathbb D$ into two connected components; 
one contains $z_1$ and the other contains $z_2$.
The curve 
$\varphi([0,\,z_1]\cup[0,\,z_2])$ is Jordan since $z_1\neq z_2$ while $\varphi(z_1)=\varphi(z_2)$, and $\varphi$ is injective in $\mathbb D$, where $[0,\,z_i]$ is the 
(radial) line segment connecting the origin and $z_i$ for $i=1,\,2$. 

We show that the points $\varphi(x_i)$ and $\varphi(x_j)$
belong to different connected components of 
$\rr^2 \setminus \varphi([z_1,0]\cup[0,z_2])$. To begin with, since $\varphi$ is continuous up to the boundary, then $$\varphi([z_1,0]\cup[0,z_2])$$
divides $\Omega$ into two components, coming from the two components of $\mathbb D\setminus ([z_1,0]\cup[0,z_2])$. 
Recall that $x_i$ and $x_j$ are in different connected components of $\partial \mathbb D\setminus\{z_1,\,z_2\}$. 
Then by taking two sequences of points inside the components of $\mathbb D\setminus ([z_1,0]\cup[0,z_2])$ approaching $x_i$ and $x_j $, respectively, we conclude that $\varphi(x_i) $ and $\varphi(x_j) $ are not in the same component of $\rr^2 \setminus \varphi([z_1,0]\cup[0,z_2])$ since $\varphi$ is continuous up to the boundary and neither $\varphi(x_i)$ nor $\varphi(x_j)$ belong to $\varphi([z_1,0]\cup[0,z_2])$. 
Consequently,
$\gamma_{i,j}$ has to intersect the Jordan curve $\varphi([z_1,0]\cup[0,z_2])$. 
Since $\varphi([z_1,0]\cup[0,z_2]) \cap \Omega^c = \{z\}$ and 
$\gamma_{i,\,j}\subset \Omega^c$, we know that 
$$\gamma_{i,j}\cap \varphi([z_1,0]\cup[0,z_2]) =\{z\}$$
and thus the claim is 
proved since $\varphi(x_i),\, z,$  and $\varphi(x_j)$ are distinct. Therefore we obtain 
\eqref{inclusion 1}. 
 

 Let us then show the other inclusion. Let $\gamma$ be a piecewise hyperbolic 
geodesic joining two different points in $\partial \Omega^c$
 and let $z \in \gamma \cap \partial\Omega$ while $\varphi(y_1) \ne z \ne \varphi(y_2)$. 
Let $y_1,y_2 \in \partial \mathbb D$ be such that 
$\varphi(y_1)$ and $\varphi(y_2)$ are the endpoints of $\gamma$.
 Now the curve 
 $$\gamma' := \varphi([y_1,\,0]\cup[0,\,y_2]) \cup \gamma$$ is 
Jordan by Lemma~\ref{lem_phg} and it divides $\Omega$ into two
 connected components $\Omega_1$ (the unbounded component) and $\Omega_2$ 
(the bounded component). 
Let $\Omega'$ be the Jordan domain given by $\gamma'$.

Let us assume contrary to the claim that $z \notin T$. 
Then there exists $r>0$ such that $B(z,\,r)$ intersects
 only one of the components, either $\Omega_1$ or $\Omega_2$; otherwise we have two 
sequences of points converging to $z$ in different components, i.e.\ there exist two sequences in different component of $\mathbb D\setminus ([y_1,
,0]\cup [0,\,y_2])$ approximating $\varphi^{-1}(z)$, which by 
the continuity of $\varphi$
contradicts the assumption that $z \notin T$ but $z\notin \varphi([y_1,\,0]\cup[0,\,y_2])$ . 

Assume that $B(z,\,r)\cap \Omega_1= \emptyset$. 
Since $z\in \gamma\cap \partial \Omega$, by the first part of 
Lemma~\ref{curve condition for geodesic} we know that there is no path 
connecting 
$\varphi(y_1)$ and $\varphi(y_2)$ in
$\Omega^c\setminus \{z\}.$ 
Towards a contradiction, we apply the fact that 
$\gamma\subset \partial \Omega'$ and the 
Jordan-Schoenflies theorem to $\Omega'$.

Indeed, by the Jordan-Schoenflies theorem (see e.g.\ \cite{T1992})  there is a homeomorphism 
$\phi\colon\mathbb R^2\to\mathbb R^2$ such that $\phi(\Omega')=\mathbb D$, 
and then $\phi(\gamma)\subset \partial \mathbb D$. 
Recall that $\varphi(y_1)\neq z \neq \varphi(y_2)$. 
Choose $0<\ez<1$ such that 
$$\min\{|\varphi(y_1)-z|,\,|\varphi(y_2)-z|\}>\ez r.$$ 
Then $\varphi(y_1),\,\varphi(y_2)\notin B(z,\,\ez r)$, and 
$\phi(B(z,\,\ez r))$ is also a Jordan domain. Certainly $\phi(z)$ is an 
interior point of $\phi(B(z,\,\ez r))$, and hence there exists $\delta>0$ 
such that $B(\phi(z),\,\delta)\subset \phi(B(z,\,\ez r))$. 
Denote by $\sigma$ the subarc of 
$\partial\left(\mathbb D\cup B(\phi(z),\,\delta)\right)$ which 
joins $\phi(\varphi(y_1))$ and $\phi(\varphi(y_2))$ and reroutes 
$\phi(\gamma)$, and let $\Gamma=\phi^{-1}(\sigma)$.

Observe first that $\Gamma\subset\Omega^c$ by our construction and the 
assumptions that $\Omega_2\subset \Omega'$ and 
$B(z,\,\ez r)\cap \Omega_1=\emptyset$. 
Also $\Gamma$ joins $\varphi(y_1)$ and $\varphi(y_2)$ by not passing through 
$z$. Therefore we obtain the desired curve, which leads to a contradiction.
A similar argument can be applied to the case where  $B(z,\,r)\cap \Omega_2= \emptyset$.  

The existence of countably many curves follows directly from the proof above. 
\end{proof}

An immediate consequence of Lemma~\ref{curve condition for geodesic} and 
Lemma \ref{lma:twosidedhyperbolic} is the following corollary.

\begin{cor}\label{cor:twosidedzeromeasure}
 Let $\Omega \subset \rr^2$ be a simply connected bounded John domain whose 
complement 
 satisfies the curve condition \eqref{eq:curvecondition} in Theorem 
\ref{thm:main}.
 Then the set of cut-points of $\partial \Omega$ 
has $\ch^1$-measure zero.
\end{cor}

Let us record the following consequence of the proof of Lemma~\ref{lma:twosidedhyperbolic}. 

\begin{lem}\label{two side cut point}
Let $\Omega \subset \rr^2$ be a simply connected bounded John domain whose complement is quasiconvex and $\varphi \colon \mathbb D \to \Omega$ be a conformal map. 
Also let $\gamma$ be a piecewise hyperbolic geodesic joining two distinct points $\varphi(y_1),\,\varphi(y_2)\in \partial \Omega$, and let $z\in \gamma \cap \partial \Omega$ with $\varphi(y_1)\neq z \neq \varphi(y_2)$. Then $z$ is a cut-point, and every open disk centered at $z$ has non-trivial intersection both with $\Omega_1$ and with $\Omega_2$, where $\Omega_1$ and $\Omega_2$ are the two components of $\Omega\setminus(\varphi([y_1,\,0]\cup[0,\,y_2]) \cup \gamma) = \Omega\setminus(\varphi([y_1,\,0]\cup[0,\,y_2]))$. 
\end{lem}

\section{Sufficiency for simply connected domains}\label{sec:sufficiency}

We construct the desired extension via a modification to our procedure
for the Jordan case. The first obstacle is that we cannot anymore use
hyperbolic triangles in the complementary domain since there need not be
a complementary domain to work with. To overcome this, we use piecewise
hyperbolic geodesics to obtain a desired decomposition in each of the 
components of the interior of the
complement of $\Omega.$ This allows us to mimic the Jordan case, but we have
to work harder to verify the regularity of our extension. 
Observe that by \eqref{eq:curvecondition}, $\Omega^c$ is $C_1$-quasiconvex with the constant $C_1$ from \eqref{eq:curvecondition}.

\subsection{Decomposition of the complement}
Let  $\varphi\colon \overline{\mathbb D}\to \overline{\Omega}$ be a
conformal map which is continuous up to the boundary, given by (1) and (2) of   Lemma~\ref{property of John domain}. 
The decomposition of $\Omega$ is the same as in Section \ref{sec:jordandeco} 
for the Jordan case.
The Whitney-type sets in $\Omega$ are denoted by $Q_k^{(j)}$ and 
$z_k^{(j)} = \varphi(x_k^{(j)})$ stand for the images of the
endpoints on $\partial \mathbb D$ of the radial rays used in the decomposition of 
$\mathbb{D}$. Again our decomposition begins with a suitable $k_0$ according to the constant $\delta$ in the previous section such that
  $$|z_{k_0}^{(j)} -z_{k_0}^{(j-1)} |\le \delta \diam(\Omega), \quad \forall 1\le j\le 2^{k_0+2}.$$

This time the complement $\Omega^c$ will be simultaneously
triangulated inside each connected component
of $\wz \Omega=\mathbb R^2\setminus\overline{\Omega}$ using hyperbolic geodesics. For this purpose 
we consider all pairs $(z_k^{(j-1)} ,\, z_k^{(j)})$ and pick for each of
them a piecewise hyperbolic geodesic. 
It may happen that $z_k^{(j-1)}=z_k^{(j)}$ for some $j,\,k\in \mathbb N$, but recall that every piecewise hyperbolic geodesic is defined via labeling on $\partial \mathbb D$, namely $x_k^{(j-1)}$ and $x_k^{(j)}$ for this pair of points, and in this case the piecewise hyperbolic geodesic is a constant curve. 
Since we deal with a countable collection of curves
and our piecewise hyperbolic geodesics are obtained via the Arzel\`a-Ascoli 
theorem, we may assume that each of them is obtained via the same subsequence
of the conformal maps $\wz \varphi_n.$ 
In other words, all the components of $\wz \Omega$ are triangulated at the same time.

Recall from Lemma~\ref{lem_phg}
that a piecewise hyperbolic geodesic is a shortest curve, up to a 
multiplicative constant, such that its
restriction to any of the connected components of $\wz \Omega$ is a 
hyperbolic geodesic (or empty).
Moreover by Lemma~\ref{lem_phg} and \eqref{distance estimate} we again have
\begin{equation}\label{length estimate}
 \ell(\gamma_k^{(j)})\ls |z_k^{j-1} - z_k^{j}| \lesssim \diam(Q_k^{(j)}),
\end{equation}
with the constant depending only on $C_1$. 


Next we study the possible cases of degenerated hyperbolic triangles; see Figure~\ref{fig:degenerated hyperbolic triangles}.

\begin{lem}\label{degenerated triangle}
For $k=1,\,2,\,3$, let $\gamma_k$ be a piecewise hyperbolic geodesics with two 
different endpoints $z_{k},\,z_{k+1}\in\partial \Omega$ respectively; suppose that $z_4=z_1$. Then there are two possibilities: 
\begin{enumerate}[1)]
\item Each of the curves is contained in the union of the other two, and the common set of these three curves is a single point. 
\item   For a specific index $1\le j_0\le N,$ the intersection of each $\gamma_k, 1\le k\le 3$ with $\wz\Omega_{j_0}$ is a hyperbolic geodesic, and these three geodesics form a hyperbolic triangle within $\wz \Omega_{j_0}$. For any other
$1\le j\le N, j\neq j_0$, if two of the curves $\gamma_k$ intersect $\wz\Omega_j$, their respective intersections must coincide.
\end{enumerate}
\end{lem}
\begin{proof}
First of all if $y\in \gamma_1\cap \partial \Omega$, 
then $\gamma_2$ or $\gamma_3$ must contain $y$ by 
Lemma~\ref{curve condition for geodesic}. Namely 
\begin{equation}\label{inclusion}
\gamma_1\cap \partial \Omega \subset (\gamma_2\cup \gamma_3) \cap \partial \Omega.
\end{equation}
The analogous statements also hold for $\gamma_2$ and $\gamma_3$.

Let $y_0$ be the last point along $\gamma_1$ from $z_2$ towards $z_1$
such that $y_0\in \gamma_1\cap\gamma_2$. We claim that 
$y_0\in \partial \Omega.$

Suppose on the contrary that $y_0\notin \partial \Omega.$ Then 
$y_0\in \wz \Omega_j$
for some $1\le j \le N,$ and $\wz \Omega_j$ is Jordan on the Riemann sphere by Lemma~\ref{complement of John}.
Let $y_1$ be the first point on $\gamma_1\cap
\partial \wz \Omega_j$ when we trace (along $\gamma_1$) towards $z_2$  from $y_0.$ 
Consider the
subcurves $\gamma_{1,0}$ and $\gamma_{2,0}$ of $\gamma_1,\gamma_2,$ respectively,
between $y_0$ and $z_2.$ By Lemma~\ref{parempi} we have that $y_1\in \gamma_{2,0}.$
It follows from Lemma~\ref{lem_phg} that the intersections of both $\gamma_1$
and $\gamma_2$ with  $\wz \Omega_j$ are hyperbolic geodesics that contain
both $y_1$ and $y_0.$ By mapping $\wz \Omega_j$ conformally to $\mathbb D$ 
 (or to the complement of the closed
unit disk if $\wz \Omega_j$ is unbounded), the 
uniqueness of hyperbolic 
geodesics yields that the intersection of $\gamma_1$ with $\wz \Omega_j$
equals to the intersection of $\gamma_2$ with $\wz \Omega_j.$ This contradicts
the definition of $y_0$, and conclude that $y_0\in \partial \Omega$.  

Consider again the subcurves  $\gamma_{1,0}$ and $\gamma_{2,0}$ from the previous
paragraph, joining $y_0$ to $z_2$. We conclude by Lemma~\ref{parempi} that 
\begin{equation}\label{equ 20}
\gamma_{1,0}\cap \partial \Omega=\gamma_{2,0}\cap \partial \Omega.
\end{equation} 
Moreover, if 
$\gamma_{1,0}$ intersects $\wz \Omega_j,$ then this intersection $\beta$  is by 
Lemma~\ref{lem_phg} a hyperbolic geodesic, say with endpoints $w_1,w_2.$ 
Then $w_1,w_2\in \partial \Omega$ and it follows that $w_1,w_2\in \gamma_{2,0}.$
Consider the corresponding subcurve $\alpha$ of $\gamma_{2,0}.$ Since $\beta \subset \wz \Omega_j,$ it follows from Lemma~\ref{parempi} that $\alpha \subset
\wz \Omega_i$ is a hyperbolic geodesic.  As in the previous paragraph, we conclude that $\alpha=\beta.$
Hence  $\gamma_{1,0}\subset \gamma_{2,0}.$
By reversing the roles of  $\gamma_{1,0}$ and $\gamma_{2,0}$ we deduce that
\begin{equation}\label{ekatsamat}
\gamma_{1,0}=\gamma_{2,0}.
\end{equation}


Suppose that $y_0\in \gamma_3.$ In this case, we repeat the argument from the
previous paragraph first for the subcurve $\gamma_{1,2}$ of $\gamma_1$ between
$z_1$ and $y_0$ and the subcurve $\gamma_{3,4}$ of $\gamma_3$ 
between $z_1=z_4$ and $y_0$ and after that for the subcurve $\gamma_{3,3}$ of
$\gamma_3$ between $z_3$ and $y_0$ and the subcurve $\gamma_{2,2}$
of $\gamma_2$ between
$z_3$ and $y_0.$ This gives us $\gamma_{1,2}=\gamma_{3,4}$ and 
$\gamma_{3,3}=\gamma_{2,2}$ and we conclude that $\gamma_3\subset \gamma_1\cup 
\gamma_2.$ 
 

We are left with the case $y_0\notin \gamma_3.$ We
claim that
$y_0\in \overline{\wz \Omega}_{j_0}$ for some $j_0$ such that 
$\gamma_1\cap \wz \Omega_{j_0}$ and $\gamma_2\cap \wz \Omega_{j_0}$ do not intersect. 
Indeed, by the definition of $y_0$ and  \eqref{inclusion}
we have that
\begin{equation}\label{empty gamma 1 2}
    (\gamma_1 \setminus \gamma_{1,\,0}) \cap \gamma_2=\emptyset 
\end{equation}
and 
$$(\gamma_1 \setminus \gamma_{1,\,0}) \cap \partial \Omega\subset \gamma_3\cap \partial \Omega. $$
Note that $\gamma_3 \cap \partial \Omega$ is a closed set, and $y_0 \in \partial \Omega$ while $y_0 \notin \gamma_3$. This implies that there is no sequence of points in $(\gamma_1 \setminus \gamma_{1,\,0}) \cap \partial \Omega$ converging to $y_0$. 
Let $y_0'$ be the last point in $(\gamma_1 \setminus \gamma_{1,\,0}) \cap \partial \Omega$  towards $y_0$. Then the open subcurve $\gamma_1(y_0,\,y_0')$ is contained in the interior of $\Omega^c$, and by the connectedness it is contained in some component $\wz \Omega_{j_0}$ of $\wz \Omega$. 
 By applying Lemma~\ref{lem_phg} to $\gamma_1$ and the Jordan domain ${\wz \Omega}_{j_0}$, we conclude that 
$$\gamma_1(y_0,\,y_0')\subset \gamma_1 \setminus \gamma_{1,\,0}$$
is a hyperbolic geodesic in the Jordan domain ${\wz \Omega}_{j_0}$, particularly $y_0\in \overline{\wz \Omega}_{j_0}$, and  
$\gamma_1(y_0,\,y_0')$
has no intersection with  $\gamma_2$ due to \eqref{empty gamma 1 2}. This yields the claim. 

Now by  Lemma~\ref{lem_phg}, \eqref{inclusion} and the definition of $y_0$, 
the endpoint $w_1$ on $\partial \wz \Omega_{j_0}$ of the part of $\gamma_1$ in 
$\wz \Omega_{j_0}$ has to be contained in $\gamma_3.$ Arguing as above,
we conclude that the remaining part of $\gamma_1$ coincides with a subcurve
of $\gamma_3.$ Similarly, the endpoint $w_2 \in \partial \wz \Omega$  
of the part of $\gamma_2$ in $\wz \Omega_{j_0}$ must be contained in  $\gamma_3,$
and we conclude that the remaining part of $\gamma_2$ coincides with a subcurve
of $\gamma_3.$ Finally, let $\gamma_{3,1,2}$ be the subcurve of $\gamma_3$ 
joining $w_1,w_2.$ Since the union of the (closures) of the subcurves
of $\gamma_1,\gamma_2$ in  $\wz \Omega_{j_0}$ also joins $w_1$ to $w_2$ and only
intersects $\partial \Omega$ in the set $\{y_0,w_1,w_2\}$ and $y_0\notin
\gamma_3,$ Lemma~\ref{parempi} implies that $\gamma_{3,1,2}$ joins $w_1$ to $w_2$ in 
$\wz \Omega_{j_0}.$
By Lemma~\ref{lem_phg} all of $\gamma_{3,1,2}$ and the subcurves
of $\gamma_1,\gamma_2$ in  $\wz \Omega_{j_0}$ are hyperbolic geodesics and it
follows that they form a hyperbolic triangle.


\end{proof}

\begin{figure} 
 \centering
 \includegraphics[width=0.85\textwidth]{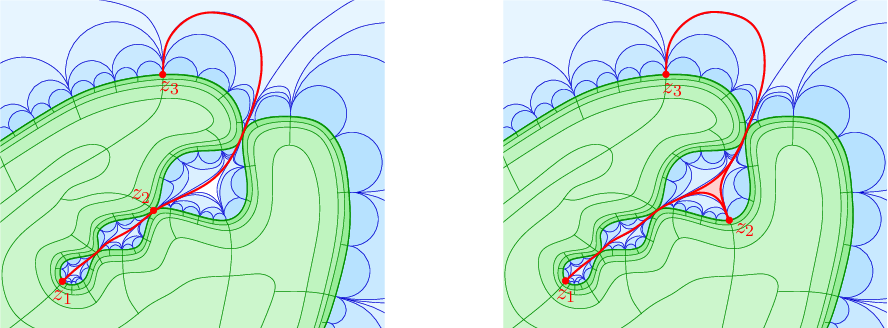}
 \caption{Here are the two possible cases of degenerated hyperbolic triangles.}
 \label{fig:degenerated hyperbolic triangles}
\end{figure}

By Lemma~\ref{degenerated triangle}, we can define the (degenerated) 
hyperbolic triangles $R_k^{(j)}$ similarly as in the Jordan case, 
namely the bounded (relatively) closed set enclosed by the union of $\gamma_k^{(j)}$, $\gamma_{k+1}^{(2j-1)}$ and 
$\gamma_{k+1}^{(2j)}$, and there is at most one connected component 
$\wz\Omega_i$ of $\wz \Omega$ such that $R_k^{(j)}\cap \wz \Omega_i$ is a (non-degenerated)
hyperbolic triangle, denoted by $R_{k,\,i}^{(j)}$.

For every $k \ge 1$, 
$1 \le j \le 2^{k_0+k+1}$ and $1\le i \le N,$  set
$$\gamma_{k,\,i}^{(j)}= \gamma_k^{(j)}\cap \wz\Omega_i.$$ Notice that 
$\gamma_{k,\,i}^{(j)}$ may well be empty for a choice of $k,j,i.$ Nevertheless
$$\bigcup_{i=1}^{N} \gamma_{k,\,i}^{(j)}=\gamma_k^{(j)}\cap \wz \Omega;$$ 
recall  also  \eqref{two sided} in Lemma~\ref{lma:twosidedhyperbolic}. We call 
$\gamma_{k,\,i}^{(j)}$ the $i$-th component of  $\gamma_k^{(j)}$, 
call $\gamma_{k,\,i}^{(j)}$ the {\it mother} of 
$\gamma_{k+1,\,i}^{(2j-1)}$ and $\gamma_{k+1,\,i}^{(2j)}$, and 
refer to $\gamma_{k+1,\,i}^{(2j-1)}$ and $\gamma_{k+1,\,i}^{(2j)}$ (possibly empty) as
{\it children} of $\gamma_{k,\,i}^{(j)}$.  By construction each mother has two 
children,  every child has a mother,  and they either overlap with each other, or they form a (non-degenerated) hyperbolic triangle by  Lemma~\ref{degenerated triangle}.
Moreover, $\gamma_{k,i}^{(j_1)}$ is called a {\it sibling} of 
$\gamma_{k,i}^{(j_2)}$ if they have the same mother. We say that $\gamma_{k,\,i}^{(j)}$ is in the {\it first generation} if $k=1$, that two curves $\gamma_{k_1,i_1}^{(j_1)},\,\gamma_{k_2,i_2}^{(j_2)}$ are   in the {\it same generation} if $k_1=k_2$, and that the generation of $\gamma_{k_1,i_1}^{(j_1)}$ is higher than that of $\gamma_{k_2,i_2}^{(j_2)}$ if $k_1<k_2$.

Notice that $\gamma_{k,\,i}^{(j)}$ may potentially  coincide with the 
$i$-th component of some other piecewise hyperbolic geodesic. 
The following lemma  describes the potential overlaps. 
\begin{lem}\label{continuous chain}
If $k_1<k_2$, then  $\gamma_{k_1,\,i}^{(j_1)}$ and $\gamma_{k_2,\,i}^{(j_2)}$ are not siblings. Furthermore, if we have that $\gamma_{k_1,\,i}^{(j_1)}$ coincides with $\gamma_{k_2,\,i}^{(j_2)}, k_1<k_2$ for some $i,\,j_1$ and $j_2$, then the mother of $\gamma_{k_2,\,i}^{(j_2)}$ also coincides with $\gamma_{k_1,\,i}^{(j_1)}$. 

Similarly if $\gamma_{k,\,i}^{(j_1)}$ and $\gamma_{k,\,i}^{(j_2)}$ coincide with each other but they are not siblings, then the mother of $\gamma_{k,\,i}^{(j_2)}$ also coincides with $\gamma_{k,\,i}^{(j_1)}$. 
\end{lem}
\begin{proof}
It is clear that $\gamma_{k_1,\,i}^{(j_1)}$ and $\gamma_{k_2,\,i}^{(j_2)}$ are not in the same generation since $k_1<k_2$. Thus they are not siblings. 

Let us show the second claim where $\gamma_{k_1,\,i}^{(j_1)}$ coincides with $\gamma_{k_2,\,i}^{(j_2)}, k_1<k_2$ for some $i,\,j_1$ and $j_2$. 
Recall that piecewise hyperbolic geodesics are obtained by taking  limits of hyperbolic geodesics in Jordan domains $\wz D_n$. Suppose that the end points of $\gamma_{k_1,\,i}^{(j_1)}$ are $x$ and $y$,  which are  the end points of $\gamma_{k_2,\,i}^{(j_2)}$ as well.  Then for each $l=1,\,2$ there exists an approaching sequence of  $\gamma_{n,\,l}\subset \wz D_n$ for $\gamma_{k_l,\,i}^{(j_l)}$ such that  $x_{n,\,l},\, y_{n,\,l}\in \gamma_{n,\,l}\cap\partial \wz D_n$ approximate $x$ and $y$, respectively. Let $\az_n\subset \wz D_n$ be the hyperbolic geodesic joining $x_{n,\,1}$ and $x_{n,\,2}$. Likewise $\beta_n \subset \wz D_n$  is the hyperbolic geodesic connecting  $y_{n,\,1}$ and $y_{n,\,2}$.

Recall that our choice of $\delta_1$ in Section~\ref{sec:phg} ensures the images under $\wz \varphi_n\colon \wz D_n\to \mathbb R^2\setminus \overline{\mathbb D}$ of the hyperbolic geodesics $\az_n$ and $\beta_n$ to be contained in $B(0,\,10)$. Then according to Corollary~\ref{GH thm} \begin{equation}\label{eqn 21}
\ell(\az_n) \ls |x_{n,\,1}-x_{n,\,2}|,\, \qquad \ell(\beta_n) \ls |y_{n,\,1}- y_{n,\,2}|
\end{equation}
for sufficiently large $n$, where the constants depend only on $C_1$. 
This allows us to take the limit of $\az_n$ and $\beta_n$ by Arzel\`a-Ascoli lemma, up to relabeling the sequences. Since $x_{n,\,l}\to x$ and $y_{n,\,l} \to y$ as $n\to \infty$ for $l=1,\,2$, we conclude from \eqref{eqn 21} that $\az_n \to x$ and $\beta_n\to y$. 

On the other hand, let  $\{\gamma'_{n,\,2}\}, \gamma'_{n,\,2}\subset \wz D_n$ be the sequence approaching the mother of $\gamma_{k_2,\,i}^{(j_2)}$. Observe that, according to the construction of the hyperbolic trangles in $\wz D_n$, the end points of $\gamma'_{n,\,2}$ are contained in either  the (closed) set enclosed by $\az_n$ and $\partial \wz D_n$, or the one enclosed by $\beta_n$ and $\partial \wz D_n$. Then via conformal mappings and  the geodesics in the complement of the unit disk,  we conclude that $\az_n$ and $\beta_n$ have non-trivial intersections with $\gamma'_{n,\,2}$. Thus the mother of $\gamma_{k_2,\,i}^{(j_2)}$ also goes through $x$ and $y$, and our conclusion follows from Lemma~\ref{parempi}. 

The second part of the lemma follows from a similar argument with notational changes. 
\end{proof}

Let $[\gamma_{k,\,i}^{(j)}]$ be the collection of all the $i$-th components of 
our piecewise hyperbolic geodesics which coincide with a non-empty curve
$\gamma_{k,\,i}^{(j)}$ in $\wz\Omega_i$ (see Lemma~\ref{lem_phg}) and order the 
elements in it as follows: $\gamma_{k_1,\,i}^{(j_1)}$ is older than 
$\gamma_{k_2,\,i}^{(j_2)}$ if and only if either $k_1< k_2$ or $k_1=k_2$ but 
$j_1< j_2$. We denote the oldest one in $[\gamma_{k,\,i}^{(j)}]$ by 
$\mathcal A[\gamma_{k,\,i}^{(j)}],$ and call it 
{\it the ancestor of $\gamma_{k,\,i}^{(j)}$}. So as to define an extension 
operator conveniently later, we force 
$\gamma_{1,\,i}^{(1)}\in [\gamma_{k,\,i}^{(j)}]$ if 
$\gamma_{1,\,i}^{(j)}\in [\gamma_{k,\,i}^{(j)}]$ for some $2\le j\le 2^{k_0+k+1}$; 
it is just an artificial curve which is only used in the next subsection, and 
we suppress the abuse of notation here.  

To clarify the definition of an ancestor, consider the curves $\gamma_{k,\,i}^{(j)}$, $\gamma_{k+1,\,i}^{(2j-1)}$ and 
$\gamma_{k+1,\,i}^{(2j)}$ enclosing a 
non-degenerated hyperbolic triangle $R_{k,\,i}^{(j)}$. Then Lemma~\ref{continuous chain} shows that
\begin{equation}\label{sons}
\mathcal A[\gamma_{k+1,\,i}^{(2j-1)}]= \gamma_{k+1,\,i}^{(2j-1)} \ \ \text{ and } \ \ \mathcal A[\gamma_{k+1,\,i}^{(2j)}]= \gamma_{k+1,\,i}^{(2j)}.
\end{equation}


\begin{lem}\label{finite overlaps}
Each generation in
$[\gamma_{k,\,i}^{(j)}]$
contains at most two distinct elements (expect the first generation which has at most three elements). Moreover if the highest (but not the first) generation of $[\gamma_{k,\,i}^{(j)}]$
contains  two distinct elements, then they are siblings. 
\end{lem}
\begin{proof}
We prove the first claim by contradiction. Suppose that $[\gamma_{k,\,i}^{(j)}]$
contains three distinct elements $\gamma_l,\,l=1,\,2,\,3$ from the same generation, and note that they coincide in the  component $\wz \Omega_i$ by definition. 
By joining each of the two end points of $\gamma_l$ via a hyperbolic geodesic  in $\Omega$ starting from the image of the origin under the conformal map $\varphi\colon \mathbb D \to \Omega$, we obtain a Jordan domain $\Omega_l$ via concatenating the two hyperbolic geodesics together with $\gamma_l$, for every $l=1,\,2,\,3$ correspondingly; recall here that $\gamma_l$ is injective by Lemma~\ref{lem_phg} for each $l=1,\,2,\,3$. 

We claim that these three domains are pairwise disjoint. Indeed recall that 
$$D_n=\varphi(B(0,\,1-2^{-n}))$$
 with $\varphi\colon \mathbb D \to \Omega$ conformal, and that each $\gamma_l$ is the limit of a uniformly converging sequence, consisting of  hyperbolic geodesics $\gamma_{l,\,n} \subset \wz D_n$, where $\wz D_n$ is the exterior of $D_n$. Then the end points $x_{l,\,n},\,y_{l,\,n}$ of $\gamma_{l,\,n}$ converge to the end points $x_l,\,y_l$ of $\gamma_l$, respectively. By the uniform continuity of $\varphi$ the hyperbolic geodesics $\varphi([0,\,\varphi^{-1}(x_{l,\,n})])$ joining $\varphi(0)$ to $x_{l,\,n}$ also converge uniformly to the hyperbolic geodesics joining $\varphi(0)$ to $x_l$ (in $D_n$); an analogous statement holds for $\varphi([0,\,\varphi^{-1}(y_{l,\,n})])$. Therefore $$\az_{l,\,n}:=\varphi([0,\,\varphi^{-1}(x_{l,\,n})])\cup \varphi([0,\,\varphi^{-1}(y_{l,\,n})])\cup \gamma_{l,\,n},\,n\in \mathbb N$$ is a sequence of Jordan curves converging to a Jordan curve 
 $$\az_l=\varphi([0,\,\varphi^{-1}(x_{l})])\cup \varphi([0,\,\varphi^{-1}(y_{l})])\cup \gamma_{l}$$
uniformly as $n\to \infty$ for $l=1,\,2,\,3$.

Note that for each $n\in \mathbb N$, the Jordan domains enclosed by $\az_{l,\,n}$'s are pairwise disjoint since $\gamma_l$ are in the same generation. Then the claim follows from the uniform convergence of $\az_{l,\,n}$. Indeed if two of $\Omega_1,\,\Omega_2,\,\Omega_3$ were to intersect, then there would exist a disk $B(z,\,3r)$ contained in the intersection, say, of $\Omega_1$ and $\Omega_2$. However, for each $l=1,\,2$, when $n$ is large enough,  every point in $\az_{l,\,n}$  has at most distance $r$ to $\az_l$. Since $\az_l$ and $\az_{l,\,n}$ are Jordan, we conclude by the triangle inequality that $B(z,\,r)$ is in the intersection of the Jordan domains given by $\az_{1,\,n}$ and $\az_{2,\,n}$. This is a contradiction. 

Next we show that, since $\az_l$'s coincide in $\wz \Omega_i$ for each $l=1,\,2,\,3$, at least two of the corresponding Jordan domains have non-trivial intersection; this gives the contradiction. In fact, by the Jordan-Schoenflies theorem (see e.g.\ \cite{T1992}) there is a homeomorphism 
$\phi\colon\mathbb R^2\to\mathbb R^2$ such that the curve $\gamma_l\cap \wz \Omega_i$ is mapped to the  interval $[-1,\,1]$ on the real line; since $\gamma_l,\,l=1,\,2,\,3$ coincide, our map $\phi$ does not depend on $l$. Notice that $\Omega_l$ are also mapped to Jordan domains. Then there exists a constant $0<r<1$ such that $B(0,\,r)\cap \phi(\Omega_l)$ is a half disk for each $l=1,\,2,\,3$; indeed choose
$$r=\min_{l=1,\,2,\,3} \dist\left(0,\,\phi(\partial \Omega_l \setminus (\gamma_l\cap \wz \Omega_i))\right).$$
This implies that at least two of the half disks have a non-trivial intersection and that the same holds for $\Omega_l$'s since $\phi$ is a homeomorphism. 
All in all we have shown the first part of the lemma.

The second part of the lemma follows from the second part of  Lemma~\ref{continuous chain}.
\end{proof}

Fix a curve $\gamma_{k,\,i}^{(j)}$. Recall that $[\gamma_{k,\,i}^{(j)}]$ is  the (equivalence) class of  (subarcs of) piecewise hyperbolic geodesics coinciding with $\gamma_{k,\,i}^{(j)}$. 
For each curve $\gamma\in [\gamma_{k,\,i}^{(j)}]$, we define the {\it family chain $\mathcal F(\gamma)$ of $\gamma$} as the finite ordered set  $\mathcal F(\gamma)=\{\gamma_1,\,\gamma_2 ,\, \dots, \gamma_l\},$ consisting of certain elements in $[\gamma_{k,\,i}^{(j)}]$ and
with the property that $\gamma_{n}$ is the mother or sibling of $\gamma_{n+1}$ for  
$1\le n\le l-1,$
or for some $1\le i\le N$, $\gamma_1=\gamma_{1,\,i}^{(1)}$,  $\gamma_2=\gamma_{1,\,i}^{(j)}, 2\le j\le 2^{k_0+k+1}$, and  $\gamma_{n}$ is the mother or sibling of $\gamma_{n+1}$ for  
$2\le n\le l-1.$
$\mathcal F(\gamma)$ is obtained
via the following procedure. Define $\hat \gamma_1=  \gamma.$
Set $\hat \gamma_2$  to be the mother of $\hat \gamma_1$ 
and continue inductively; see Lemma~\ref{continuous chain}. 
This procedure stops after a finite number of steps,
when we reach $\hat \gamma_{l-1}=\gamma_{k_0,i}^{j_0}$ such that the mother of $\gamma_{k_0,i}^{j_0}$ is not in  $[\gamma_{k,\,i}^{(j)}]$, or $k_0=1$; namely $\gamma_{k_0,i}^{j_0}$ is in the highest generation of $[\gamma_{k,\,i}^{(j)}]$ (by Lemma~\ref{continuous chain}). 
Then define $\hat \gamma_{l}=\mathcal A[\gamma_{k,\,i}^{(j)}]$; note that it is possible that $\hat \gamma_{l}=\hat \gamma_{l-1}$. 
We now set $\gamma_n=
\hat \gamma_{l-n+1},$ $1\le n \le l.$ Observe that for $n\ge 3$, $\gamma_n$ is a child of $\gamma_{n-1}$ by our construction. 

We say that a family chain $\mathcal F(\gamma)$ is {\it maximal} if it is maximal (with respect to set inclusion) amongst all family chains of elements in the curve family $[\gamma]$; recall that $\mathcal F(\gamma')\subset [\gamma]$ as a set for every $\gamma'\in [\gamma]$.  
Lemma~\ref{finite overlaps} together with Lemma~\ref{degenerated triangle} and Lemma~\ref{continuous chain} tells that, from each $[\gamma_{k,\,i}^{(j)}]$,  one can extract at most two different maximal family chains. 
Now we are ready to construct our extension operator. 
See Figure~\ref{fig:family chain}.

\begin{figure} 
 \centering
 \includegraphics[width=0.8 \textwidth]{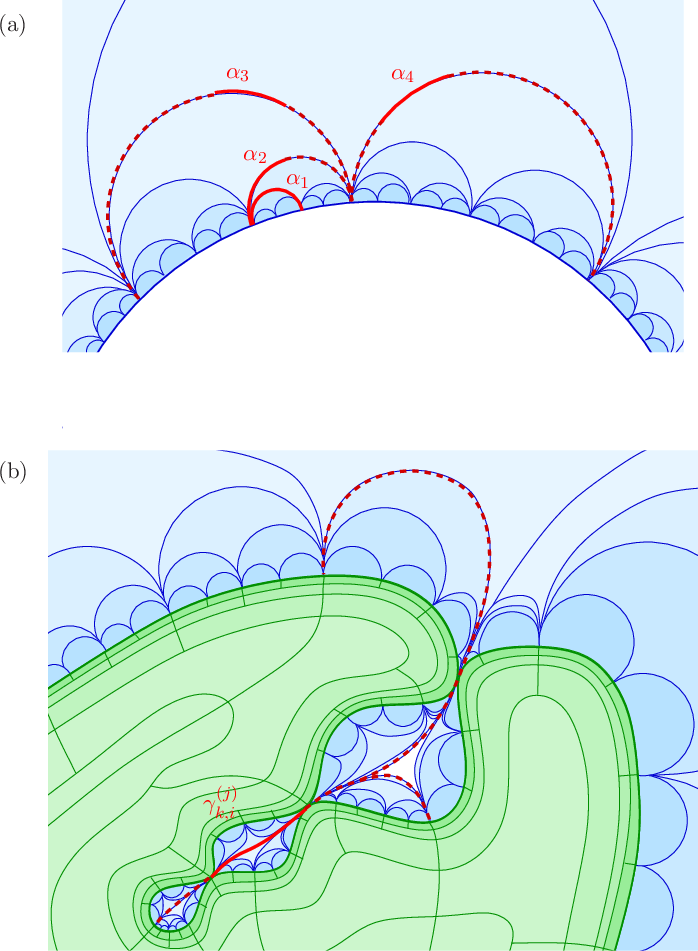}
 \caption{ A collection  of  hyperbolic geodesics (in the 
approximation to the piecewise hyperbolic geodesics in the picture (b)) related to the elements 
in $[\gamma_{k,i}^{(j)}]$ is illustrated in the figure (a). 
Here the collection $[\gamma_{k,i}^{(j)}]$ gives rise to two different maximal
family chains: One corresponds to (the limit of) $ \{\az_3,\,\az_3,\,\az_2,\,\az_1\}$ and the other corresponds to (the limit of) $\{\az_3,\,\az_4\}$.  }
 \label{fig:family chain}
\end{figure}

\subsection{Definition of the extension}

Let $u \in W^{1,\infty}(\Omega)$; then $u\in W^{1,\,1}(\Omega)$. As in the Jordan case, we define
\[
 a_k^{(j)}=\bint_{Q_k^{(j)}} u(x) \,dx.
\]
We still associate this value to $\gamma_k^{(j)}$ and write it as $a_{\gamma_k^{(j)}}$; we associate this value also to all the subarcs $\gamma_{k,\,i}^{(j)}$ of $\gamma_k^{(j)}$. However, one cannot define our extension directly via them and Lemma~\ref{axillary function} because of the degeneracy.

To overcome this, we use the terminology defined in the previous subsection. 
We define 
$Eu(x)=a_{\mathcal A[\gamma_{k,\,i}^{(j)}]}$ on $ \gamma_{k,\,i}^{(j)}$, where $a_{\mathcal A[\gamma_{k,\,i}^{(j)}]}$ is the value associated to $\mathcal A[\gamma_{k,\,i}^{(j)}]$. Then $Eu$ is 
well defined on these hyperbolic geodesics by the definition of 
$[\gamma_{k,\,i}^{(j)}]$.

Next we  associate to each non-degenerated 
$R_{k,\,i}^{(j)}$ in $\wz \Omega_i$
a function $\phi_{k}^{(j)}\in W^{1,\,1}(R_{k}^{(j)})$ so that (recall again  that there is at most one possible $i$ such that $R_{k,\,i}^{(j)}$ is non-degenerated by Lemma~\ref{degenerated triangle}. )
\[
\phi_{k}^{(j)}(x) = \begin{cases}
a_{\mathcal A[\gamma_{k,\,i}^{(j)}]}, \ & \ x\in \gamma_k^{(j)}\cap R_{k}^{(j)}\\
a_{\mathcal A[\gamma_{k+1,\,i}^{(2j-1)}]}, \ & \ x\in \gamma_{k+1}^{(2j-1)}\cap R_{k}^{(j)}\\
a_{\mathcal A[\gamma_{k+1,\,i}^{(2j)}]}, \ & \ x\in \gamma_{k+1}^{(2j)}\cap R_{k}^{(j)}
                    \end{cases}.
\]
We choose $\phi_k^{(j)}$ to be the function $\phi$ from Lemma~\ref{axillary function}  with the  boundary values given above. Then we  have the gradient estimate \eqref{norm of function} that will be employed later. 
Observe that since $R_{k,\,i}^{(j)}$  is non-degenerated, via \eqref{sons} we actually have 
$$\mathcal A[\gamma_{k+1,\,i}^{(2j-1)}]= \gamma_{k+1,\,i}^{(2j-1)} \ \ \text{ and } \ \ \mathcal A[\gamma_{k+1,\,i}^{(2j)}]= \gamma_{k+1,\,i}^{(2j)}.$$

On the rest of $\wz \Omega$ we define $Eu=a_{\gamma_{1}^{(1)}}$. This is consistent because we forced $\gamma_{1,\,i}^{(1)}\in [\gamma_{k,\,i}^{(j)}]$ if $\gamma_{1,\,i}^{(j)}\in [\gamma_{k,\,i}^{(j)}]$ for some $2\le j\le 2^{k_0+k+1}.$ 
To conclude, the operator $E$ is defined by 
\[
Eu(x) = \begin{cases}
u(x), \ & \ x\in \Omega\\
a_{\mathcal A[\gamma_{k,\,i}^{(j)}]}, \ & \ x \in \gamma_{k,\,i}^{(j)}\\
\phi_{k}^{(j)}(x), \ & \ x \text{ is in the interior of a non-degenerated } R_{k,\,i}^{j} \\
a_{\gamma_{1}^{(1)}}, \ & \ \text{in the rest of $\wz \Omega$}
                    \end{cases}.
\]

Note that currently in the simply connected case,  $Eu$ is defined as $a_{\gamma_{1}^{(1)}}$ for points far away from $\partial\Omega$, differently from the construction used in the Jordan case. This discrepancy is purely technical. Recall that for Jordan domains, we introduced additional hyperbolic triangles $R_{0}^{(j)}$ to cover a neighborhood of $\partial \Omega$ in $\wz \Omega$, and then defined $Eu$ as the average of $u$ over $\Omega$  for points away from $\partial\Omega$. However, when $\Omega$ is simply connected, controlling the limit of $R_{0}^{(j)}$  poses significant technical challenges. Consequently, we have slightly modified the argument, while the current variation remains valid for the Jordan case. We retain the original construction for Jordan domains only because the $R_{0}^{(j)}$-based definition is somewhat more straightforward; one does not need to handle overlaps there.

Observe that $Eu\in W^{1,\,1}_{loc}{(\wz \Omega)}.$  
Indeed by Lemma~\ref{axillary function} and  Lemma~\ref{degenerated triangle}, $Eu$ is locally Lipschitz continuous in $\wz \Omega$; the consistency of boundary values between (non-degenerated) hyperbolic triangles is guaranteed by the definition of $\phi_{k}^{(j)}(x)$. 

This time we estimate the Sobolev-norm of the extension over a non-degenerated hyperbolic triangle $R_{k,\,i}^{(j)} \subset \wz \Omega_i$
by inserting the intermediate values given to the corresponding curves in the family chain. Namely by Lemma~\ref{axillary function}, \eqref{length estimate} and \eqref{sons}, 
\begin{align*}
 \int_{R_{k,\,i}^{(j)}} |\nabla \phi_{k}^{(j)}(x)| \, dx & \lesssim 
 \bigg(|a_{\mathcal A[\gamma_{k,\,i}^{(j)}]}-a_{k+1}^{(2j-1)}|\ell(\gamma_{k,i}^{(j)})
 + |a_{\mathcal A[\gamma_{k,\,i}^{(j)}]}-a_{k+1}^{(2j)}|\ell(\gamma_{k,i}^{(j)})\\
 & \qquad+ |a_{k+1}^{(2j-1)}-a_{k+1}^{(2j)}|\min\{\ell(\gamma_{k+1,i}^{(2j-1)}),\ell(\gamma_{k+1,i}^{(2j)})\}\bigg)\\
  & \lesssim \left(|a_k^{(j)}-a_{k+1}^{(2j-1)}|+|a_k^{(j)}-a_{k+1}^{(2j)}| +|a_{k+1}^{(2j-1)}-a_{k+1}^{(2j)}|\right) \diam (Q_k^{(j)})\\
 & \qquad+ \sum_{\gamma_l\in \mathcal F(\gamma_{k,\,i}^{(j)})} |a_{\gamma_l} - a_{\gamma_{l+1}}|\ell(\mathcal A[\gamma_{k,i}^{(j)}])\\
 & = S_k^{(j)}(1) + S_k^{(j)}(2),
\end{align*}
where we used the fact $\ell(\mathcal A[\gamma_{k,i}^{(j)}])=\ell(\gamma_{k,i}^{(j)})$.  Here the constants depend only on $C_1$ in \eqref{eq:curvecondition}.

The first term $S_k^{(j)}(1)$ is the same as in the Jordan case, and so we know 
that, after summing over all the hyperbolic triangles, it can be controlled by 
$\|\nabla u\|_{L^{1}(\Omega)}$ up to a multiplicative constant. Notice that here each $Q_k^{(j)}$ corresponds to at most one non-degenerated $R_{k,\,i}^{(j)}$ by Lemma~\ref{degenerated triangle}. 

Observe that
\begin{equation}\label{equ 61}
 \sum_{i=1}^N \ell(\mathcal A[\gamma_{k,i}^{(j)}]) =\sum_{i=1}^N \ell(\gamma_{k,i}^{(j)}) \le C(C_1) \diam (Q_k^{(j)}), 
\end{equation}
by \eqref{length estimate}.

Fix $i$ and two distinct non-degenerated triangles $R_{k,\,i}^{(j)},\,R_{k',\,i}^{(j')}.$ Then  
$\mathcal F(\gamma_{k,\,i}^{(j)})$ is disjoint from 
$\mathcal F(\gamma_{k',\,i}^{(j')})$ unless $[\gamma_{k,\,i}^{(j)}]=[\gamma_{k',\,i}^{(j')}]$, or equivalently $R_{k,\,i}^{(j)}$ and $R_{k',\,i}^{(j')}$ share the same boundary curve. 
Recall that in each  $[\gamma_{k,\,i}^{(j)}]$ there are at most two different family chains starting from the highest generation to the lowest one, and hence the ancestor in each $[\gamma_{k,\,i}^{(j)}]$ is counted at most twice if $k\neq 1$; if $k=1$ the multiplicity of $\gamma_{1,\,i}^{(1)}$ is  at most $2^{k_0+3}$. 
Hence, as a result of changing the order of summation together with \eqref{equ 61} and \eqref{equ 30} the following estimate 
holds: 
\begin{align*}
 &\sum_{R_{k,\,i}^{(j)}} S_k^{(j)}(2) = \sum_{R_{k,\,i}^{(j)}} \sum_{\gamma_l\in \mathcal F(\gamma_{k,\,i}^{(j)})} |a_{\gamma_l} - a_{\gamma_{l+1}}|\ell(\mathcal A[\gamma_{k,i}^{(j)}])\\
 \lesssim& \sum_{k,j} (|a_k^{(j)} - a_{k+1}^{(2j-1)}|+|a_k^{(j)} - a_{k+1}^{(2j-1)}|+|a_k^{(j)} - a_{k}^{(j-1)}|+|a_k^{(j)} - a_{k}^{(j+1)}|) \sum_{i=1}^N \ell(\mathcal A[\gamma_{k,i}^{(j)}])\\
\lesssim &\sum_{k,j}(|a_k^{(j)} - a_{k+1}^{(2j-1)}|+|a_k^{(j)} - a_{k+1}^{(2j-1)}|+|a_k^{(j)} - a_{k}^{(j-1)}|+|a_k^{(j)} - a_{k}^{(j+1)}|)\diam (Q_k^{(j)})\\
\ls & \sum_{k,\,j} S^{(j)}_{k}(1),
\end{align*}
where the constants depend only on $C_1$. 
Therefore the estimate for $\sum_{k,\,j} S^{(j)}_{k}(1)$ gives
$$\|\nabla Eu\|_{L^{1}(\widetilde\Omega)} \le C(C_1) \|\nabla u\|_{L^{1}(\Omega)}. $$

Finally let $B$ be a disk of radius $4\diam(\Omega)$ such that  $\Omega$ is  contained in $\frac 1 2 B$. Similarly as in the Jordan case, by part (4) of Lemma~\ref{property of John domain}, we have that
\begin{align}\label{L1 control simply}
\| Eu\|_{L^{1}(B)}\ls & \sum_{R_{k,\,i}^{(j)}}\int_{R_{k,\,i}^{(j)}} |Eu(x)| \, dx + \int_{\Omega} |u(x)|\, dx \nonumber \\
\ls & \sum_{\substack{k\ge 1\\ 1\le j\le 2^{k_0+k+1}}}\int_{Q_k^{(j)}} |u(x)| \, dx  +\int_{\Omega} |u(x)|\, dx\ls \int_{\Omega} |u(x)|\, dx,
\end{align}
where we used \eqref{bound of function}, \eqref{equ 30} and the fact that
$$|R_k^{(j)}|\ls |Q_k^{(j)}| $$
coming from \eqref{distance estimate} and the fact that $Q_k^{(j)}$ is of $C(C_1)$-Whitney-type. Here all the constants depend only on $C_1$. 

\subsection{Continuity at one-sided points}

By the curve condition  \eqref{eq:curvecondition}  and Corollary \ref{cor:twosidedzeromeasure}, 
the set $T \subset \partial\Omega$ of cut-points  
has $\ch^1$-measure zero. Therefore if we can extend $u$ continuously to the 
one-sided points, then we have defined $Eu$ everywhere except for a set of 
$\ch^1$-measure zero. 

Let us first show that we may assume $u$ to be continuous at the one-sided points in the following sense:

\begin{lem}\label{continuous extension}
Let $\Omega$ be a bounded simply connected $J$-John domain, 
and suppose that
$u\in W^{1,\,\infty}(\Omega)\cap C(\Omega)$. Then we can extend $u$ (uniquely) to $
\wz u$ so that 
$\wz u\in C(\overline{\Omega}\setminus T)$, 
where $T$ is the collection of cut-points of  $\partial \Omega.$
\end{lem}

\begin{proof}
Let $z\in \partial{\Omega}\setminus T.$ 
Since 
$u\in   C(\Omega),$ it suffices to show that $\lim_{i\to \infty} u(z_i)$ exists   whenever $z_i\in \Omega$ satisfies $|z_i -z |\to 0$ as $i\to \infty.$ Let 
$\varphi\colon \overline{\mathbb D}\to \overline{\Omega}$ be continuous 
and conformal inside $\mathbb D;$ see part (2) of Lemma~\ref{property of John domain}. 
By Lemma~\ref{inner continuous}, we also have that $\varphi\colon (\mathbb D,\,|\cdot|)\to (\Omega,\,\dist_{\Omega})$ is uniformly continuous.

Since $z\notin T,$ we have that 
$\varphi^{-1}(z)$ is a singleton; say $\varphi^{-1}(z)=w.$
Our claim follows if $|\varphi^{-1}(z_i)-w|\to 0$ when $i\to \infty;$ indeed since $\varphi$ is uniformly continuous with respect to the $\dist_{\Omega}$-metric of $\Omega$, then by the assumption that $u\in W^{1,\,\infty}(\Omega)\cap C(\Omega)$ we conclude that $\lim_{i\to \infty} u(z_i)$ exists. 
This is 
necessarily the case; otherwise a subsequence converges to some $w'\neq w$
and then $\varphi(w')=z$ by the continuity of $\varphi$ at $w',$
contradicting the fact that $\varphi^{-1}(z)$ is a singleton.  
\end{proof}

%

Recall that by the definition of the ancestors and that of $\phi_{k}^{(j)}(x)$, $Eu$ is (locally Lipschitz) continuous in $\wz \Omega$. 
Then the continuity of $Eu$ (under $u\in W^{1,\,\infty}(\Omega)\cap C(\Omega)$) restricted to  $\mathbb R^2\setminus H$ with 
$H := T \cup Z$ follows exactly as in the Jordan case, where $Z$ is also 
defined like the previous one: the collection of the countably many 
points $z_{k}^{(j)}$ on the boundary. 
This time the fact that $y \notin \Gamma_s$ (defined in \eqref{curve family}) 
comes from Lemma \ref{lma:twosidedhyperbolic}, where 
we observed that, except for the endpoints, piecewise hyperbolic geodesics can
meet the boundary only at cut-points. 
Therefore, by further defining $Eu(x)=\wz u(x)$ for $x\in \partial \Omega$ for $u\in W^{1,\,\infty}(\Omega)\cap C(\Omega)$, 
we have that $Eu(x)$ is continuous in $\mathbb R^2\setminus H$.

\subsection{Absolute continuity}

We show the absolute continuity of $Eu$ 
along almost every horizontal or vertical line under the assumption that $\wz u\in  C(\Omega) \cap W^{1,\,\infty}(\Omega)$. By Lemma~\ref{continuous extension} we may assume that 
$$u\in   W^{1,\,\infty}(\Omega)\cap C(\overline{\Omega}\setminus T).$$ 
We begin by applying the Jordan case in the following way.

Since $\ch^1(T) = 0$ by Corollary~\ref{cor:twosidedzeromeasure},
for any $\epsilon >0,$ we find 
$$U_\epsilon :=\Omega \cap  \bigcup_{z \in T}  {B(z,r_z)}$$
 with the property that the $\ch^1$-measure 
of its vertical and horizontal projections are less than $\epsilon$; the existence of such disks follows from the definition of $\ch^1$-measure and the fact that projections are $1$-Lipschitz.

Since $\varphi$ is a homeomorphism in $\mathbb D$, the set 
$\varphi^{-1}(U_{\ez})$ is open. 
For each $z \in T$ and every $w\in \varphi^{-1}(\{z\})$, by the continuity of $\varphi$ we may pick a
small enough $\delta_w \in (0,\,\frac 1 2)$ 
such that  the set
\[
 V_w := \{v \in \mathbb{D}  \colon  \langle w, v \rangle > (1-\delta_w)\}
\]
satisfies
\begin{equation}\label{equ 1001}
\dist(\varphi(V_w),\,z)\le \frac 1 2 r_z.
\end{equation}
where $\langle w, v \rangle $ means the inner product of $w$ and $v$. See Figure~\ref{fig:vz}. 

\begin{figure} 
 \centering
 \includegraphics[width=0.85\textwidth]{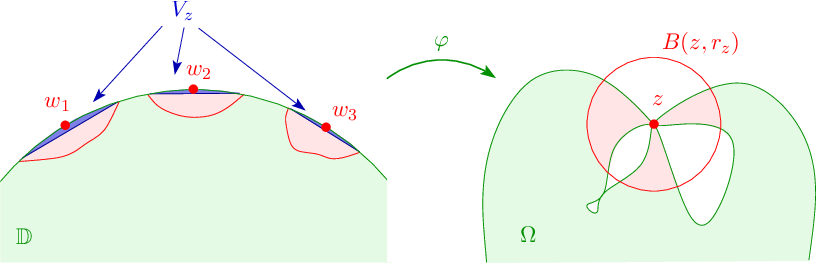}
 \caption{Here we show in red the disk $B(z,\,r_z)$  and its preimage under $\varphi$. The set $V_z$, whose boundaries are line segments inside the unit disk, is shown  with blue color. }
 \label{fig:vz}
\end{figure}

Finally we define 
\[
 \Omega_\epsilon :=\varphi\left(\mathbb D \setminus \overline{\bigcup_{w \in \varphi^{-1}(T)}V_w}\right).
\]
We claim that $\Omega\setminus \Omega_{\ez} \subset  U_{\ez}$. If this were to hold, then  the $\ch^1$-measure of the projection of $\Omega\setminus \Omega_{\ez}$ on each coordinate axis would be
less than 
$\ez$. 

Let us verify the claim. Since $\varphi$ is a homeomorphism in $\mathbb D$ and continuous up to the boundary, we have
 $$ \Omega\setminus \Omega_\epsilon =  \varphi\left(\overline{\bigcup_{w \in \varphi^{-1}(T)}V_w}\right)\cap \Omega= \overline{\bigcup_{w \in \varphi^{-1}(T)} \varphi(V_w)}\cap \Omega.$$
Fix $y\in \Omega\setminus \Omega_{\ez}$ and let $y_n\in \bigcup_{w \in \varphi^{-1}(T)} \varphi(V_w)$ satisfy $y_n\to y$. Let $\delta:=\dist(y,\,\partial \Omega)$. Then $\delta>0$. 
For $n$ large enough, we have $|y_n-y|\le \delta/4$, and hence 
$$\dist(y_n,\,\partial \Omega)\ge  \dist(y,\,\partial \Omega)-|y_n-y|\ge \frac 3 4 \delta.$$
Moreover by \eqref{equ 1001} there exists $z_n\in T$ such that $|z_n-y_n|\le  r_{z_n}/2$, and then we obtain
$$\delta\le \frac 2 3 r_{z_n}.$$
Therefore
$$|y-z_n|\le |y_n-y|+|y_n-z_n|\le  \frac 1 4 \delta + \frac 1 2 r_{z_n}\le \frac 2 3  r_{z_n},$$
and then $y\in B(z_n,\,r_{z_n})\cap \Omega\subset U_{\ez}.$ Thus the claim follows. 




Since $\varphi$ is a conformal map and continuous up to the boundary, then $\varphi|_{\partial  (\varphi^{-1}(\Omega_\epsilon))}$ is continuous and injective as we took out neighborhoods of all the cut-points. Then by the compactness of $\partial  (\varphi^{-1}(\Omega_\epsilon))$ we conclude that $\varphi|_{\partial ( \varphi^{-1}(\Omega_\epsilon))}$ is an embedding to $\mathbb R^2$ and hence
$\Omega_\epsilon$ is Jordan. 
We claim that $\mathbb R^2\setminus \overline{\Omega}_\ez$ is  quasiconvex.  If this claim were to hold, by the result on the $W^{1,\,1}$-extension in the Jordan case (in Section~\ref{sec:sufficiency_Jordan}) we could conclude that there exists an extension 
$u_\epsilon \in W^{1,\,1}(\rr^2)$ of $u|_{\Omega_\epsilon}$.

Indeed, the proof of the claim is similar to the one of Lemma~\ref{inherited}. If $x,\,y\in \Omega^c$, then \eqref{eq:curvecondition} gives us the desired curve. If $x\in \Omega^c$
 while $y\in \Omega\setminus \Omega_\ez$, then $|x-y|\ge \dist(y,\,\partial \Omega)$. Hence by connecting $y$ to $y_0\in \partial \Omega$ such that  the line segment $[\varphi^{-1}(y),\,\varphi^{-1}(y_0)]$ joining $\varphi^{-1}(y)$ and $\varphi^{-1}(y_0)$ lies in the line segment $[0,\,\varphi^{-1}(y_0)]$, \eqref{eq:curvecondition} and part (1) and (3) of Lemma~\ref{property of John domain}  give us the desired curve; note that $[\varphi^{-1}(y),\,\varphi^{-1}(y_0)]$ is contained in the union of $V_z$'s by definition. Thus by symmetry we may assume that $x,\,y\in \Omega\setminus \Omega_\ez$.
 
We claim that, there exists an absolute constant $C$ such that, whenever 
\begin{equation}\label{C}
|x-y|\le C \min\{\dist(x,\,\partial \Omega),\,\dist(y,\,\partial \Omega)\},
\end{equation}
we have 
\begin{equation}\label{distance to the boundary}
|\varphi^{-1}(x)-\varphi^{-1}(y)|\le \frac 1 {16} \min\{\dist(\varphi^{-1}(x),\,\partial \mathbb D),\,\dist(\varphi^{-1}(y),\,\partial \mathbb D)\}. 
\end{equation}
To show the existence of $C$, we apply Lemma~\ref{whitney preserving} to the Whitney-type disk 
\[
B_1=B(\varphi^{-1}(x),\, 1 / {16} \dist(\varphi^{-1}(x),\,\partial \mathbb D)),
\]
which implies that
$\varphi(B_1)$ is of Whitney-type. Thus there exists an absolute constant $c$ such that 
$$B(x,\,c\dist(x,\,\partial \Omega))\subset \varphi(B_1).$$
By  choosing $C=\frac 1 2 c$, we obtain the claim. 

Suppose that $x,\,y$ satisfy 
$$|x-y|\ge C \min\{\dist(x,\,\partial \Omega),\,\dist(y,\,\partial \Omega)\}$$ 
where $C$  is the one in \eqref{C}. Then we can also define $x_0\in \partial \Omega$ in a similar way, connect $x_0$ with $x$ via a hyperbolic geodesic in $\Omega$ and join $x_0$ and $y_0$ in $\Omega^c$ by a curve with length controlled by $C(C_1)|x_0-y_0|$. This gives us a desired curve by \eqref{eq:curvecondition} and Lemma~\ref{property of John domain} again. Thus we only need to consider the case where 
$$|x-y|\le C \min\{\dist(x,\,\partial \Omega),\,\dist(y,\,\partial \Omega)\}. $$

Let $B=B(\varphi^{-1}(x),\, 4 |\varphi^{-1}(x)-\varphi^{-1}(y)|)$. Then it is a disk contained in a 
$4$-Whitney-type set because of \eqref{distance to the boundary}. 
We claim that  there exists a curve $\gamma\subset B\setminus\varphi^{-1}(\Omega_\ez)$ joining $\varphi^{-1}(x),\,\varphi^{-1}(y)$ in $\mathbb D$ such that its length is controlled by $|\varphi^{-1}(x)-\varphi^{-1}(y)|$ up to a multiplicative constant. If this holds,  by Lemma~\ref{whitney preserving} the curve $\varphi(\gamma)$ joining $x,\,y$ in $\varphi(B)\cap (\Omega\setminus\Omega_\ez)$ has length smaller than $|x-y|$ up to a multiplicative constant. Consequently $\mathbb R^2\setminus\Omega_\ez$ is quasiconvex with the constant independent of $\ez$. 

Now let us show the existence of $\gamma$ under the assumption \eqref{distance to the boundary}. Observe first that by  definition  
$$\varphi^{-1}(\Omega_\epsilon)=  \mathbb D \setminus \overline{\bigcup_{z \in T}V_z} = \left(\bigcap_{z\in T} \left( \mathbb D \setminus V_z \right)\right)^o$$
is a convex domain since it is the interior of the intersection of convex sets and the connectedness follows directly from the definition. Since we always require that $\delta_w \in (0,\,\frac 1 2 )$, then $B(0,\,\frac 1 2)\subset \varphi^{-1}(\Omega_\ez)$. This combined with convexity yields that for every $z\in \varphi^{-1}(\overline{\Omega}_\ez)$, locally there is a Euclidean (open) sector contained in $\varphi^{-1}(\Omega_\ez)$ with angle at least  $\frac \pi 3$, centered at $z$, towards the origin. 
Thus if we cannot join $\varphi^{-1}(x),\,\varphi^{-1}(y)$ in $\frac 3 4 B\setminus\varphi^{-1}(\Omega_\ez)$, then $\frac 3 4 B\setminus\varphi^{-1}(\Omega_\ez)$ has at least two components. However  this is impossible since $\varphi^{-1}(x),\,\varphi^{-1}(y) \in \frac 1 4 B$ but both $\varphi^{-1}(x)$ and $\varphi^{-1}(y)$ are not in the convex set $\varphi^{-1}(\Omega_\ez)$; planar geometry tells that such a sector with angle at least  $\frac \pi 3$ cannot exist. Therefore we can join $\varphi^{-1}(x),\,\varphi^{-1}(y)$ in $\frac 3 4 B\setminus\varphi^{-1}(\Omega_\ez)$ by a curve $\gamma'$ with $\ell(\gamma')\le (8+\frac 3 2  \pi)|\varphi^{-1}(x)-\varphi^{-1}(y)|$. Thus by letting $\gamma=\varphi(\gamma')$, we conclude  from Lemma~\ref{whitney preserving} and a change of variable that $\ell(\gamma)\ls |x-y|$ with an absolute constant, as desired.

Notice that $Eu$ is absolutely continuous along almost 
every closed line segment  parallel to the coordinate axes and contained in $\wz \Omega$ since $Eu$ is locally Lipschitz in $\wz\Omega$ by our construction. Recall that by
Lemma~\ref{continuous extension} we may assume that 
$u\in C(\overline{\Omega}\setminus T)\cap W^{1,\,\infty}(\Omega)$. 
Also take a representative of $u_{1/n}\in W^{1,\,1}(\mathbb R^2)$, the extension of the restriction of $u$ to the Jordan domain $\Omega_{1/n}$ given by Theorem~\ref{Jordan suff}, so that it is absolutely continuous (and hence continuous)
along  almost every line segment parallel to the coordinate axes.

By  Fubini's theorem, 
\begin{equation}\label{equ 301}
\int_{L} |\nabla Eu|\, dx <\infty, \qquad \int_{L} |\nabla u_{1/n}|\, dx <\infty
\end{equation}
for $\ch^1$-almost every line segment $L$ parallel to the coordinate axes. It follows that for almost every horizontal (or vertical) line $S$, \eqref{equ 301} holds, each $u_{1/n}$ is absolutely continuous on every closed line segment of $S$, and so is also $Eu$ on each of these closed line segments of $S$ that is additionally contained in $\wz \Omega$.

Define 
$$Z_0=\cap_{n\ge 1} \left(\Omega\setminus \Omega_{1/ n}\right).$$
Since the vertical and horizontal projections of  $\Omega\setminus \Omega_{1/n}$ have $\ch^1$-measure no more than $1/ n$, we conclude that almost every line $S$ parallel to the coordinate axes is disjoint from $Z_0$. Equivalently, $S\cap \Omega\subset \Omega_{1/ n}$ for some $n\in \mathbb N$. Furthermore, recalling that $H=T\cup Z$ has $\ch^1$-measure zero, we have that almost every $S$ parallel to  the coordinate axes is disjoint from $H$. By Lemma~\ref{continuous extension}, we may assume that $u$ is continuous on $S\cap \overline{\Omega}$ for each such $S$. 


Fix a horizontal line segment $S$ that satisfies the conclusions of the preceding two paragraphs. Let us  verify that $Eu|_{I}$ is absolutely continuous, where $I\subset S$ is a closed line segment.  
By the definition of $S$, there exists $M\in \mathbb N$ such that $Eu|_{\overline{\Omega}\cap S}=u_{1/M}|_{\overline{\Omega}\cap S}$, and $Eu_{1/M}$ is continuous on $S$; notice that by Lemma~\ref{continuous extension} and our assumption, $Eu(x)=u_{1/M}(x)$ for $x\in \partial \Omega\cap S$. Let $z_1,\,\dots,\,z_{2n}\in I$ be points whose 
first components are strictly increasing.

If $z_1,\,z_2\in \overline{\Omega}$, then by the assumption that $u\in C(\overline{\Omega}\setminus T)\cap W^{1,\,\infty}(\Omega)$, we get
\[
|Eu(z_1)-Eu(z_2)| = |u_{1/M}(z_1)-u_{1/M}(z_2)| \le \int_{[z_1,\,z_2]} |\nabla u_{1/M}|\,dx. 
\]

If $[z_1,\,z_2] \subset \wz \Omega$, then since $Eu$ is absolutely continuous on
closed line segments in $S$, we obtain 
\[
|Eu(z_1)-Eu(z_2)|\le \int_{[z_1,\,z_2]} |\nabla Eu|\,dx. 
\]
For the remaining case, by symmetry we may assume that $z_1\in \overline{\Omega}$ while $z_2\in \wz \Omega$. 
Then let $z$ be the nearest point to $z_2$ on $\overline{\Omega}\cap I$ between $z_1$ and $z_2$, and 
observe by the same absolute continuity properties as above that
\begin{align*}
|Eu(z_1)-Eu(z_2)| & \le |Eu(z_1)-Eu(z)|+|Eu(z)-Eu(z_2)| \\
&= |u_{1/M}(z_1) - u_{1/M}(z)| + |Eu(z)-Eu(z_2)|\\
& \le \int_{[z_1,\,z]} |\nabla u_{1/M}|\,dx+\int_{[z,\,z_2]} |\nabla Eu|\, dx,
\end{align*}
where in the last inequality we applied the continuity of $Eu$ on $S$. 

Using a similar argument for the other pairs, we get
\begin{align*}
 \sum_{i=1}^{n} |Eu(z_{2i-1})-Eu(z_{2i})| & \le \sum_{i=1}^{n} 2 \int_{[z_{2i-1},\,z_{2i}]} (|\nabla u_{1/M}| + |\nabla Eu|)\, dx\\
 & = 2 \int_{\cup_i [z_{2i-1},\,z_{2i}]} (|\nabla u_{1/M}| + |\nabla Eu|)\, dx
\end{align*}
from which we get the desired absolute continuity of $Eu$ on $I$ by the absolute continuity of the integral. 
The case of vertical lines is dealt with analogously, 
 and we conclude the absolute continuity of $Eu$ along almost every vertical and 
horizontal line.

\subsection{Conclusion of the proof}

Combining the discussions in all the subsections above, we have shown that  there exists a linear extension operator 
$$E\colon W^{1,\,1}(\Omega)\cap C(\Omega)\cap W^{1,\,\infty}(\Omega) \to W^{1,\,1}(B)$$
such that 
$$\|\nabla Eu\|_{L^{1}(B)} \lesssim \|\nabla u\|_{L^{1}(\Omega)}$$
and
$$\| Eu\|_{L^{1}(B)}\ls \|u\|_{L^{1}(\Omega)},$$
where the constants depend only on the constant in \eqref{eq:curvecondition}; recall that the boundary of $\Omega$   has  Lebesgue measure zero by part (4) of Lemma~\ref{property of John domain}. 
Here $B$ is a disk of radius $4\diam(\Omega)$ such that  $\Omega$ is  contained in $\frac 1 2 B$. By the fact that $B$ is an $W^{1,\,1}$-extension domain, and the fact that  $C(\Omega)\cap W^{1,\,\infty}(\Omega)$ is dense in $W^{1,\,1}(\Omega)$ \cite{KZ2015}, we conclude the sufficiency of the curve condition \eqref{eq:curvecondition}. 

\section{Necessity for simply connected domains}\label{sec:necessity}

In this section we prove the necessity of the curve condition \eqref{eq:curvecondition} for  bounded simply connected
$W^{1,\,1}$-extension domains. 
We use the following result stated in \cite[Corollary 1.2]{kosmirsha10}, which was shown via results in \cite{BM1967}.

\begin{thm}[{\cite[Corollary 1.2]{kosmirsha10}}]\label{thm:quasiconvex}
 Let $\Omega \subset \rr^2$ be a bounded simply connected $W^{1,\,1}$-extension 
domain. Then
 $\Omega^c$ is $C_1$-quasiconvex with $C_1$ depending only on the norm of the extension operator (with respect to the homogeneous norm).
\end{thm}

Let $\varphi \colon \mathbb{D} \to \Omega$ be a conformal map. 
By Theorem \ref{thm:quasiconvex}
the complement of $\Omega$ is quasiconvex for some constant $C_1.$ 
Therefore $\Omega$ is $J(C_1)$-John  by part (1) of Lemma~\ref{property of John domain}  and also resulting from part (2) of
Lemma~\ref{property of John domain} the map $\varphi$ extends as a
continuous map (still denoted  $\varphi$) to the boundary $\partial \mathbb D$.
As before, let us denote $\widetilde\Omega := \rr^2 \setminus\overline{\Omega}$
and let $\{\widetilde\Omega_i\}_{i=1}^N$ be an enumeration of the connected components of 
$\widetilde\Omega$ with $N\in \mathbb N \cup\{+\infty\}$.

We show that \eqref{eq:curvecondition}  is necessary. Fix $x,\,y\in \Omega^c$. If
$$|x-y|\le \frac 1 2 \max\{\dist(x,\,\partial \Omega),\,\dist(y,\,\partial \Omega)\},$$
then we may take $\gamma$ to be the line segment joining $x$ and $y$, and then \eqref{eq:curvecondition} holds for this $\gamma$. Thus we may assume that
$$|x-y|\ge \frac 1 2 \max\{\dist(x,\,\partial \Omega),\,\dist(y,\,\partial \Omega)\}. $$

By Theorem~\ref{thm:quasiconvex} there exists a curve $\gamma'\subset \Omega^c$ joining $x,\,y$ with length controlled by $C_1|x-y|$. If $\gamma'$ does not touch the boundary, then we are done. 
Otherwise starting from $z_0=x$, we choose $z_1$ such that 
$$\ell(\gamma'[z_0,\,z_1])=\delta \diam(\Omega) $$
for the subcurve $\gamma'[z_0,\,z_1]$ of $\gamma'$ joining $z_0$ to $z_1$, with $\delta$ in Lemma~\ref{lem_phg}; if there is no such a point, we let $z_1=y$ and define $\gamma_1$ to be the piecewise hyperbolic geodesic joining $z_0$ and $z_1$. If $\gamma'[z_0,\,z_1]$ touches the boundary at most once, then we define $\gamma_1=\gamma'[z_0,\,z_1]$. If not, choose the first and last points $x_1,\,y_1$ of $\gamma'[z_0,\,z_1]\cap \partial \Omega$ according to the parametrization of $\gamma'$, and reroute $\gamma'[z_0,\,z_1]$ via the piecewise hyperbolic geodesic joining $x_1$ and $y_1$ to obtain a new curve $\gamma_1$; the existence of the piecewise hyperbolic geodesic follows from Lemma~\ref{lem_phg}. Also according to Lemma~\ref{lem_phg}, 
$$\ell(\gamma_1)\ls \ell(\gamma'[z_0,\,z_1])$$
with the constant depending only on $C_1$. 
Continue this procedure via replacing $z_0$ by $z_1$ to find $z_2$ and to define $\gamma_2$, and iterate until we finally get the point $y$. Observe that this process stops after finitely many steps since $\gamma'$ has finite length, and it produces a finite sequence of curves $\gamma_i,\,1\le i\le N$ such that
$$\sum_{i=1}^{N} \ell(\gamma_i)\ls \sum_{i=1}^{N} \ell(\gamma'[z_{i-1},\,z_i])\ls  \ell(\gamma')\ls |x-y|.$$
We claim that $\gamma=\cup_{i=1}^{N} \gamma_i$ satisfies \eqref{eq:curvecondition} (with the constant depending only on $C_1$). 
By our construction  and the injectivity of piecewise hyperbolic geodesics from Lemma~\ref{lem_phg}, it suffices to show that for $x,\,y\in \partial \Omega$ with $|x-y|\le \delta \diam(\Omega)$ we have 
$$\ch^1(I)=0$$
where $I$ is the intersection between $\partial \Omega$ and the piecewise hyperbolic geodesic $\Gamma$ joining $x$ and $y$.


Let $\Gamma$ be a piecewise 
hyperbolic geodesic joining $x,\,y\in \partial \Omega$ in $\Omega^c$ with
$$|x-y|\le \delta \diam(\Omega).$$
By Lemma~\ref{two side cut point}  any $z \in I\setminus\{x,y\}$ is a cut-point, and every 
neighborhood of $z$ intersects both $\Omega_0$ and $\Omega_1$, where $\Omega_1$ and $\Omega_2$ are the two components of $\Omega\setminus(\varphi([v_1,\,0]\cup[0,\,v_2]) \cup \Gamma)$ with $x=\varphi(v_1)$ and $y=\varphi(v_2)$. Write $\az= [v_1,\,0]\cup[0,\,v_2]$.  



Suppose that $\ch^1(I)>0$. 
Let us next construct a test function that will give us a contradiction.
Pick $z_0 \in I \setminus\{x,y\}$ with 
\[
\ch^1(I\cap B(z_0,r))>0\qquad \text{for all }r>0;
\]
this is possible by the assumption that $\ch^1(I)>0$; see 
\cite[Theorem 2, Page 72]{EG1992}. Since $\varphi(\alpha)$ is a compact set, 
\[
 r_0 := \frac1 8\dist(z,\varphi(\az))>0.
\]
We define $\psi $ by setting
\[
 \psi(z) := \max\left\{0,1-\frac{1}{4 r_0}\dist(z,B(z_0,r_0))\right\}\chi_{\Omega_1}(z).
\]
By our definition of $r_0$, we have $\psi \in W^{1,\,1}(\Omega)$; notice that $\partial \Omega_1 \cap \Omega\subset \varphi(\az)$, where $\chi_{\Omega_1}$ is the characteristic function of $\Omega_1$.

Suppose that there exists an extension $E\psi\in W^{1,\,1}(\rr^2)$ of $\psi$. By   Lemma \ref{two side cut point} and the John property of $\Omega$, for every $z \in K := I \cap B(z_0,r_0)$ and every $0 < r_z <\delta <r_0$ with any $\delta>0$, for each $i =0,\,1$, there is a point
$z^{(i)} \in B(z,\,r_z)$ such that 
$$B(z^{(i)},\,cr_z)\subset B(z,\,r_z) \subset B(z_0,\, 2r_0)\cap \Omega_i \subset \Omega_i, $$
where the constant $c$ depends only on the John constant $J$.

Therefore, by the Poincar\'e inequality,
\begin{equation}\label{eq:gradestim}
 1 \ls  |E\psi_{B(z^{(0)},\,cr_z)}- E\psi_{B(z^{(1)},\,cr_z)}|
 \ls r_z\bint_{B(z,\,r_z)}|\nabla E\psi(w)|\,dw
\end{equation} for every $z \in K$, where $E\psi_{B}$ denotes the integral average of $E\psi$ on $B$.

Notice that $K\subset \bigcup_{z\in K}B(z,\,r_z)$.  By the
$5r$-covering theorem, 
there exists a collection of countably many pairwise disjoint disks
\[
\{B(z_i,r_i)\} \subset \{B(z,r) \,:\,z \in K, 0 < r < \delta) \} 
\]
such that 
\[
K \subset \bigcup_i B(z_i,5r_i). 
\]
Moreover since $\mathcal H^2(K)=0$, for any $\ez>0$, 
there exists $\delta_0>0$ such that 
\begin{equation}\label{Lebesgue measure}
 \sum_i r_i^2 < \ez
\end{equation}
whenever $0<\delta<\delta_0.$

Now by \eqref{eq:gradestim} we have
\[
 \int_{\bigcup_iB(z_i,r_i)} |\nabla E\psi(w)|\,dw \gs \sum_ir_i \gs \ch^1_\delta(K) \gtrsim \ch^1(K). 
\]
for $\delta>0$ small enough. However, since $|\nabla E\psi|\in L^{1}(\mathbb R^2)$, the absolute continuity of the integral and \eqref{Lebesgue measure}, 
give that
\[
\int_{\bigcup_iB(z_i,r_i)} |\nabla E\psi(w)|\,dw 
\]
can be made arbitrarily small by choosing a sufficiently small
$\ez$. Thus we have obtained a contradiction and hence the 
necessity of \eqref{thm:quasiconvex} for $W^{1,\,1}$-extension domains follows.

Let us end this section by giving a proof for Corollary~\ref{mainthm}.
\begin{proof}[Proof of Corollary~\ref{mainthm}]
Recall  that a Jordan domain is always bounded and simply connected.   Let us assume first that $\Omega$ is a $W^{1,\,1}$-extension domain. Then its complement $\mathbb R^2\setminus \Omega$ is quasiconvex by Theorem~\ref{thm:main}, and then Lemma~\ref{qconvex} yields that the complementary domain $\wz\Omega:=\mathbb R^2\setminus \overline{\Omega}$ is also quasiconvex. 

Now suppose that $\wz\Omega$ is quasiconvex. Then Lemma~\ref{qconvex} tells that $\mathbb R^2\setminus \Omega$ is also quasiconvex. Thus, any two points $x,\,y\in \mathbb R^2\setminus \Omega$ can be joined by a union $\gamma$ of hyperbolic geodesics via at most finitely many intermediate points, as established in the proof of Lemma~\ref{qconvex}. This curve $\gamma$ touches the boundary $\partial \Omega$ at most finitely many points, and thus \eqref{eq:curvecondition} follows. Then applying Theorem~\ref{thm:main}, we conclude that $\Omega$ is a $W^{1,\,1}$-extension domain. 
\end{proof}

\appendix

\section{Auxiliary Lemmas}

In this appendix we provide the proofs of Theorem~\ref{GM unbounded} and Lemma~\ref{ulkodist}. The proofs given here use  capacity estimates. 

\subsection{Conformal capacity}

Let $\Omega$ be a  domain. For a given pair of continua $E,\,F \subset \overline \Omega\subset \mathbb R^2$, 
define the {\it conformal capacity between $E$ and $F$ in $\Omega$} as
 $${\rm Cap}(E,\,F,\,\Omega)=\inf\{\|\nabla u\|^2_{L^{2} (\Omega)}: \ u\in\Delta(E,\,F,\,\Omega )\},$$
 where $\Delta(E,\,F,\,\Omega )$ denotes the class of all $u\in W^{1,\,2}_{\loc}(\Omega)$ that are continuous
in $\Omega\cup E\cup F$ and satisfy $u=1$ on $E$, and $u=0$ on $F$. 
By definition, we see that the conformal capacity is increasing with respect to
$\Omega.$  

In what follows, we introduce the properties of conformal capacity which are used in this paper; see e.g.\ \cite[Chapter 1]{V1971} for additional properties. We remark that, even though \cite{V1971} (with some other references below) states the results for ``modulus'',   ``modulus'' is equivalent with
conformal capacity in our setting below (see e.g. \cite[ Proposition 10.2, Page 54]{R1993}).

\begin{lem}
The conformal capacity is conformally invariant, that is, for  domains $\Omega$ and $\Omega'$ in $\mathbb R^2$, a conformal (onto) map $\varphi \colon \Omega \to \Omega'$ and continua $E$ and $F$ in $\Omega$, we have
\begin{equation}\label{cap inv}
{\rm Cap}(\varphi(E),\,\varphi(F),\,\Omega') = {\rm Cap}(E,\,F,\,\Omega).
\end{equation}
Moreover, if $\varphi$ has a homeomorphic extension, still denoted by  $\varphi$, $\varphi\colon \overline{\Omega}\to \overline{\Omega'}$, then \eqref{cap inv} also holds for continua in $\overline{\Omega}$. Especially this is the case if both $\Omega$ and $\Omega'$ are Jordan. 
\end{lem}

\begin{proof}
Let $0<\ez<\frac{1}{3}$ and fix a function  $u\in W^{1,\,2}(\Omega)$ that is continuous in $\Omega\cup E\cup F$ and satisfies $u = 1$ on $E$ and $u=0$ on $F$. Let now $u_\varepsilon$ be the function
\[
u'_{\ez} = \frac {{u}-\ez}{1-2\ez}
\]
in $\Omega$. Then the sets $U_{\ez}=\left\{ u_{\ez}>1\right\}$ and $V_{\ez}=\left\{ u_{\ez}< 0\right\}$ are relatively open in $\Omega\cup E \cup F$, and contain $E$ and $F$, respectively. 
Define 
$$u_{\ez}=\max\{\min\{u'_{\ez},\,1\},\,0\}. $$
Observe that $u_{\ez}\to u$ in $W^{1,\,2}(\Omega)$ as $\ez\to 0$.  

For each $\epsilon>0$, $u_\epsilon \circ \varphi^{-1}$ is an admissible function in $\Delta(\varphi(E),\varphi(F),\,\Omega')$. Thus, by the chain rule, conformality of $\varphi^{-1}$ and a change of variables, we have the estimate
\begin{align*}
{\rm Cap}(\varphi(E),\,\varphi(F),\,\Omega')&\le    \int_{\Omega'} |\nabla u(\varphi^{-1}(x))|^2 |D\varphi^{-1}(x)|^2 \, dx\\
&\le  \int_{\Omega} |\nabla u(x)|^2 |D\varphi^{-1}(\varphi(x))|^2 J_{\varphi}(x) \, dx \le \int_\Omega |\nabla u|^2 \, dx.
\end{align*}
The claim now follows by taking an infimum over functions $u$ above and symmetry.
\end{proof}

In what follows, whenever we mention the conformal invariance of conformal capacity, we always refer to this lemma above.

In the unit disk $\mathbb D$, and the complementary domain of the unit disk $\mathbb R^2\setminus \overline{\mathbb D}$, we have the following \emph{Loewner estimate} for the conformal capacity. Let $E$ and $F$ be continua in $\overline{\mathbb D}$. Then 
\begin{equation}
\label{condition of lower bound}
{\rm Cap}(E,\,F,\,\mathbb D) \ge c \log \left(1+\frac{\min\{\diam (E),\,\diam (F)\}}{\dist(E,\,F)}\right)
\end{equation}
where $c>0$ is a universal constant, and the analogous inequality holds for  $E,\,F\subset \mathbb R^2 \setminus {\mathbb D}$; see e.g.\ \cite[Lemma 7.38]{V1988}  together with \cite[Remark 2.12]{GM1985 2}.

We call a domain $A\subset\mathbb R^2$ a {\it ring domain} if its complement has exactly two components. If the components of $A$ are $U_0$ and $U_1$, then we denote $A=R(U_0,\,U_1)$. 
It follows from topology that $\partial A$ also has two components, $V_0=U_0\cap \overline{A}$ and $V_1=U_1\cap \overline{A}$.  If $U_0$, $V_0$ and $V_1$ are compact, we have
\begin{equation}\label{same capacity}
  {\rm Cap}( V_0,\, V_1 ,\,A)= {\rm Cap}( U_0,\, V_1 ,\, A\cup U_0 ); 
\end{equation}
indeed, ``$\le$'' directly follows from the definition and ``$\ge$'' follows by extending each $u\in \Delta(V_0,\,V_1,\,A)$ as constant $1$ to $U_0\setminus V_0$, see also \cite[Theorem 11.3]{V1971} (and its proof). 
Furthermore, we have the following estimate for the capacity of boundary components in a ring domain. 

\begin{lem}
Let $A=R(U_0,\,U_1)\subset\mathbb R^2$ be a ring domain with $U_1$ unbounded. Assume that $V_0=U_0\cap \overline{A}$ and $V_1=U_1\cap \overline{A}$ are compact.  There exist two universal increasing functions $\phi_i \colon (0,\infty) \to (0,\infty),\, i=1,\,2,$ so that $\lim_{t\to 0+} \phi_i(t) = 0$ and $\lim_{t\to \infty}\phi_i(t) = \infty$, and so that  
\begin{equation}\label{capacity of balls}
\phi_1\left(\frac{ \diam (U_0) }{\dist(U_0,\, U_1 )}\right) \le {\rm Cap}( V_0,\, V_1 ,\,A)\le \phi_2\left(\frac{ \diam (U_0) \}}{\dist(U_0,\, U_1)}\right). 
\end{equation}
\end{lem}

\begin{proof}
Define
$$\sigma=\frac{ \diam (U_0) }{\dist(U_0,\, U_1 )}.$$
The lower bound of \eqref{capacity of balls} follows from \cite[Theorem 11.7, Theorem 11.9]{V1971}. For the upper bound, if $0<\sigma<\frac 1 2$ we have $U_0 \subset B,\, 2B\subset A\cup U_0$ for a suitable disk $B$, and then by the monotonicity of capacity and \cite[Example 7.5]{V1971} one obtains the upper bound $C(\log (1/\sigma))^{-1}$ for some absolute constant $C>0$. When $\frac 1 2 \le \sigma<\infty$
 one just applies the test function
$$u(x)=\min\left\{1,\,\max\left\{0,\,1-\frac{\dist(x,\,U_0)}{\dist(U_0,\, U_1)}\right\}\right\}.$$
Then $u$ is Lipschitz,  and 
$$|\nabla u|\le \frac 1 {\dist(U_0,\, U_1)}. $$
Hence  by letting $x_0\in U_0$
\begin{align*}
\|\nabla u\|^2_{L^2(\Omega)}&\le {\dist(U_0,\, U_1)}^{-2} \left|B(x_0,\, \diam(U_0)+\dist(U_0,\,U_1))\right|\\ &\lesssim \left(\frac{\diam(U_0)+\dist(U_0,\,U_1)}{ \dist(U_0,\,  U_1)}\right)^2\sim(1+\sigma)^2.
\end{align*}
\end{proof}

We record the following estimate, which states a kind of converse to  \eqref{condition of lower bound}; see e.g.\ \cite[Lemma 2.2]{KZ2015} for a weaker version of the claim.

\begin{lem}{\label{inner capacity}}
Let $\Omega$ be a domain and  $E,\,F\subset \Omega$ be a pair of disjoint continua. 
Then if ${\rm Cap}(E,\,F,\, \Omega)\ge c_0$, we have
\begin{equation}\label{inequat 1}
{\min\{\diam_{\Omega}(E),\,\diam_{\Omega}(F)\}}\gtrsim {\dist_{\Omega}(E,\,F)},
\end{equation}
where the constant only depends on $c_0$. Especially
$${\min\{\diam_{\Omega}(E),\,\diam_{\Omega}(F)\}}\gtrsim {\dist (E,\,F)}, $$
and if $\Omega=\mathbb R^2$
\begin{equation}\label{inequat 10}
\min\{\diam(E),\,\diam(F)\}\gtrsim {\dist (E,\,F)}.  
\end{equation}
If we further assume that $\Omega$ is Jordan, then \eqref{inequat 1} also holds if $E\subset \partial\Omega$ and $F\subset\Omega$ (or $E\subset \partial\Omega$ and $F\subset\partial\Omega$) are continua with ${\rm Cap}(E,\,F,\, \Omega)\ge c_0$. 
\end{lem}
\begin{proof}
\noindent{\bf Step 1:}
We begin with the case where $E,\,F\subset \Omega$. 
We may clearly assume that $\diam_{\Omega}(E)\le \diam_{\Omega}(F)$ and also that
$2\diam_{\Omega}(E) \le \dist_{\Omega}(E,\,F)$; otherwise the claim holds trivially. 
Choose $z\in E$, and write $\frac{\dist_{\Omega}(E,\,F)}{\diam_{\Omega}(E)}=\sigma$. 
We define
$$ u(x)=
 \begin{cases}
   1, & \text{if } \dist_{\Omega}(x,\,z)\le \diam_{\Omega}(E) \\
  0, & \text{if}  \dist_{\Omega}(x,\,z)\ge \dist_{\Omega}(E,\,F)\\
      \frac {\log (\dist_{\Omega}(E,\,F))-\log (\dist_{\Omega}(x,\,z))}{\log (\dist_{\Omega}(E,\,F))-\log (\diam_{\Omega}(E))}, &\text{otherwise }
 \end{cases}.$$

Then $u$ is locally Lipschitz and 
$$|\nabla u(x)|\le (\log \sigma)^{-1} \dist_{\Omega}(x,\,z)^{-1}.$$
Write
$$R=B_{\Omega}(z,\,\dist_{\Omega}(E,\,F))\setminus B_{\Omega}(z,\,\diam_{\Omega}(E)),$$
and for $i\ge 1$
$$A_i=B_{\Omega}(z,\,2^{i}\diam(E))\setminus B_{\Omega}(z,\,2^{i-1}\diam(E)),$$
where $B_{\Omega}(z,\,r)$ is the disk centered at $z$ with radius $r$ with respect to the inner distance. 
A direct calculation via a dyadic annular decomposition with respect
to the inner distance gives
\begin{align*}
c_0\le & \int_{\Omega} |\nabla u|^2\, dx  
\ls   (\log \sigma)^{-2}\int_{R} \dist_{\Omega}(x,\,z)^{-2} \, dx\\ 
\ls &   (\log \sigma)^{-2} \sum_{i=1}^{\infty}\int_{R\cap A_i } 4^{-i}\diam(E)^{-2} \, dx\\
\ls &   (\log \sigma)^{-2} \sum_{i=1}^{[\log \sigma]+1} 1\\
\ls &  (\log \sigma)^{-2} \log \sigma \lesssim  {(\log \sigma)}^{-1}, 
\end{align*}
where  $[\log \sigma]$ denotes the integer part of $[\log \sigma]$, and in the fourth inequality we used the fact that $B_{\Omega}(z,\,r)\subset B(z,\,r)$. 
Hence $\sigma\le C(c_0)$, which means that $\dist_{\Omega}(E,\,F) \lesssim \diam_{\Omega}(E)$.

\noindent{\bf Step 2:} We continue with the case $E,\,F\subset\partial\Omega$. 
To begin with, let us first show that 
$$\dist_{\Omega}(E,\,F)<\infty. $$
Let $\varphi\colon \overline{\mathbb D} \to \overline{\Omega}$ be a homeomorphism given by the Caratheodory-Osgood theorem. 
Since 
$${\rm Cap}(E,\,F,\, \mathbb R^2)\ge {\rm Cap}(E,\,F,\, \Omega)\ge c_0,$$
we conclude by \eqref{inequat 10} that neither $E$ nor $F$ is a singleton. 
Then $\varphi^{-1}(E)$ is an closed arc contained in the unit circle and of positive $1$-Hausdorff measure. Since $\varphi\in W^{1,\,2}(\mathbb D)$, then
$$\dist_{\Omega}(E,\,\varphi(0))<\infty;$$
otherwise by noticing that $\varphi$ is a smooth map on $\overline{B(0,\,1/2)}$ and applying Fubini's theorem  in the annulus $\mathbb D \setminus \overline{B(0,\,1/2)}$ with respect to the polar coordinates, we obtain a contradiction to the fact that $\varphi\in W^{1,\,2}(\mathbb D)$. 
An analogous argument also shows that 
$$\dist_{\Omega}(F,\,\varphi(0))<\infty,$$
and hence by the triangle inequality $\dist_{\Omega}(E,\,F)<\infty. $

Let $\phi\colon \mathbb H\to \Omega$ be a conformal map (homeomorphically extended  to  $\mathbb R$), where $\mathbb H$ denotes the upper half plane so that $\phi^{-1}(E)$ and $\phi^{-1}(F)$ are two continua in the real line. 
We define a sequence of continua $E_j \subset  \phi^{-1}(E)$ (for large $j$) by setting
$$ E_j =\{ z \in \mathbb R \colon   \dist(z,\,\mathbb R \setminus \phi^{-1}(E))\ge 2^{-j+1}\}. $$
The sets $ F_j\subset \phi^{-1}(F)$ are defined similarly. 
Furthermore, define
$$\hat E_j= E_j +i  2^{-j}, \ \hat F_j= F_j +i 2^{-j}. $$

Note that for every two points $z_1,\,z_2\in E$, the hyperbolic geodesic $\Gamma'$ joining them satisfies (see Lemma~\ref{GM bdd})
$$\ell(\Gamma')\lesssim \dist_{\Omega}(z_1,\,z_2), $$
where the constant is absolute. 
Given $w_1,\,w_2\in \phi(\hat E_j)$, denote the hyperbolic geodesic connecting them by $\Gamma$. Extend $\Gamma$ to a full hyperbolic geodesic $\Gamma_1$ joining points $z_1,\,z_2\in \partial \Omega$. Since $w_1,\,w_2\in \phi(\hat E_j)$, from the definition of $\hat E_j$ and planar geometry  we conclude via $\phi$ that $z_1,\,z_2\in  E$. Therefore
$$\dist_{\Omega}(w_1,\,w_2)\le \ell(\Gamma)\le \ell(\Gamma_1)\lesssim \dist_{\Omega}(z_1,\,z_2)\le \diam_{\Omega}(E). $$
By the arbitrariness of $w_1,\,w_2$, we conclude that
\begin{equation}\label{inequ 19}
\diam_{\Omega}(\phi(\hat E_j)) \lesssim  \diam_{\Omega}(E) 
\end{equation}
with an absolute constant. Similarly we get  
\begin{equation}\label{inequ 20}
\diam_{\Omega}(\phi(\hat F_j)) \lesssim  \diam_{\Omega}(F). 
\end{equation}

By the monotonicity of capacity, we have
$${\rm Cap}(\hat E_j,\,\hat F_j,\, \mathbb H)\ge {\rm Cap}(\hat E_j,\,\hat F_j,\, \mathbb H + i 2^{-j}  ). $$ 
Via the translation map $z\mapsto  z - i2^{-j}$ and the conformal invariance of capacity, we further have
$${\rm Cap}(\hat E_j,\, \hat F_j,\, \mathbb H + i 2^{-j} )={\rm Cap}(E_j,\, F_j,\,  \mathbb H). $$
 If we can show that
\begin{equation}\label{inequ 21}
{\rm Cap}(E_j,\, F_j,\,  \mathbb H)\ge \frac 1 {16} {\rm Cap}(\phi^{-1}(E) ,\, \phi^{-1}(F) ,\,  \mathbb H)
\end{equation}
for $j$ large enough, then by the conformal invariance of capacity together with the chain of inequalities above
$$ {\rm Cap}(\phi(\hat E_j),\,\phi(\hat F_j),\,\Omega )\ge \frac 1 {16} {\rm Cap}(\phi^{-1}(E) ,\, \phi^{-1}(F) ,\,  \mathbb H)= \frac 1 {16} {\rm Cap}( E ,\,  F  ,\,  \Omega)\ge \frac {c_0} {16}.$$
Thus Lemma~\ref{inner capacity} implies that
$${\min\{\diam_{\Omega}(\phi(\hat E_j)),\,\diam_{\Omega}(\phi(\hat F_j))\}}\gtrsim {\dist_{\Omega}(\phi(\hat E_j),\,\phi(\hat F_j))}, $$
where the constant only depends on $c_0$.
This together with \eqref{inequ 19} and  \eqref{inequ 20} gives that there is a rectifiable curve 
$\gamma_j$ connecting $\phi(\hat E_j)$ and $\phi(\hat F_j)$ such that
$$\ell(\gamma_j)\lesssim {\min\{\diam_{\Omega}(\phi(\hat E_j)),\,\diam_{\Omega}(\phi(\hat F_j))\}}\ls \min\{\diam_{\Omega}(E),\,\diam_{\Omega}(F)\},$$
with the constant independent of $j$. 
We may assume that these curves are hyperbolic geodesics by Lemma~\ref{GM bdd}.
Then by parameterizing $\gamma_j$ with arc length and applying the Arzel\'a-Ascoli lemma, we get a  curve $\gamma\subset \overline{\Omega}$ joining $E,\,F$ with the length bounded by $\min\{\diam_{\Omega}(E),\,\diam_{\Omega}(F)\}$ up to a multiplicative constant. Observe that, since $\varphi$ is a homeomorphism, the uniform convergence of $\gamma_j\to \gamma$ implies that of $\varphi^{-1}(\gamma_j)\to \varphi^{-1}(\gamma)$. Then the definition of hyperbolic geodesics in the unit disk gives that $\varphi^{-1}(\gamma)$ is also a hyperbolic geodesic, and so is  $\gamma$. 
The desired estimate \eqref{inequat 1} then follows. 
Thus it suffices to show \eqref{inequ 21}.

 Set 
$$a=\min\{(\diam(\phi^{-1}(E)),\,\diam(\phi^{-1}(F)),\,\dist(\phi^{-1}(E),\,\phi^{-1}(F))\}.$$ 
Let the middle point of $\phi^{-1}(E)$ be $z_1$ and that of $\phi^{-1}(F)$ be $z_2$. 
Write 
\begin{equation}\label{inequat 2}
b_E=\frac 1 2 \diam(\phi^{-1}(E)) +\frac a 8,\ b_F=\frac 1 2 \diam(\phi^{-1}(F)) +\frac a 8.
\end{equation}
Recall that
$$\diam(\phi^{-1}(E))-\diam(E_j) = 2^{-j+2},\ \diam(\phi^{-1}(F))-\diam(F_j) = 2^{-j+2}. $$
For $j\in \mathbb N,\, 2^{-j+5}\le a$, we define a map $f$ by setting
$$ f(x)=
 \begin{cases}
   z_1+\frac{\diam(\phi^{-1}(E))}{\diam(E_j)}(x-z_1), & \,  |x-z_1|\le \frac 1 2 \diam(E_j) \\
  z_1+\left(\frac 1 2 \diam(\phi^{-1}(E)) + \frac{ a \left(|x-z_1|-\frac 1 2 \diam(E_j)\right)}{ 2^{-j+4}+ a}  \right)\frac {(x-z_1)}{|x-z_1|}, &   \, \frac 1 2 \diam(E_j)\le |x-z_1|\le b_E \\
   z_2+\frac{\diam(\phi^{-1}(F))}{\diam(F_j)}(x-z_2), & \,  |x-z_2|\le \frac 1 2 \diam(F_j) \\
   z_2+\left(\frac 1 2 \diam(\phi^{-1}(F)) + \frac{ a \left(|x-z_2|-\frac 1 2 \diam(F_j)\right) }{ 2^{-j+4}+ a}\right) \frac {(x-z_2)}{|x-z_2|}, &   \, \frac 1 2 \diam(F_j)\le |x-z_2|\le b_F \\
    x , & \text{ otherwise.}
 \end{cases}$$
 Then $f$ is a well-defined homeomorphism; recall \eqref{inequat 2} and notice that
 $$b_E-\frac 1 2 \diam(E_j)= 2^{-j+1}+\frac a 8, \ \ b_F-\frac 1 2 \diam(F_j)= 2^{-j+1}+\frac a 8.$$
By our choice of $j$, $f$  maps $\mathbb H$ to $\mathbb H$, $E_j$ to $\phi^{-1}(E)$ and $F_j$ to $\phi^{-1}(F)$. Moreover since $2^{-j+5}\le a$, 
$$ \frac 1 2 \le \frac { a }{ 2^{-j+4}+a}\le 1 ,\ 1\le \frac{\diam(\phi^{-1}(E))}{\diam(E_j)}\le 2,\  1\le \frac{\diam(\phi^{-1}(F))}{\diam(F_j)} \le 2.$$
Thus $f$ is $2$-biLipschitz. For any $u\in \Delta(E_j ,\, F_j,\,\mathbb H)$ we have $u\circ f^{-1}\in \Delta( \phi^{-1}(E) ,\,  \phi^{-1}(F),\,\mathbb H)$ since $f$ is a homeomorphism. Thus by the chain rule and  a change of variables we get
\begin{align*}
{\rm Cap}(\phi^{-1}(E) ,\, \phi^{-1}(F) ,\,  \mathbb H)& \le  \int_{\mathbb H} |\nabla u(f^{-1}(x))|^2 |Df^{-1}(x)|^2 \, dx\\
&\le 4 \int_{\mathbb H} |\nabla u(x)|^2  J_{f}(x) \, dx \le 16 \int_{\mathbb H} |\nabla u|^2 \, dx,
\end{align*}
which implies \eqref{inequ 21} as desired, and then \eqref{inequat 1} follows. 

\noindent{\bf Step 3:} We are left with the case  where $E\subset \partial\Omega$ and $F\subset\Omega$. In this case we  only perform the above approximation for $E$ (and consider $j$ large enough so that $E_j\cap \phi^{-1}(F)=\emptyset$). The argument in the previous case applies with this modification and gives the remaining claim. 
\end{proof}

\subsection{A variant of the Gehring-Hayman theorem}

We employ following lemma on conformal annuli to prove Theorem~\ref{GM unbounded}.

\begin{lem}\label{lengthtransfer}
Let  $\Omega\subset \mathbb R^2$ be a Jordan domain, and let a homeomorphism 
$\wz \varphi\colon \mathbb R^2\setminus \Omega \to 
 \mathbb R^2\setminus \mathbb D$ be conformal in $\mathbb R^2 \setminus
\overline{\Omega}$. 
For $z_1
\in \partial \Omega$, define
\[
A(z_1,\,k):=\{x\in \mathbb R^2\setminus \overline{\mathbb D} 
\mid 2^{k-1}< |x-\wz \varphi (z_1)|\le 2^{k}\},  
\]
for 
$k\in \mathbb Z$. 
Furthermore, let $\Gamma\subset \mathbb R^2\setminus \overline{\Omega}$ 
be a hyperbolic geodesic joining $z_1$  and $z_2\in \partial \Omega$ such that $\wz \varphi (\Gamma)\subset B(0,\,100)$. Also let
$\gamma\subset \mathbb R^2 \setminus \Omega$ 
be
a curve connecting $z_1$ and $z_2$. 
Set
\[
\Gamma_{k}:=\wz \varphi^{-1}(A(z_1,\,k))\cap \Gamma 
\]
and let $\gamma_{k}$ 
be any subcurve of $\gamma$ in $\wz \varphi^{-1}(A(z_1,\,k))$
joining the inner and outer boundaries of
$\wz \varphi^{-1}(A(z_1,\,k)),$ if such a subcurve exists. 
(Here the inner and outer boundaries of $\wz \varphi^{-1}(A(z_1,\,k))$ are the preimages under $\wz \varphi$ of the inner and outer boundaries of $A(z_1,\,k)$. )
Then 
\begin{equation}\label{equ 1002}
\ell(\Gamma_{k})\sim \dist(\Gamma_{k},\,\partial 
\Omega)
\end{equation}
 and 
$$\ell(\gamma_{k})\gtrsim 
\ell(\Gamma_{k})\sim \diam(\Gamma_{k}).$$
Here all the constants are independent of $\Omega$ and the choice
of $\wz \varphi,z_1,\gamma,z_2,k.$ 
\end{lem}
\begin{proof}
We may assume that $\Gamma_{k}$ is non-empty; otherwise the claim of the lemma is trivial.

The fact that 
$\ell(\Gamma_k)\sim \dist(\Gamma_k,\,\partial \Omega)\sim \diam(\Gamma_k)$ 
follows immediately from  Lemma~\ref{whitney preserving} and a change of variable, since by definition 
$\wz \varphi(\Gamma_k)$ is contained in a Whitney-type set 
in $\mathbb R^2\setminus \overline{\mathbb D}$. 

Hence we only need to prove that $\ell(\gamma_k)\gtrsim \ell(\Gamma_k)$. 
Observe that, since $\gamma_k$ by definition joins the inner and outer 
boundaries of
$\wz \varphi^{-1}(A(z_1,\,k)),$ then $\wz \varphi(\gamma_k)$ joins the inner and outer 
boundaries of
$A(z_1,\,k),$ and hence
\begin{equation}\label{eqn1}
\ell(\wz \varphi (\gamma_k))\gtrsim \diam(\wz \varphi (\Gamma_k))\sim \dist(\wz \varphi (\Gamma_k),\,\partial \mathbb D).
\end{equation}
We next argue by case study. 

\noindent{\it Case 1: $\dist(\wz \varphi (\gamma_k),\,\wz \varphi (\Gamma_k))\ge \frac 1 3 \dist(\wz \varphi (\Gamma_k),\,\partial \mathbb D)$.}  By Lemma~\ref{whitney preserving}, the assumption and the fact that $\wz \varphi (\Gamma_k)$ is contained in a Whitney-type set, we know that for any curve $\gamma'$ joining $\gamma_k$ and $\Gamma_k$, its length satisfies
$$\ell(\gamma')\gtrsim \diam(\Gamma_k),$$
and hence
\begin{equation}\label{eqn00}
\dist_{\Omega}( \gamma_k ,\, \Gamma_k )\gtrsim \diam(\Gamma_k). 
\end{equation}

Moreover by \eqref{condition of lower bound} for the exterior of the unit disk, \eqref{eqn1} and the monotonicity of the capacity we obtain
$$1\ls {\rm Cap}(\wz \varphi (\overline{\gamma}_k),\,\wz \varphi (\overline{\Gamma}_k),\,\mathbb R^2\setminus \overline{\mathbb D})={\rm Cap}( \overline{\gamma}_k ,\, \overline{\Gamma}_k ,\,\mathbb R^2 \setminus \overline{\Omega})\le {\rm Cap}( \overline{\gamma}_k ,\, \overline{\Gamma}_k ,\,\mathbb R^2). $$
Hence by \eqref{eqn00} and Lemma~\ref{inner capacity} we know that
$$\ell(\gamma_k)\ge \diam(\gamma_k)\gtrsim \diam(\Gamma_k)\sim \ell(\Gamma_k).$$

\noindent{\it Case 2: $\dist(\wz \varphi (\gamma_k),\,\partial \mathbb D)\ge \frac 1 3 \dist(\wz \varphi (\Gamma_k),\,\partial \mathbb D)$.} This assumption implies that  $\wz \varphi (\gamma_k)\cup \wz \varphi (\Gamma_k)$ is contained in a Whitney-type set. Then $\gamma_k\cup \Gamma_k$ is also contained in a Whitney-type set by Lemma~\ref{whitney preserving}, and then the desired estimate follows directly from Lemma~\ref{whitney preserving} and \eqref{eqn1} via a change of variable.

\noindent{\it Case 3: $$\dist(\wz \varphi (\gamma_k),\,\wz \varphi (\Gamma_k))< \frac 1 3 \dist(\wz \varphi (\Gamma_k),\,\partial \mathbb D)$$ and $$\dist(\wz \varphi (\gamma_k),\,\partial \mathbb D)<\frac 1 3 \dist(\wz \varphi (\Gamma_k),\,\partial \mathbb D).$$}
In this case,  there is a subcurve $\wz \gamma_k\subset \gamma_k$ such that $\ell(\wz \varphi (\wz \gamma_k))\gtrsim \ell(\wz \varphi (\Gamma_k))$ and $\dist(\wz \varphi (\wz \gamma_k),\,\partial \mathbb D)\gtrsim \dist(\wz \varphi (\Gamma_k),\,\partial \mathbb D)$, as $\gamma_k$ is a (connected) curve. Then we are reduced to a case similar to the second one, and it follows that
$$\ell(\gamma_k)\gtrsim \ell(\wz \gamma_k)\ge \diam(\wz \gamma_k) \gtrsim \diam(\Gamma_k)\sim \ell(\Gamma_k). $$
Consequently we obtain the desired estimate. 
\end{proof}

We call the set $\wz \varphi^{-1}(A(z,\,k))$  defined above a {\it conformal annulus} centered at $z$. Now let us prove Theorem~\ref{GM unbounded}. 

\begin{proof}[Proof of Theorem~\ref{GM unbounded}]
We first deal with the case where $\wz \varphi(x),\,\wz \varphi(y)$ are not on the same radial ray. 
Extend the hyperbolic geodesic joining $x$ and $y$ to the boundary of $\wz \Omega$ at $z_1$ and $z_2$, and construct the conformal annuli as in Lemma~\ref{lengthtransfer} 

We first consider the case where $x,\,y$ are in the same conformal annulus (centered at $z_1$ or $z_2$). Then  since $\wz \varphi(x),\,\wz \varphi(y)$ are contained in some Whitney-type set of the exterior of the unit disk by the planar geometry, via  from Lemma~\ref{whitney preserving} we know that $x,\,y$ are in some Whitney-type set of $\wz \Omega$, and the claim follows directly.

Then we show the claim in the case where $x$ and $y$ are in different annuli. We may assume that they lie on the boundary of different conformal annuli, since the general case follows from the triangle inequality with the conclusion of the first case. 

We employ the notation in Lemma~\ref{lengthtransfer}. 
If $|\wz \varphi (x)-\wz \varphi (y)|\le 1$, then by applying the Jordan curve theorem to the union of the subarc of $\partial \mathbb D$ inside
$B(\wz \varphi(z),\,2^{k-1})$
 with the inner boundary of the conformal annulus $A(z,\,k)$, with the help of the conformal map we conclude that, any curve joining $x$ and $y$ intersects the inner boundary of $\varphi(A(z,\,k))$ whenever $\Gamma_k$ is non-empty. Similarly via the Jordan curve theorem we obtain that any curve joining $x$ and $y$  intersects the outer boundary of $\varphi(A(z,\,k))$ whenever $\Gamma_k$ is non-empty.  
In other words,  any curve joining $x$ and $y$ should go across the conformal annulus $\varphi(A(z,\,k))$ whenever $\Gamma_k$ is non-empty. 
Then by Lemma~\ref{lengthtransfer} we have
$$\ell(\Gamma)\le \sum_{k} \ell(\Gamma_k) \ls \sum_k \ell(\gamma_k) \ls \ell(\gamma) $$
for some absolute constants. 
Therefore we obtain the desired conclusion.

Suppose $|\wz \varphi (x)-\wz \varphi (y)|> 1$. If $\Gamma_k\neq \emptyset$   but $\gamma_k=\emptyset$, then the Jordan curve theorem (with the argument above) implies that $k\ge 1$. By $|\wz \varphi (x)-\wz \varphi (y)|> 1$ and  the fact that $\wz  \varphi (\Gamma)\subset B(0,\,100)$, according to \eqref{condition of lower bound} we have
$${\rm Cap}(\Gamma_k,\,\gamma,\,\wz \Omega)={\rm Cap}(\wz \varphi (\Gamma_k),\,\wz \varphi(\gamma),\,\mathbb R^2\setminus \overline{\mathbb D})\ge C.$$
Thus by Lemma~\ref{inner capacity} and Lemma~\ref{lengthtransfer} we have
$$\diam(\Gamma_k) \le C \diam(\gamma)\le C \ell(\gamma)$$
for some absolute constants. 
Therefore by Lemma~\ref{lengthtransfer} again we conclude that
$$\ell(\Gamma)\le \sum_{k<0} \ell(\Gamma_k) +\sum_{0\le k\le 7}  \ell(\Gamma_k) \ls \sum_{k<0} \ell(\gamma_k) + \sum_{0\le k\le 7}  \ell(\Gamma_k)  \ls \ell(\gamma). $$
Hence we conclude the theorem in the case  where $\wz \varphi(x),\,\wz \varphi(y)$ are not on the same radial ray. 

If  $\wz \varphi(x),\,\wz \varphi(y)$ are on the same radial ray, we extend the hyperbolic geodesic joining $x$ and $y$ to the boundary of $\wz \Omega$ at $z_1$, and only consider the conformal annuli centered at $z_1$. Then the above argument also applies similarly to this case, and we conclude the theorem. 
\end{proof}


\subsection{A lemma on quasiconvex sets}\label{appendix quasiconvex}

\begin{proof}[Proof of Lemma~\ref{ulkodist}]
We may assume that $0\notin \wz \Omega$. 
Let $S=\partial B(0,\, 10 \diam(\partial \wz  \Omega))$, and $Q$ be a $4\sqrt 2$-Whitney square of $\wz\Omega$ intersecting $S$. Then $S\subset \wz\Omega$ and $\diam(Q)/\dist(Q,\,\partial \wz \Omega)$ is bounded from below by an absolute constant. We claim that
\begin{equation}\label{equ 2}
{\rm Cap}(Q,\, \partial \wz \Omega,\,\wz \Omega)\gs 1.
\end{equation}
To begin with, by \eqref{same capacity} we have
$${\rm Cap}(Q,\, \partial {\Omega} ,\, \wz \Omega) ={\rm Cap}(\partial Q,\, \partial {\Omega} ,\, \wz \Omega \setminus Q). $$
Denote  the center of $Q$ by $x_0$. 
Since the M\"obius transformation $\phi\colon z\mapsto \frac {\diam(Q)} {(z-x_0)}$ is   uniformly biLipschitz in
$\{z\in \mathbb R^2\colon 0< \dist(z,\,Q)\le 2\dist(Q,\,\partial \Omega)\}$ 
as $Q$ is a Whitney square, we have
$$\frac{\diam(\partial Q)}{\dist(Q,\,\partial \Omega)}\sim \frac{\phi(\diam(\partial Q))}{\dist(\phi(Q),\,\phi(\partial \Omega))}$$
and thus by \eqref{capacity of balls} together with our assumption 
$${\rm Cap}(\phi(\partial Q),\, \phi(\partial {\Omega}),\, \phi(\wz \Omega\setminus Q))\ge C(\sigma).$$
Then \eqref{equ 2} follows from the conformal invariance of capacity as $\phi$ is conformal in the ring domain $G\setminus Q$. 

By the conformal invariance of capacity, we conclude from \eqref{equ 2} that
\begin{equation*}
{\rm Cap}(\wz \varphi(Q),\, \partial \mathbb D,\, \mathbb R^2\setminus \overline{\mathbb  D})\gs 1.
\end{equation*}
Notice that by Lemma~\ref{whitney preserving}, $\wz \varphi (Q)$ is of $\lambda$-Whitney-type for some absolute constant $\lambda$. 
Hence  
$$\diam(\wz \varphi (Q))\sim_{\lambda} \dist( \partial \mathbb D,\,\wz \varphi (Q)) \ls \diam(\mathbb D)=2$$ 
according to Lemma~\ref{inner capacity}. By the triangle inequality and the arbitrariness of $Q$ we have 
\begin{equation}\label{equ 3}
 \dist(\wz \varphi (S),\,\partial \mathbb D)\le  C_2^{-1} \ \text{  and  } \ 1\le \diam(\wz \varphi (S))\le C_2^{-1}. 
\end{equation}
where $C_2<1$ is an absolute constant. 

Since $\wz \Omega$ is $C_1$-quasiconvex, then for the points $z_1,\,z_2$ in the statement of the lemma, 
 there exists a curve $\gamma\subset \wz \Omega$ joining $z_1$ and $z_2$ such that
$$\ell(\gamma)\le \delta_1 C_1   \diam(\partial \wz  \Omega). $$
By letting $\delta_1<C_1^{-1}\le 1$, with the definition of $z_1,\,z_2$, we have $\gamma\subset B(0,\,4 \diam(\partial \wz  \Omega))$ by the triangle inequality. Then by \eqref{capacity of balls} again we obtain
\begin{equation}\label{equ 1}
{\rm Cap}(\gamma,\,S,\,\wz \Omega)\le \phi_2(\delta_1 C_1/6).
\end{equation}
By conformal invariance we have
\begin{equation}\label{equ 4}
{\rm Cap}(\wz \varphi (\gamma) ,\,\wz \varphi (S),\, \mathbb R^2\setminus \overline{\mathbb  D})\le \phi_2(\delta_1 C_1/6). 
\end{equation}
Notice that  since $\gamma\subset B(0,\, 4 \diam(\partial \wz  \Omega))$, then 
$\wz \varphi (\gamma)$ is contained in the Jordan domain enclosed by $\wz \varphi (S)$. 
Hence if 
$$|\wz \varphi (z_1)-\wz \varphi (z_2)|\ge \frac 1 4, $$
by \eqref{equ 3}, \eqref{capacity of balls} and \eqref{equ 4} we further have
\begin{equation}\label{equ 5}
 \phi_1(C_3/4)\le {\rm Cap}(\wz \varphi (\gamma) ,\,\wz \varphi (S),\, \mathbb R^2\setminus \overline{\mathbb  D}) \le \phi_2(\delta_1 C_1/6). 
\end{equation}
Thus by choosing $\delta_1$ small enough such that $\delta_1< C_1^{-1}$ but \eqref{equ 5} fails, we obtain the first part of the lemma. 

Towards the second part, by symmetry we only need to prove the inequality for $z_1$. If $z_1\in \partial \wz \Omega$, then by the 
Caratheodory-Osgood 
theorem we know that $\wz \varphi(z_1)\in \partial \mathbb D$, and the desired claim follows. If $z_1$ is in the interior of $\wz \Omega$, let $\Gamma$ be the hyperbolic geodesic joining $z_1$ to $\infty$. Then since $\dist(z_i,\,\partial \wz  \Omega)\le \delta_1 \diam(\partial \wz  \Omega)$,   by \eqref{capacity of balls}  we have
\begin{equation}\label{equ 200}
{\rm Cap}( \partial \wz \Omega,\,\Gamma,\,\wz \Omega)\ge \phi_1(\delta_1^{-1})
\end{equation}
for some absolute constant. 
By the conformal invariance of capacity, we conclude from \eqref{equ 200} that
\begin{equation*}
{\rm Cap}( \partial \mathbb D,\,\wz \varphi(\Gamma),\, \mathbb R^2\setminus \overline{\mathbb  D})\ge \phi_1(\delta_1^{-1}).
\end{equation*}
Then the desired estimate follows from Lemma~\ref{inner capacity} by choosing $\delta_1$ suitably.
\end{proof}

\section*{Acknowledgement}
We thank the anonymous referee for the comments and suggestions to
the manuscript.

\end{document}